\documentclass[fleqn,11pt]{elsarticle}

\usepackage{hyperref}
\hypersetup{
	colorlinks=true,
}

\usepackage{amsmath,amssymb,amsthm,graphicx,epstopdf,float,multirow,booktabs,color,appendix,bm,subfigure,tikz,tabularx,multicol}
\usepackage[top=1.25in, bottom=1.0in, left=1.0in, right=1.0in]{geometry}

\biboptions{sort&compress} 

\numberwithin{equation}{section}
\newtheorem{thm}{Theorem}[section]

\newtheorem{rmk}{Remark}[section]
\newtheorem{lem}{Lemma}[section]

\newtheorem*{prf}{Proof}

\graphicspath{{FIGS/}}
\usepackage[overload]{empheq}

\usepackage{algorithm,tikz-3dplot}
\usepackage{algpseudocode}

\begin{document}

\begin{frontmatter}

	\title{Explicit and CPU/GPU parallel energy-preserving schemes for the Klein-Gordon-Schr\"odinger equations}
	\author{Xuelong Gu$^{a}$}
	\author{Yushun Wang$^{b}$}
	\author{Ziyu Wu$^{c}$}
	\author{Jiaquan Gao$^{c}$}
	\author{Wenjun Cai$^{b,*}$}

	\address[1]{Department of Mathematics, University of South Carolina, Columbia, SC, 29208, USA\\ \vspace{0cm}}

	\address[2]{Ministry of Education Key Laboratory for NSLSC\\ School of Mathematical Sciences,
		Nanjing Normal University, Nanjing, 210023, China  \\ \vspace{0cm} }

	\address[3]{School of Computer and Electronic Information, Nanjing Normal University, Nanjing, 210023, China  \\ \vspace{0cm} }

	\begin{abstract}

		A highly efficient energy-preserving scheme for univariate conservative or dissipative systems was recently proposed in [Comput. Methods Appl. Mech. Engrg. 425 (2024) 116938]. This scheme is based on a grid-point partitioned averaged vector field (AVF) method, allowing for pointwise decoupling and easy implementation of CPU parallel computing. In this article, we further extend this idea to multivariable coupled systems and propose a dual-partition AVF method that employs a dual partitioning strategy based on both variables and grid points. The resulting scheme is decoupled, energy-preserving, and exhibits greater flexibility. For the Klein-Gordon-Schr\"odinger equations, we apply the dual-partition AVF method and construct fully explicit energy-preserving schemes with pointwise decoupling, where the computational complexity per time step is $\mathcal{O}(N^d)$, with $d$ representing the problem dimension and $N$ representing the number of grid points in each direction. These schemes not only enable CPU parallelism but also support parallel computing on GPUs by adopting an update strategy based on a checkerboard grid pattern, significantly improving the efficiency of solving high-dimensional problems. Numerical experiments confirm the conservation properties and high efficiency of the proposed schemes.

	\end{abstract}


\end{frontmatter}

\begin{figure}[b]
	\small \baselineskip=10pt
	\rule[2mm]{1.8cm}{0.2mm} \par
	$^{*}$Corresponding author.\\
	E-mail address: caiwenjun@njnu.edu.cn (W. Cai).
\end{figure}

\section{Introduction}


In this paper, we numerically investigate the dynamics governed by the following nonlinear Klein-Gordon-Schrödinger (KGS) equations:
\begin{equation}\label{kgs}
	\left\lbrace\begin{aligned}
		 & i\partial_t\psi + \frac{\kappa_1}{2}\Delta \psi + \gamma u\psi = 0,                                 &  & \text{in}~~\Omega \times (0,T], \\
		 & \partial_{tt} u - \kappa_2 \Delta u + \mu^2 u - \gamma |\psi|^2 = 0,                                &  & \text{in}~~\Omega \times (0,T], \\
		 & \psi(\bm{x},0) = \psi_0(\bm{x}),~~u(\bm{x},0) = u_0(\bm{x}),~~\partial_t u(\bm{x},0) = v_0(\bm{x}), &  & \text{in}~~\Omega,
	\end{aligned}\right.
\end{equation}
subject to periodic boundary conditions. Here, $\Omega \subset \mathbb{R}^d$ is a bounded domain, $i$ denotes the imaginary unit ($i^2 = -1$), and $\kappa_1, \kappa_2, \mu, \gamma$ are real constants representing various physical parameters of the system. The KGS equations \eqref{kgs} describe the dynamics of interactions between a complex-valued scalar nucleon field $\psi$ and a real-valued scalar meson field $u$ via a Yukawa-type coupling, playing a significant role in quantum field theory. The existence, uniqueness, and stability properties of the KGS equations have been extensively studied in the literature. We ref to Refs. \cite{fukuda1975coupledI,fukuda1975yukawa,fukuda1978coupledII,fukuda1980coupledIII,guo1982global,guo1995asymptotic} and the references therein for comprehensive theoretical analysis of the KGS equations \eqref{kgs}.

%

Soliton solutions of the KGS equations are crucial for modeling localized wave packets that represent particles and quantum fields \cite{Makhankov1978,Hioe-2003}. The stability of these solitons is intrinsically tied to the conservation of total energy, expressed as:
\begin{equation}\label{ene-origin}
	\frac{d\mathcal{E}(\psi, u)}{dt} = 0, \quad \text{where} \quad \mathcal{E}(\psi, u) = \int_\Omega \left[ \frac{1}{2} \left(\kappa_1 |\nabla \psi|^2 + \kappa_2 |\nabla u|^2 + |u_t|^2 + \mu^2 u^2 \right) - \gamma |\psi|^2 u \right] d \bm{x}.
\end{equation}
Accordingly, the development of energy-preserving numerical algorithms is essential for accurately capturing the dynamics of these solitons in computational simulations. Numerical evidences demonstrate that such energy-preserving algorithms exhibit superior performance in long-term simulations and are effective in eliminating spurious discrete solutions \cite{vu-quoc1993,li1995,sharma2020}. A significant amount of foundational work has been devoted to constructing energy-preserving algorithms, among which the discrete gradient method \cite{itoh1988,gonzalez1996,mclachlan1999geometric} is the most widely used. This method includes the well-known average vector field (AVF) method \cite{avfJcp,avfJPA,avf6} and the discrete variational derivative (DVD) method \cite{furihata1999,furihata2018} for partial differential equations. Additionally, methods such as the Hamiltonian boundary value method (HBVM) \cite{HBVM1,luiji_book}, the energy-preserving collocation method \cite{hairer2010a}, the time finite element method \cite{tangTFE}, and the quadratic auxiliary variable (QAV) method \cite{gongQAV,qavSIAM} have been proposed to construct high-order energy-preserving schemes. Some of these existing energy-preserving algorithms have already been successfully applied to the KGS equations \cite{dgkgs_anm,dgkgs_amm,kgs_hbvms,fujsc}, resulting in a series of stable schemes that conserve the original energy. However, most of these schemes are implicit, which can be computationally expensive for high-dimensional problems.

To improve the computational efficiency of energy-preserving algorithms, several linearly implicit construction methods have been introduced. These methods include the multi-step DVD approach for polynomial energies \cite{matsuo2001, dahlby2011}; the invariant energy quadratization (IEQ) \cite{yang2017c, zhao2017} and scalar auxiliary variable (SAV) \cite{shen2018, shen2019a} methods, both based on energy quadratization frameworks; and several approaches involving the solution of nonlinear algebraic systems, such as the Lagrange multiplier method \cite{cheng2020}, the supplementary variable method (SVM) \cite{svm1, svm2}, and the relaxation method \cite{relaxationLi}. Several of these linearly implicit methods have also been applied to the KGS equations  \cite{fkgs_sav_mcs, kgs_sav_jsc, li2023, li2024}, demonstrating greater efficiency compared to fully implicit schemes.  In addition to the general linearly implicit schemes mentioned above, there are also specific linearly implicit schemes tailored for the KGS equations. For instance, in Refs. \cite{kgs_cifd, li2019a}, the authors developed a decoupled linearly implicit scheme for the KGS equations by applying the implicit midpoint method to the Schr\"odinger part and the leapfrog method to the Klein-Gordon part, with computations performed sequentially. As an improvement, in Refs. \cite{kgs_difd, yang2021}, a fully decoupled linearly implicit scheme was constructed by using the leapfrog method for both parts, allowing for concurrent computation.

Linearly implicit schemes require solving a system of linear equations at each time step, which can be further categorized into two types based on whether the coefficient matrix is constant. The first type, including methods like the DVD and IEQ, as well as the specific decoupled schemes for the KGS equations \cite{kgs_cifd, li2019a,kgs_difd, yang2021}, involves variable-coefficient systems that still require iterative algorithms. The second type, represented by the SAV method, involves constant-coefficient systems, allowing for efficient algorithms like the fast Fourier transform (FFT) with a computational complexity of $\mathcal{O}(N^d \log N)$, with $d$ representing the problem dimension and $N$ representing the number of grid points in each direction. In this paper, we propose a more efficient energy-preserving method for the KGS equations with a computational complexity only of $\mathcal{O}(N^d)$. The essence of this method is a synthesis of our previous work: a variable-based partitioned AVF method for multivariable coupled Hamiltonian systems \cite{pavf}, and a grid-point-based partitioned AVF method for univariate conservative or dissipative systems \cite{cai_cmame}. We therefore refer to this method as the dual-partition AVF method.  Compared to the approaches presented in Refs. \cite{pavf} and \cite{cai_cmame}, the proposed dual-partition AVF method not only retains their decoupling and conservative properties but also offers greater flexibility in the partitioning strategy. For the KGS equations, the resulting schemes allow for explicit calculations without the need to solve linear systems. Furthermore, by employing appropriate partitioning strategies, these schemes facilitate parallel computation on both CPUs and GPUs, significantly improving the computational efficiency of energy conservation algorithms for simulating high-dimensional problems.

The remainder of this paper is organized as follows: In Section 2, we present a semi-discrete energy-preserving scheme for the KGS equations based on the spatial central difference method. In Section 3, we propose a novel dual-partition AVF method that employs a dual partitioning strategy both by variables and grid points for general multivariable coupled Hamiltonian systems. This method is then applied to the KGS equations to obtain fully discrete energy-preserving schemes. Section 4 provides a detailed discussion on the explicit implementation of the proposed scheme, including specific partitioning strategies that enable parallel computation on both CPUs and GPUs. In Section 5, we present extensive numerical experiments to verify the accuracy, conservation properties, and efficiency of our method. Finally, Section 6 offers some concluding remarks.

\section{Energy-preserving spatial discretization}\label{sec:2}
In this section, we introduce an energy-preserving spatial discretization for the KGS equations \eqref{kgs}. To facilitate the application of our proposed AVF-type methods, we first reformulate the KGS equations into an equivalent first-order real-valued system. We then employ a spatial central difference discretization to construct a semi-discrete energy-preserving scheme, which is subsequently rewritten in the form of a standard Hamiltonian system. For simplicity and without loss of generality, we focus on the two-dimensional square domain $\Omega = (a, b)^2$, with the one-dimensional and three-dimensional cases following in a similar manner.

Let $p = {\rm Re} \psi$, $q = {\rm Im} \psi$ and $v = \partial_t u$, where ${\rm Re} \psi$, ${\rm Im} \psi$ represent real part and imaginary part of a complex-valued function $\psi$, respectively. Then, \eqref{kgs} can be recast equivalently as follows:
\begin{equation}\label{kgs-equiv}
	\left\lbrace
	\begin{aligned}
		 & \partial_t p = - \frac{\kappa_1}{2} \Delta q - \gamma u q,      \\
		 & \partial_t q =  \frac{\kappa_1}{2} \Delta p + \gamma u p,       \\
		 & \partial_t u = v,                                               \\
		 & \partial_t v = \kappa_2 \Delta u - \mu^2 u + \gamma(p^2 + q^2).
	\end{aligned}
	\right.
\end{equation}
Taking the inner product on both sides of each equation in \eqref{kgs-equiv} with $\partial_t q$, $\partial_t p$, $\partial_t v$, $\partial_t u$ respectively, and combining the resulting identities, it is readily to confirm that \eqref{kgs-equiv} preserves energy in the sense that
\begin{equation*}
	\begin{aligned}
		 & \frac{d\mathcal{E}(p, q, u, v)}{dt} = 0, \quad                                                                                                                                           \\
		 & \mathcal{E}(p, q, u, v) = \int_\Omega \left[\frac{1}{2} \left(\kappa_1 |\nabla p|^2 + \kappa_1 |\nabla q|^2 + \kappa_2 |\nabla u|^2 + |v|^2 + \mu^2 u^2\right) - \gamma (p^2 + q^2) u \right] d\bm{x}.
	\end{aligned}
\end{equation*}

To present the discretization, we introduce some preliminary notations related to the finite difference method. Let $N$ be a positive integer, and partition the spatial domain uniformly with a step size $h = (b-a)/N$. We define the grid points under periodic boundary condition as $\Omega_h = \{ (x_j, y_k)| x_j = a + (j-1) h, \ y_k = a + (k-1) h, \ 1 \leq j,k \leq N \}$, and the set of periodic grid functions as $\mathbb{V}_h = \{ U| U_{j,k} = U_{j+lN, k+mN}, \ l, m \in \mathbb{Z}, \ 1\leq j,k \leq N   \}$. The discrete inner product and norm on $\mathbb{V}_h$ are
\begin{equation*}
	(U, V)_h = h^2 \sum\limits_{j = 1}^{N} \sum\limits_{k=1}^{N} U_{j,k} \overline{V}_{j,k}, \quad  \|U\|_h = \sqrt{(U, U)_h}.
\end{equation*}
Given $U \in \mathbb{V}_h$, the discrete Laplace operator $\Delta_h : \mathbb{V}_h \to \mathbb{V}_h$ with respect to the central difference method and the discrete gradient operator $\nabla_h: \mathbb{V}_h \to \mathbb{V}_h \times \mathbb{V}_h$  are defined by
\begin{equation*}
	\Delta_h U_{j,k} = \frac{U_{j+1, k} + U_{j-1, k} + U_{j, k+1} + U_{j,k-1} - 4 U_{j,k}}{h^2}, \quad \nabla_h U_{j,k} = \left( \frac{U_{j+1,k} - U_{j,k}}{h}, \frac{U_{j,k+1} - U_{j,k}}{h} \right)^\top.
\end{equation*}
Due to the periodic boundary condition, it is readily to confirm the following summation by parts formula
\begin{equation*}
	- (\Delta_h U, V)_h = (\nabla_h U, \nabla_h V)_h = -(U, \Delta_h V)_h.
\end{equation*}

The semi-discrete system of the KGS equations \eqref{kgs} is to find $P, Q, U, V \in \mathbb{V}_h$, such that
\begin{equation}\label{kgs-semi}
	\left\lbrace
	\begin{aligned}
		 & \dot{P}  = - \frac{\kappa_1}{2} \Delta_h Q - \gamma U Q,                  \\
		 & \dot{Q} =  \frac{\kappa_1}{2} \Delta_h P + \gamma U P,                    \\
		 & \dot{U} = V,                                                              \\
		 & \dot{V} =  \kappa_2 \Delta_h U - \mu^2 U + \gamma \left(Q^2 + P^2\right).
	\end{aligned}
	\right.
\end{equation}
Set $\bm p=(p_1,p_2,\cdots,p_M)^\top$ and $p_i=P_{j,k}$ for $1\leq i\leq M=N^2$ with $i=(j-1)N+k$, $1\leq j\leq N, 1\leq k\leq N$. Define $\bm z=(\bm p, \bm q, \bm u,\bm v)^\top$. Then we can reformulate \eqref{kgs-semi} into a compact Hamiltonian system:
\begin{equation}\label{kgs-semi-compact}
	\dot{\bm z} = \mathcal{S}\nabla \mathcal{E}_h(\bm z).
\end{equation}
Here, $\mathcal{E}_h\left(\bm z\right)$ denotes the semi-discrete energy defined by
\begin{equation*}
	\mathcal{E}_h\left(\bm z\right) = \frac{1}{2} \left(-\kappa_1 \bm p^\top D\bm p -\kappa_1 \bm q^\top D\bm q -\kappa_2 \bm u^\top D\bm u + \bm v^\top \bm v + \mu^2 \bm u^\top\bm u\right) - \gamma (\bm p^2 + \bm q^2)^\top\bm u,
\end{equation*}
where $D=I_N\otimes D_2+D_2\otimes I_N$ with $D_2$ being the second-order differential matrix and $I_N$ being the identity matrix of dimension $N\times N$. $\mathcal{S}$ is a skew-symmetric matrix as follows:
\begin{equation*}
	\mathcal{S} =
	\begin{pmatrix}
		O                  & \frac{1}{2} I_{M} & O      & O     \\
		-\frac{1}{2} I_{M} & O                 & O      & O     \\
		O                  & O                 & O      & I_{M} \\
		O                  & O                 & -I_{M} & O
	\end{pmatrix}.
\end{equation*}

Thanks to \eqref{kgs-semi-compact}, the conservative property of the semi-discrete scheme can be easily derived.
\begin{thm}
	The semi-discrete scheme \eqref{kgs-semi} preserves the energy conservation law
	\begin{equation*}
		\frac{d}{dt} \mathcal{E}_h(\bm z) = 0.
	\end{equation*}
\end{thm}
\begin{proof}
	Taking the inner product on both sides of \eqref{kgs-semi-compact} with $\nabla \mathcal{E}_h(\bm z)$ and utilizing the skew-symmetry of $\mathcal{S}$, we have
	\begin{equation*}
		\frac{d}{dt} \mathcal{E}_h(\bm z) = \nabla \mathcal{E}_h(\bm z)^\top \dot{\bm z}  =  \nabla \mathcal{E}_h(\bm z)^\top \mathcal{S}\nabla \mathcal{E}_h(\bm z)  = 0.
	\end{equation*}
	The proof is thus completed.
\end{proof}

\section{Dual-partition AVF method for the KGS equations}\label{sec:3}

In this section, we present fully discrete energy conservative schemes for the KGS equations. First, we provide a detailed description of the novel dual-partition AVF method that employs a dual partitioning strategy based on both variables and grid points, and we prove its energy conservation property. Subsequently, we apply this method to the KGS equations and discuss the impact of different partitioning strategies on the resulting schemes.

\subsection{Dual-partition AVF methods}\label{subs:3.1}

We begin with a brief overview of the AVF-type methods, including the original AVF method, the variable-based partitioned AVF method, and the grid-point-based partitioned AVF method. Following this, we introduce the dual-partition AVF method and demonstrate its energy conservation property.

Consider the following multivariate Hamiltonian system
\begin{equation}\label{ham}
	\dot{\bm Y}=\bm J\nabla H(\bm Y),
\end{equation}
where $\bm Y=(\bm y_1, \bm y_2, \cdots, \bm y_m)^\top$, $\bm y_l\in\mathbb{R}^M$ for $l=1,2,\cdots,m$. $\bm J$ is a constant skew-symmetric matrix. The original AVF method for the system \eqref{ham} yields
\begin{equation}\label{avf}
	\frac{\widehat{\bm Y}-\bm Y}{\tau}=\bm J\overline{\nabla}^{\text{avf}}H(\widehat{\bm Y},\bm Y):=\bm J\int_0^1\nabla H\big(\bm Y^\theta\big)\rm d\theta,
\end{equation}
where ${\bm Y}^\theta=\theta\widehat{\bm Y}+(1-\theta)\bm Y$, $\tau$ is the time step size, $\overline{\nabla}^{\rm avf}H(\widehat{\bm Y},\bm Y)$ represents a discrete gradient, and for the AVF method the mean-value discrete gradient is applied. One can easily verify that
\[
	\begin{aligned}
		H(\widehat{\bm Y})-H(\bm Y) & =\int_0^1\frac{\rm d}{\rm d\theta}H\big(\bm Y^\theta\big){\rm d\theta}=\int_0^1\nabla H\big(\bm Y^\theta\big){\rm d}\theta\cdot (\widehat{\bm Y}-\bm Y).
	\end{aligned}
\]
Substituting \eqref{avf} and using the skew-symmetric property of $\bm J$, we have $H(\widehat{\bm Y})=H(\bm Y)$, that is, the AVF method conserves the Hamiltonian or energy.

Note that the discrete gradient of the AVF method with respect to the $ l $-th variable is defined as
\begin{equation}\label{avf-dg}
	\Big[\overline{\nabla}^{\text{avf}} H(\widehat{\bm{Y}}, \bm{Y})\Big]_l = \int_0^1 \nabla_{\bm{y}_l} H\big( \bm{y}_1^\theta, \bm{y}_2^\theta, \cdots, \bm{y}_m^\theta \big) \, d\theta,
\end{equation}
which involves all the unknown variables $\widehat{\bm{y}}_l$ for $l = 1, 2, \ldots, m$. Consequently, the schemes resulting from the AVF method are typically fully implicit, which may lead to high computational costs.

To enhance computational efficiency, we introduced the variable-based partitioned AVF (VP-AVF) method \cite{pavf,epavf}, where the corresponding discrete gradient is given by
\begin{equation*}
	\Big[\overline{\nabla}^{\text{var}} H(\widehat{\bm{Y}}, \bm{Y})\Big]_{\widetilde{l}} = \int_0^1 \nabla_{\bm{y}_{\widetilde{l}}} H\big( \widehat{\bm{y}}_{\widetilde{1}}, \cdots, \widehat{\bm{y}}_{\widetilde{l-1}}, \bm{y}_{\widetilde{l}}^\theta, \bm{y}_{\widetilde{l+1}}, \cdots, \bm{y}_{\widetilde{m}} \big) \, d\theta,
\end{equation*}
where $(\widetilde{1}, \widetilde{2}, \ldots, \widetilde{m})$ is a permutation of $(1, 2, \ldots, m)$. Unlike the AVF discrete gradient \eqref{avf-dg}, which is obtained by averaging $\nabla_{\bm{y}_l} H$ along the segment connecting $\bm{Y}$ and $\widehat{\bm{Y}}$, the discrete gradient of the VP-AVF method averages along a piecewise linear path, with each segment being parallel to one of the variables. The piecewise linear path is uniquely determined by the ordering $(\widetilde{1}, \widetilde{2}, \ldots, \widetilde{m})$, so each ordering corresponds to a VP-AVF scheme. By appropriately ordering the updates of variables, we can construct a variable-wise decoupled VP-AVF scheme for multivariate Hamiltonian systems \eqref{ham}. However, when the systems involve only a single variable, the VP-AVF method is ineffective.

\begin{rmk}
	Here, the perturbed energy should be denoted as $\widetilde{H}({\bm{y}}_{\widetilde{1}}, \cdots, {\bm{y}}_{\widetilde{l-1}}, {\bm{y}}_{\widetilde{l}}, \bm{y}_{\widetilde{l+1}}, \cdots, \bm{y}_{\widetilde{m}})$. However, in the absence of ambiguity and for the sake of notational convenience, we continue to use the notation $H$ to represent the energy after the permutation of variables. 
\end{rmk}

To address this limitation, we proposed the grid-point-based AVF (GP-AVF) method for univariate conservative or dissipative systems \cite{cai_cmame}, where the discrete gradient with respect to the $i$-th grid point is defined as
\begin{equation}\label{grd-dg}
	\Big[\overline{\nabla}^{\text{grd}} H(\widehat{\bm{y}}_1, \bm{y}_1)\Big]_{\widetilde{i}} = \int_0^1 \nabla_{y_{1,\widetilde{i}}} H\big( \widehat{y}_{1,\widetilde{1}}, \cdots, \widehat{y}_{1,\widetilde{i-1}},  y_{1,\widetilde{i}}^\theta, y_{1,\widetilde{i+1}}, \cdots, y_{1,\widetilde{M}} \big) \, d\theta.
\end{equation}
Here, the single variable is denoted by $\bm{y}_1$ and $\widetilde{i}$ is the permuted index of the grid point. The two subscripts in $y_{1,\widetilde{i}}$ denotes the $\widetilde{i}$-th component of the variable $\bm{y}_1$. For notational convenience, we use a single linear index $\widetilde{i}$, where ${i} = {1}, {2}, \ldots, {M}$, to represent the positions of the grid points after permutation. The resulting GP-AVF schemes enable pointwise decoupled computation, avoiding the need to solve large linear or nonlinear systems. Instead, a simple loop from 1 to $M$ is executed, solving only a scalar nonlinear equation in each iteration. Furthermore, with specific updating strategies, these schemes are well-suited for parallel computation.

In this paper, we propose the dual-partition AVF (DP-AVF) method for multivariable coupled Hamiltonian systems by combining the above variable-based and grid-point-based partitioning strategies. The general form of the DP-AVF method is given by
\begin{equation}\label{dp-avfm}
	\frac{\widehat{\bm Y}-\bm Y}{\tau}=\bm J\overline{\nabla}^{\text{dp}}H(\widehat{\bm Y},\bm Y),
\end{equation}
where the discrete analog of the gradient with respect to $y_{\widetilde{l},\widetilde{i}}$ is defined as
\begin{equation}\label{dp-avf}
	\begin{aligned}
		  & \Big[\overline{\nabla}^{\text{dp}}H(\widehat{\bm Y},\bm Y)\Big]_{\widetilde{l},\widetilde{i}}                                                                                                                                                                                                                                                                                                                                        \\
		= & \Big[\overline{\nabla}^{\text{grd}}_{\bm{y}_{\widetilde{l}}} H\left(\widehat{\bm{y}_{\widetilde{l}}},\, \bm{y}_{\widetilde{l}}\,\big|\,\widehat{\bm{y}}_{\widetilde{1}}, \cdots, \widehat{\bm{y}}_{\widetilde{l-1}},  \bm{y}_{\widetilde{l+1}}, \cdots, \bm{y}_{\widetilde{m}}\right)\Big]_{\widetilde{i}}                                                                                                                               \\
		= & \int_0^1\nabla_{y_{\widetilde{l},\widetilde{i}}} H\left(\widehat{y}_{\widetilde{l},\widetilde{1}},\cdots,\widehat{y}_{\widetilde{l},\widetilde{i-1}}, y_{\widetilde{l},\widetilde{i}}^\theta, {y}_{\widetilde{l},\widetilde{i+1}},\cdots, y_{\widetilde{l},\widetilde{M}}\,\big|\,\widehat{\bm y}_{\widetilde{1}}, \cdots, \widehat{\bm y}_{\widetilde{l-1}}, \bm y_{\widetilde{l+1}},\cdots, \bm y_{\widetilde{m}}\right){\rm d}\theta.
	\end{aligned}
\end{equation}
It is easy to see that the dual-partition discrete gradient defined in \eqref{dp-avf} combines the variable-based and grid-point-based partitioning strategies from the VP-AVF and GP-AVF methods. As a result, it offers greater flexibility compared to either method individually.

Next, we will demonstrate that the above-defined DP-AVF method \eqref{dp-avfm} preserves the energy. To begin, we verify that $\overline{\nabla}^{\text{dp}} H(\widehat{\bm{Y}}, \bm{Y})$ is indeed a discrete gradient.
\begin{lem}\label{lem1}
	$\overline{\nabla}^{\text{dp}} H(\widehat{\bm{Y}}, \bm{Y})$ defined by \eqref{dp-avf} is a discrete gradient of $\nabla H(\bm{Y})$.
\end{lem}

\begin{prf}
	By direct calculation, we can obtain
	\[
		\begin{aligned}
			  & \overline{\nabla}^{\text{dp}}H(\widehat{\bm Y},\bm Y)\cdot (\widehat{\bm Y}-\bm Y)                                                                                                                                                                                                                                                                                                                      \\
			= & \sum_{{l}=1}^{m}\sum_{{i}=1}^{M}\Big[\overline{\nabla}^{\text{dp}}H(\widehat{\bm Y},\bm Y)\Big]_{\widetilde{l},\widetilde{i}}\cdot (\widehat{y}_{\widetilde{l},\widetilde{i}}-y_{\widetilde{l},\widetilde{i}})                                                                                                                                                                                          \\
			= & \sum_{{l}=1}^{m}\Bigg[\sum_{{i}=1}^{M}H\big(\widehat{\bm y}_{\widetilde{1}}, \cdots, \widehat{\bm y}_{\widetilde{l-1}}, \widehat{y}_{\widetilde{l},\widetilde{1}},\cdots,\widehat{y}_{\widetilde{l},\widetilde{i-1}}, \widehat{y}_{\widetilde{l},\widetilde{i}}, {y}_{\widetilde{l},\widetilde{i+1}},\cdots, y_{\widetilde{l},\widetilde{M}}, \bm y_{\widetilde{l+1}}\cdots, \bm y_{\widetilde{m}}\big) \\
			  & \qquad\quad -\sum_{{i}=1}^{M}H\big(\widehat{\bm y}_{\widetilde{1}}, \cdots, \widehat{\bm y}_{\widetilde{l-1}}, \widehat{y}_{\widetilde{l},\widetilde{1}},\cdots, \widehat{y}_{\widetilde{l}, \widetilde{i-1}}, y_{\widetilde{l},\widetilde{i}}, {y}_{\widetilde{l},\widetilde{i+1}},\cdots, y_{\widetilde{l},\widetilde{N}}, \bm y_{\widetilde{l+1}}\cdots, \bm y_{\widetilde{m}}\big)\Bigg]            \\
			= & \sum_{{l}=1}^{m}\Big[H\big(\widehat{\bm y}_{\widetilde{1}}, \cdots, \widehat{\bm y}_{{\widetilde{l-1}}}, \widehat{\bm y}_{\widetilde{l}},\bm y_{\widetilde{l+1}}\cdots, \bm y_{\widetilde{m}}\big)-H\big(\widehat{\bm y}_{\widetilde{1}}, \cdots, \widehat{\bm y}_{{\widetilde{l-1}}}, \bm y_{\widetilde{l}},\bm y_{\widetilde{l+1}}\cdots, \bm y_{\widetilde{m}}\big)\Big]                             \\
			= & H\big(\widehat{\bm y}_{\widetilde{1}}, \widehat{\bm y}_{\widetilde{2}}, \cdots \widehat{\bm y}_{\widetilde{m}}\big)-H\big(\bm y_{\widetilde{1}},  \bm y_{\widetilde{2}}, \cdots, \bm y_{\widetilde{m}}\big)                                                                                                                                                                                             \\
			= & H(\widehat{\bm Y})-H(\bm Y).
		\end{aligned}
	\]
	In addition, it is apparently that $\overline{\nabla}^{\text{dp}}H({\bm Y},\bm Y)=\nabla H({\bm Y},\bm Y)$. Thus, following the definition of a discrete gradient \cite{mclachlan1999geometric}, we finish the proof.\qed

\end{prf}

\begin{thm}\label{thm1}
	The DP-AVF method \eqref{dp-avfm} preserves the energy, i.e., $H(\widehat{\bm Y})=H(\bm Y)$.
\end{thm}

\begin{prf}
	According to Lemma \ref{lem1}, we have
	\[
		H(\widehat{\bm Y})-H(\bm Y)=\overline{\nabla}^{\text{dp}}H(\widehat{\bm Y},\bm Y)\cdot (\widehat{\bm Y}-\bm Y)=\tau \overline{\nabla}^{\text{dp}}H(\widehat{\bm Y},\bm Y)^\top\bm J\overline{\nabla}^{dp}H(\widehat{\bm Y},\bm Y)=0.
	\]
\end{prf}

\begin{rmk}\label{rmk1}
	Although each partitioning strategy yields an energy-preserving DP-AVF method, the key to effectively utilizing the DP-AVF method lies in selecting a partitioning approach that enhances computational efficiency. We first ensure that the method obtained after partitioning by variables is efficient, such as being linearly implicit. Then, we further partition by grid points to achieve additional efficiency, such as enabling parallel computation. In this paper, we employ this strategy to construct efficient energy-preserving algorithms for the KGS equations.
\end{rmk}

\begin{rmk}

	Although the discrete gradient \eqref{dp-avf} is based on the perturbed indices, when actually constructing the DP-AVF method \eqref{dp-avfm}, the discrete gradient on the right-hand side, $\overline{\nabla}^{\text{dp}}H(\widehat{\bm Y},\bm Y)$, is still arranged according to the ascending order of variables and grid points. Different partitioning strategies only affect the update order of variables and grid points and do not introduce additional complexity to the construction of the schemes.

\end{rmk}

By using Taylor expansion, it is straightforward to verify that the DP-AVF method \eqref{dp-avfm} is, in fact, a first-order method. To construct higher-order schemes, we can combine the DP-AVF method with its adjoint method. In this work, we focus on constructing a second-order method. The adjoint method of the DP-AVF method \eqref{dp-avfm} is given by:
\begin{equation}\label{dp-avfm-adj}
	\frac{\widehat{\bm Y}-\bm Y}{\tau}=\bm J\overline{\nabla}^{\text{dp},*}H(\widehat{\bm Y},\bm Y),
\end{equation}
where $\overline{\nabla}^{\text{dp},*}H(\widehat{\bm Y},\bm Y)$ is defined as follows
\begin{equation*}
	\begin{aligned}
		  & \Big[\overline{\nabla}^{\text{dp},*}H(\widehat{\bm Y},\bm Y)\Big]_{\widetilde{l},\widetilde{i}}                                                                                                                                                                                                                                                                                                                                                      \\
		= & \Big[\overline{\nabla}^{\text{grd},*}_{\bm{y}_{\widetilde{l}}} H\left(\widehat{\bm{y}_{\widetilde{l}}},\, \bm{y}_{\widetilde{l}}\,\big|\,{\bm{y}}_{\widetilde{1}}, \cdots, {\bm{y}}_{\widetilde{l-1}},  \widehat{\bm{y}}_{\widetilde{l+1}}, \cdots, \widehat{\bm{y}}_{\widetilde{m}}\right)\Big]_{\widetilde{i}}                                                                                                                                         \\
		= & \int_0^1\nabla_{y_{\widetilde{l},\widetilde{i}}} H\left({y}_{\widetilde{l},\widetilde{1}},\cdots,{y}_{\widetilde{l},\widetilde{i-1}}, \widehat{y}_{\widetilde{l},\widetilde{i}}^\theta, \widehat{y}_{\widetilde{l},\widetilde{i+1}},\cdots, \widehat{y}_{\widetilde{l},\widetilde{M}}\,\big|\,{\bm y}_{\widetilde{1}}, \cdots, {\bm y}_{\widetilde{l-1}}, \widehat{\bm y}_{\widetilde{l+1}},\cdots, \widehat{\bm y}_{\widetilde{m}}\right){\rm d}\theta.
	\end{aligned}
\end{equation*}

%
Let the DP-AVF method \eqref{dp-avfm} be denoted as $\widehat{\bm{Y}} = \Phi_\tau(\bm{Y})$, and its adjoint method \eqref{dp-avfm-adj} as $\widehat{\bm{Y}} = \Phi^*_\tau(\bm{Y})$. By applying a symmetric composition, the second-order method can be constructed as
\begin{equation}\label{dp-comp}
	\widehat{\bm Y}=\Phi^*_\frac{\tau}{2}\circ\Phi_\frac{\tau}{2}({\bm Y}).
\end{equation}

\begin{thm}
	The second-order composition method \eqref{dp-comp} also preserves the energy, i.e., $H(\widehat{\bm Y})=H(\bm Y)$.
\end{thm}

\begin{prf}
	The proof is straightforward since any composition of energy-preserving methods that preserve the same discrete energy remains energy-preserving.
\end{prf}

\subsection{Energy-preserving integrators for the KGS equations}

According to Remark \ref{rmk1}, constructing an efficient energy-preserving DP-AVF scheme for the KGS equations involves two steps. In the first step, an appropriate variable-based partitioning strategy is identified, ensuring that the corresponding VP-AVF scheme exhibits efficiency, such as being linearly implicit. It is important to note that the actual construction of the VP-AVF scheme is not required; only the determination of an efficient variable-based partitioning strategy is needed. In the second step, each partitioned variable is further partitioned by grid points to enable pointwise computation and even parallel computation.

\subsubsection{Partition by variables}

For the semi-discrete KGS equation \eqref{kgs-semi-compact}, there are four variables, $(\bm{p}, \bm{q}, \bm{u}, \bm{v})$, resulting in 24 possible orderings. However, not all of these orderings lead to an efficient VP-AVF scheme; we illustrate the approach using two specific orderings as examples. For simplicity, we use the notation $\bm{p} > \bm{q}$ to indicate that the variable $\bm{p}$ is updated before $\bm{q}$.

	{\bf Case I:} $\bm p>\bm q>\bm u>\bm v$. The resulting variable-based discrete gradient $\overline{\nabla}^{\text{var}}\mathcal{E}_h(\bm z^{n+1},\bm z^{n})$ is defined by
\begin{equation*}
	\overline{\nabla}^{\text{var}}\mathcal{E}_h(\bm z^{n+1},\bm z^{n})=\left(\begin{aligned}
			 & \int_0^1{\nabla}_{\bm p}\mathcal{E}_h(\bm p^{\theta},\bm q^n,\bm u^n,\bm v^n)d\theta             \\
			 & \int_0^1{\nabla}_{\bm q}\mathcal{E}_h(\bm p^{n+1},\bm q^{\theta},\bm u^n,\bm v^n)d\theta         \\
			 & \int_0^1{\nabla}_{\bm u}\mathcal{E}_h(\bm p^{n+1},\bm q^{n+1},\bm u^{\theta},\bm v^n)d\theta     \\
			 & \int_0^1{\nabla}_{\bm v}\mathcal{E}_h(\bm p^{n+1},\bm q^{n+1},\bm u^{n+1},\bm v^{\theta})d\theta
		\end{aligned}\right),
\end{equation*}
and the corresponding VP-AVF scheme yields
\begin{subequations}\label{vp-avf1}
	\begin{align}[left={\empheqlbrace}]\label{vp-a}
		\delta_t^+\bm p^{n} & =-\frac{\kappa_1}{2} D \bm q^{n+1/2} - \gamma \bm u^n \bm q^{n+1/2},                                     \\\label{vp-b}
		\delta_t^+\bm q^{n} & =\frac{\kappa_1}{2} D \bm p^{n+1/2} + \gamma \bm u^n \bm p^{n+1/2},                                      \\\label{vp-c}
		\delta_t^+\bm u^{n} & =\bm v^{n+1/2},                                                                                          \\\label{vp-d}
		\delta_t^+\bm v^{n} & =\kappa_2 D \bm u^{n+1/2} - \mu^2 \bm u^{n+1/2} + \gamma \left((\bm q^{n+1})^2 + (\bm p^{n+1})^2\right),
	\end{align}
\end{subequations}
where $\delta_t^+\bm p^n=(\bm p^{n+1}-\bm p^n)/\tau$, $\bm p^{n+1/2}=(\bm p^{n+1}+\bm p^n)/2$ and so on. It is easy to observe that the VP-AVF scheme \eqref{vp-avf1} consists of two linear implicit subsystems. By solving the linear system derived from subsystems \eqref{vp-a} and \eqref{vp-b}, we can obtain $\bm{p}^{n+1}$ and $\bm{q}^{n+1}$. Substituting these values into subsystems \eqref{vp-c} and \eqref{vp-d}, we can similarly obtain $\bm{u}^{n+1}$ and $\bm{v}^{n+1}$ by solving the corresponding linear system.

	{\bf Case II:} $\bm p>\bm u>\bm q>\bm v$. The corresponding variable-based discrete gradient $\overline{\nabla}^{\text{var}}\mathcal{E}_h(\bm z^{n+1},\bm z^{n})$ now becomes
\begin{equation*}
	\overline{\nabla}^{\text{var}}\mathcal{E}_h(\bm z^{n+1},\bm z^{n})=\left(\begin{aligned}
			 & \int_0^1{\nabla}_{\bm p}\mathcal{E}_h(\bm p^{\theta},\bm q^n,\bm u^n,\bm v^n)d\theta             \\
			 & \int_0^1{\nabla}_{\bm q}\mathcal{E}_h(\bm p^{n+1},\bm q^{\theta},\bm u^{n+1},\bm v^n)d\theta     \\
			 & \int_0^1{\nabla}_{\bm u}\mathcal{E}_h(\bm p^{n+1},\bm q^{n},\bm u^{\theta},\bm v^n)d\theta       \\
			 & \int_0^1{\nabla}_{\bm v}\mathcal{E}_h(\bm p^{n+1},\bm q^{n+1},\bm u^{n+1},\bm v^{\theta})d\theta
		\end{aligned}\right),
\end{equation*}
and the resulting VP-AVF scheme is given by
\begin{equation}\label{vp-avf2}
	\left\{\begin{aligned}
		\delta_t^+\bm p^{n} & =-\frac{\kappa_1}{2} D \bm q^{n+1/2} - \gamma \bm u^{n+1} \bm q^{n+1/2},                               \\
		\delta_t^+\bm q^{n} & =\frac{\kappa_1}{2} D \bm p^{n+1/2} + \gamma \bm u^n \bm p^{n+1/2},                                    \\
		\delta_t^+\bm u^{n} & =\bm v^{n+1/2},                                                                                        \\
		\delta_t^+\bm v^{n} & =\kappa_2 D \bm u^{n+1/2} - \mu^2 \bm u^{n+1/2} + \gamma \left((\bm q^{n})^2 + (\bm p^{n+1})^2\right).
	\end{aligned}\right.
\end{equation}
It is evident that the VP-AVF scheme \eqref{vp-avf2} does not contain any linearly implicit subsystems, requiring iterative methods to solve a nonlinear system involving four variables. This example highlights the critical role of reordering in constructing efficient VP-AVF schemes.

It is worth noting that the reordering used in scheme \eqref{vp-avf1} is not the only option for constructing a linearly implicit VP-AVF scheme for the KGS equations. For example, the reordering $\bm{v} > \bm{u} > \bm{q} > \bm{p}$ also produces a linearly implicit VP-AVF scheme, which actually is the adjoint of the scheme \eqref{vp-avf1}, obtained by swapping $n+1 \leftrightarrow n$ and $\tau \leftrightarrow -\tau$. Additionally, different reorderings may result in the same scheme. For instance, the reorderings $\bm{q} > \bm{p} > \bm{u} > \bm{v}$ and $\bm{p} > \bm{q} > \bm{v} > \bm{u}$ both lead to a scheme equivalent to \eqref{vp-avf1}.

Without loss of generality, in the following discussion, we will adopt two variable reorderings, $\bm{p} > \bm{q} > \bm{u} > \bm{v}$ and $\bm{v} > \bm{u} > \bm{q} > \bm{p}$, to further partition by grid points. Using these reorderings, we will construct the DP-AVF scheme and its adjoint for the KGS equations, and through symmetric composition, derive a second-order efficient energy-preserving scheme.

\subsubsection{Partition by grid points}

Based on the ordering $\bm{p} > \bm{q} > \bm{u} > \bm{v}$, we have already determined a variable-based discrete gradient or the VP-AVF method to construct a linearly implicit scheme. Building upon this, in this section, we will further partition each variable by grid points to construct the dual-partition discrete gradient and the corresponding DP-AVF scheme.

For the KGS equations, the dual-partition discrete gradient defined in \eqref{dp-avf} can be written as
\begin{equation}\label{dp-kgs}
	\begin{aligned}
		\overline{\nabla}^{\text{dp}}\mathcal{E}_h(\bm z^{n+1},\bm z^n)=\left(\begin{array}{l}
		\overline{\nabla}^{\text{grd}}_{\bm p}\mathcal{E}_h(\bm p^{n+1},\bm p^{n}\,\big|\, \bm q^n,\bm u^n,\bm v^n)         \\[1ex]
		\overline{\nabla}^{\text{grd}}_{\bm q}\mathcal{E}_h(\bm q^{n+1},\bm q^{n}\,\big|\, \bm p^{n+1},\bm u^n,\bm v^n)     \\[1ex]
		\overline{\nabla}^{\text{grd}}_{\bm u}\mathcal{E}_h(\bm u^{n+1},\bm u^{n}\,\big|\, \bm p^{n+1},\bm q^{n+1},\bm v^n) \\[1ex]
		\overline{\nabla}^{\text{grd}}_{\bm v}\mathcal{E}_h(\bm v^{n+1},\bm v^{n}\,\big|\, \bm p^{n+1},\bm q^{n+1},\bm u^{n+1})
		\end{array}\right),
	\end{aligned}
\end{equation}
which means that, based on the variable reordering, we construct the grid-point-based discrete gradient for each variable separately. Next, we will derive the four discrete gradients $\overline{\nabla}^{\text{grd}}_{\bm p}\mathcal{E}_h, \overline{\nabla}^{\text{grd}}_{\bm q}\mathcal{E}_h, \overline{\nabla}^{\text{grd}}_{\bm u}\mathcal{E}_h$ and $ \overline{\nabla}^{\text{grd}}_{\bm v}\mathcal{E}_h$, respectively.

For simplicity in notation, we continue to use the index $i$ to represent the linear index after reordering, instead of $\widetilde{i}$. We also denote $\bm p^{\theta(i)} = (p_1^{n+1}, \cdots, p_{i-1}^{n+1}, p_i^{\theta}, p_{i+1}^n, \cdots, p_M^n)^\top$, and so on. According to the definition in \eqref{grd-dg}, we have
\begin{equation}\label{grd-p}
	\begin{aligned}
		\left[\overline{\nabla}^{\text{grd}}_{\bm p}\mathcal{E}_h(\bm p^{n+1}, \bm p^{n}\,\big|\, \bm q^n, \bm u^n, \bm v^n)\right]_{i}
		 & = \int_0^1 \left[\nabla_{\bm p}\mathcal{E}_h(\bm p^{\theta(i)}; \bm q^n, \bm u^n, \bm v^n)\right]_i \, d\theta \\
		 & = \int_0^1 \left(-\kappa_1 D \bm p^{\theta(i)} - 2\gamma \bm u^n \bm p^{\theta(i)}\right)_i \, d\theta         \\
		 & = \int_0^1 \left(-\kappa_1 \Delta_h P_{j,k}^\theta - 2\gamma U_{j,k}^n P_{j,k}^\theta\right) \, d\theta,
	\end{aligned}
\end{equation}
where in the last equality, we use the correspondence between the linear indices $i$ and the double indices $(j,k)$, i.e., $\left[D \bm p\right]_i = \Delta_h P_{j,k}$. Due to the local nature of the discrete Laplace operator, we only need to know the ordering of the four neighboring grid points around $(j, k)$ to evaluate the integral of $\Delta_h P_{j,k}^\theta$. Without loss of generality, we denote
\begin{equation}\label{LapP}
	\int_0^1\Delta_h P_{j,k}^\theta d\theta= \frac{1}{h^2} \left(P_{j,k-1}^{n+\text{S}} + P_{j-1, k}^{n+\text{W}} - 4 P_{j,k}^{\theta} + P_{j+1, k}^{n+\text{E}} + P_{j, k+1}^{n+\text{N}} \right),
\end{equation}
where the superscripts $(\text{S}, \text{W}, \text{E}, \text{N})$ take values of either $1$ or $0$, depending on the grid point's relative position to $(j, k)$.

Specifically, if the ordering of the five grid points is:
\begin{equation}\label{order1}
	(j, k-1) > (j-1, k) > (j, k) > (j+1, k) > (j, k+1),
\end{equation}
then we have:
\[
	\int_0^1\Delta_h P_{j,k}^\theta d\theta = \frac{1}{h^2} \left(P_{j,k-1}^{n+1} + P_{j-1, k}^{n+1} - 4 P_{j,k}^{\theta} + P_{j+1, k}^n + P_{j, k+1}^n \right).
\]
This means that the grid points ordered before $(j, k)$ use the values at the $n+1$ level for approximation, while those ordered after use the values at the $n$ level. Based on this criterion, if the opposite ordering:
\begin{equation}\label{order2}
	(j, k+1) > (j+1, k) > (j, k) > (j-1, k) > (j, k-1)
\end{equation}
is adopted, we have:
\[
	\int_0^1\Delta_h P_{j,k}^\theta d\theta = \frac{1}{h^2} \left(P_{j,k-1}^n + P_{j-1, k}^n - 4 P_{j,k}^{\theta} + P_{j+1, k}^{n+1} + P_{j, k+1}^{n+1} \right).
\]
Figure~\ref{order}(a) and (b) illustrate two specific grid point ordering strategies, where each grid point $(j,k)$ follows the above two kinds of orderings \eqref{order1} and \eqref{order2}, respectively.
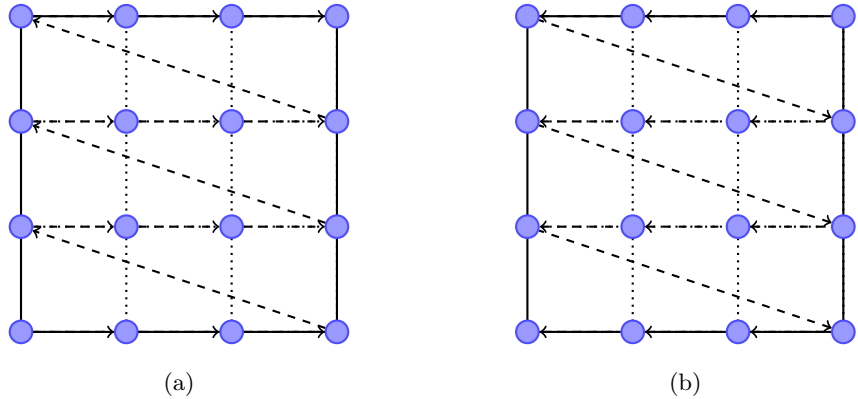
\begin{figure}[H]
	\centering
	\begin{minipage}[t]{0.4\textwidth}
		\centering
		\begin{tikzpicture}[scale=0.7,inner sep=0pt,minimum size=3mm,thick]

			\draw[solid] (0,0) -- (6,0)--(6,-6)--(0,-6)--cycle;

			\foreach \x in {-6,-4,-2}
			\draw[dotted] (0,\x)--(6,\x);

			\foreach \x in {2,4,6}
			\draw[dotted] (\x,0)--(\x,-6);

			\node(10) at (0,-6) [circle,draw=blue!70,fill=blue!40] {};
			\node(11) at (2,-6) [circle,draw=blue!70,fill=blue!40] {};
			\node(12) at (4,-6) [circle,draw=blue!70,fill=blue!40] {};
			\node(13) at (6,-6) [circle,draw=blue!70,fill=blue!40] {};

			\foreach \x in {2,3,4}
			\foreach \y in {0,1,2,3}
			\node(\x\y) at (2*\y,2*\x-8) [circle,draw=blue!70,fill=blue!40] {};

			\path[->,dashed] (10) edge (11);
			\path[->,dashed] (11) edge (12);
			\path[->,dashed] (12) edge (13);
			\path[->,dashed] (13) edge (20);
			\path[->,dashed] (20) edge (21);
			\path[->,dashed] (21) edge (22);
			\path[->,dashed] (22) edge (23);
			\path[->,dashed] (23) edge (30);
			\path[->,dashed] (30) edge (31);
			\path[->,dashed] (31) edge (32);
			\path[->,dashed] (32) edge (33);
			\path[->,dashed] (33) edge (40);
			\path[->,dashed] (40) edge (41);
			\path[->,dashed] (41) edge (42);
			\path[->,dashed] (42) edge (43);

			\node at (3,-7) {\footnotesize (a)};
		\end{tikzpicture}

	\end{minipage}
	\begin{minipage}[t]{0.4\textwidth}
		\centering

		\begin{tikzpicture}[scale=0.7,inner sep=0pt,minimum size=3mm,thick]

			\draw[solid] (0,0) -- (6,0)--(6,-6)--(0,-6)--cycle;

			\foreach \x in {-6,-4,-2}
			\draw[dotted] (0,\x)--(6,\x);

			\foreach \x in {2,4,6}
			\draw[dotted] (\x,0)--(\x,-6);

			\node(10) at (0,-6) [circle,draw=blue!70,fill=blue!40] {};
			\node(11) at (2,-6) [circle,draw=blue!70,fill=blue!40] {};
			\node(12) at (4,-6) [circle,draw=blue!70,fill=blue!40] {};
			\node(13) at (6,-6) [circle,draw=blue!70,fill=blue!40] {};

			\foreach \x in {2,3,4}
			\foreach \y in {0,1,2,3}
			\node(\x\y) at (2*\y,2*\x-8) [circle,draw=blue!70,fill=blue!40] {};

			\path[<-,dashed] (10) edge (11);
			\path[<-,dashed] (11) edge (12);
			\path[<-,dashed] (12) edge (13);
			\path[<-,dashed] (13) edge (20);
			\path[<-,dashed] (20) edge (21);
			\path[<-,dashed] (21) edge (22);
			\path[<-,dashed] (22) edge (23);
			\path[<-,dashed] (23) edge (30);
			\path[<-,dashed] (30) edge (31);
			\path[<-,dashed] (31) edge (32);
			\path[<-,dashed] (32) edge (33);
			\path[<-,dashed] (33) edge (40);
			\path[<-,dashed] (40) edge (41);
			\path[<-,dashed] (41) edge (42);
			\path[<-,dashed] (42) edge (43);

			\node at (3,-7) {\footnotesize (b)};
		\end{tikzpicture}

	\end{minipage}
	\caption{Illustration of two orderings for defining the discrete gradient. The blue circles represent the unknowns, and the dashed arrows indicate the computing orders. (a) Each grid point follows the ordering \eqref{order1}; (b) Each grid point follows the ordering \eqref{order2}.} \label{order}
\end{figure}

Substituting the discrete Laplace \eqref{LapP} into the last equality of \eqref{grd-p}, we obtain
\[
	\begin{aligned}
		\left[\overline{\nabla}^{\text{grd}}_{\bm p}\mathcal{E}_h(\bm q^{n+1},\bm q^{n}\,\big|\, \bm p^{n+1},\bm u^n,\bm v^n)\right]_{i}
		= & -\frac{\kappa_1}{h^2} \left(P_{j,k-1}^{n+\text{S}} + P_{j-1, k}^{n+\text{W}} - 4 P_{j,k}^{n+1/2} + P_{j+1, k}^{n+\text{E}} + P_{j, k+1}^{n+\text{N}}\right) \\
		  & - 2\gamma U_{j,k}^n P_{j,k}^{n+1/2}.
	\end{aligned}
\]
where the values of the superscripts $(\text{S}, \text{W}, \text{E}, \text{N})$ depend on the ordering of the five grid points around $(j, k)$.

Assuming that the variables $\bm{q}$, $\bm{u}$, and $\bm{v}$ follow the same ordering as $\bm{p}$ for the five grid points around $(j, k)$, we can derive their discrete gradients similarly. For $\bm q$ and $\bm u$, since the gradients involve the Laplace operator, the cooresponding discrete gradients are given by
\[
	\begin{aligned}
		\left[\overline{\nabla}^{\text{grd}}_{\bm q}\mathcal{E}_h(\bm q^{n+1},\bm q^{n}\,\big|\, \bm p^{n+1},\bm u^n,\bm v^n)\right]_{i}
		= & \int_0^1 \left(-\kappa_1 \Delta_h Q_{j,k}^\theta - 2\gamma U_{j,k}^n Q_{j,k}^\theta\right) \, d\theta                                                       \\
		= & -\frac{\kappa_1}{h^2} \left(Q_{j,k-1}^{n+\text{S}} + Q_{j-1, k}^{n+\text{W}} - 4 Q_{j,k}^{n+1/2} + Q_{j+1, k}^{n+\text{E}} + Q_{j, k+1}^{n+\text{N}}\right) \\
		  & - 2\gamma U_{j,k}^n Q_{j,k}^{n+1/2},                                                                                                                        \\
		\left[\overline{\nabla}^{\text{grd}}_{\bm u}\mathcal{E}_h(\bm u^{n+1},\bm u^{n}\,\big|\, \bm p^{n+1},\bm q^{n+1},\bm v^n)\right]_{i}
		= & \int_0^1 \left(-\kappa_2 \Delta_h U_{j,k}^\theta + \mu^2 U_{j,k}^\theta - \gamma \left((P_{j,k}^{n+1})^2 + (Q_{j,k}^{n+1})^2\right)\right) \, d\theta       \\
		= & -\frac{\kappa_2}{h^2} \left(U_{j,k-1}^{n+\text{S}} + U_{j-1, k}^{n+\text{W}} - 4 U_{j,k}^{n+1/2} + U_{j+1, k}^{n+\text{E}} + U_{j, k+1}^{n+\text{N}}\right) \\
		  & + \mu^2 U_{j,k}^{n+1/2} - \gamma \left((P_{j,k}^{n+1})^2 + (Q_{j,k}^{n+1})^2\right).
	\end{aligned}
\]
While for $\bm{v}$, since there are no spatial derivatives, the discrete gradient simplifies to:
\[
	\left[\overline{\nabla}^{\text{grd}}_{\bm v}\mathcal{E}_h(\bm v^{n+1},\bm v^{n}\,\big|\, \bm p^{n+1},\bm q^{n+1},\bm u^{n+1})\right]_{i} = V_{j,k}^{n+1/2}.
\]

The advantage of the energy-preserving scheme constructed through grid-point-based partitioning lies in its ability to perform pointwise computations. Therefore, we do not need to integrate the discrete gradients at each point into the vector form as \eqref{dp-kgs}. Instead, we can directly consider the DP-AVF scheme corresponding to any grid point. According to the DP-AVF method \eqref{dp-avfm} and the compact form of the KGS equations \eqref{kgs-equiv}, we derive the system for $(\bm p_i, \bm q_i, \bm u_i, \bm v_i)$ as follows
\begin{equation}\label{dp-avf-i}
	\left(\begin{aligned}
			\delta_t^+\bm p_i^n \\
			\delta_t^+\bm q_i^n \\
			\delta_t^+\bm u_i^n \\
			\delta_t^+\bm v_i^n \\
		\end{aligned}\right)=
	\begin{pmatrix}
		0            & \frac{1}{2} & 0  & 0 \\
		-\frac{1}{2} & 0           & 0  & 0 \\
		0            & 0           & 0  & 1 \\
		0            & 0           & -1 & 0
	\end{pmatrix}\left(\begin{aligned}
			 & \left[\overline{\nabla}^{\text{grd}}_{\bm p}\mathcal{E}_h(\bm p^{n+1},\bm p^{n}\,\big|\, \bm q^n,\bm u^n,\bm v^n)\right]_{i}             \\
			 & \left[\overline{\nabla}^{\text{grd}}_{\bm q}\mathcal{E}_h(\bm q^{n+1},\bm q^{n}\,\big|\, \bm p^{n+1},\bm u^n,\bm v^n)\right]_{i}         \\
			 & \left[\overline{\nabla}^{\text{grd}}_{\bm u}\mathcal{E}_h(\bm u^{n+1},\bm u^{n}\,\big|\, \bm p^{n+1},\bm q^{n+1},\bm v^n)\right]_{i}     \\
			 & \left[\overline{\nabla}^{\text{grd}}_{\bm v}\mathcal{E}_h(\bm v^{n+1},\bm v^{n}\,\big|\,\bm p^{n+1},\bm q^{n+1},\bm u^{n+1})\right]_{i}
		\end{aligned}\right).
\end{equation}
Substituting the obtained grid-point-based discrete gradient and transforming to the index $(j,k)$, we have the equivalent form of \eqref{dp-avf-i}:
\begin{equation}\label{dp-avf-kj}
	\left\{\begin{aligned}
		\delta_t^+P_{j,k}^n= & -\frac{\kappa_1}{2h^2}(Q_{j,k-1}^{n+\text{S}}+Q_{j-1, k}^{n+\text{W}} - 4 Q_{j,k}^{n+1/2}+ Q_{j+1, k}^{n+\text{E}} + Q_{j, k+1}^{n+\text{N}})-\gamma U_{j,k}^nQ_{j,k}^{n+1/2},                                             \\
		\delta_t^+Q_{j,k}^n= & \frac{\kappa_1}{2h^2}(P_{j,k-1}^{n+\text{S}}+P_{j-1, k}^{n+\text{W}} - 4 P_{j,k}^{n+1/2}+ P_{j+1, k}^{n+\text{E}} + P_{j, k+1}^{n+\text{N}})+\gamma U_{j,k}^nP_{j,k}^{n+1/2},                                              \\
		\delta_t^+U_{j,k}^n= & V_{j,k}^{n+1/2},                                                                                                                                                                                                           \\
		\delta_t^+V_{j,k}^n= & \frac{\kappa_2}{h^2}(U_{j,k-1}^{n+\text{S}}+U_{j-1, k}^{n+\text{W}} - 4 U_{j,k}^{n+1/2}+ U_{j+1, k}^{n+\text{E}} + U_{j, k+1}^{n+\text{N}})\\
		&-\mu^2 U_{j,k}^{n+1/2}+\gamma \big((P_{j,k}^{n+1})^2+(Q_{j,k}^{n+1})^2\big).
	\end{aligned}\right.
\end{equation}

To construct a second-order energy-preserving scheme, we consider to the reverse ordering of variables $\bm{v} > \bm{u} > \bm{q} > \bm{p}$ and define the adjoint of the dual-partition discrete gradient \eqref{dp-kgs} as
\begin{equation*}
	\begin{aligned}
		\overline{\nabla}^{\text{dp},*}\mathcal{E}_h(\bm z^{n+1},\bm z^n)=\left(\begin{array}{l}
		\overline{\nabla}^{\text{grd},*}_{\bm p}\mathcal{E}_h(\bm p^{n+1},\bm p^{n}\,\big|\, \bm q^{n+1},\bm u^{n+1},\bm v^{n+1}) \\[1ex]
		\overline{\nabla}^{\text{grd},*}_{\bm q}\mathcal{E}_h(\bm q^{n+1},\bm q^{n}\,\big|\, \bm p^{n},\bm u^{n+1},\bm v^{n+1})   \\[1ex]
		\overline{\nabla}^{\text{grd},*}_{\bm u}\mathcal{E}_h(\bm u^{n+1},\bm u^{n}\,\big|\, \bm p^{n},\bm q^{n},\bm v^{n+1})     \\[1ex]
		\overline{\nabla}^{\text{grd},*}_{\bm v}\mathcal{E}_h(\bm v^{n+1},\bm v^{n}\,\big|\, \bm p^{n},\bm q^{n},\bm u^{n})
	 \end{array}\right),
	\end{aligned}
\end{equation*}
where the grid-point-based discrete gradient is defined by 
\begin{equation*}
	\begin{aligned}
		\left[\overline{\nabla}^{\text{grd},*}_{\bm p}\mathcal{E}_h(\bm p^{n+1},\bm p^{n}; \bm q^{n+1},\bm u^{n+1},\bm v^{n+1})\right]_{i}= & \int_0^1\left[\nabla_{\bm p}\mathcal{E}_h(\bm p^{\theta(i),*}\,\big|\,\bm q^{n+1},\bm u^{n+1},\bm v^{n+1})\right]_i d\theta                                                                      \\
		=                                                                                                                                   & \int_0^1\left[-\kappa_1D \bm p^{\theta(i),*}-2\gamma \bm u^{n+1}\bm p^{\theta(i),*}\right]_i d\theta                                                                                     \\
		=                                                                                                                                   & \int_0^1\left(-\kappa_1\Delta_h P_{j,k}^{\theta,*}-2\gamma U_{j,k}^{n+1}P_{j,k}^{\theta,*}\right) d\theta                                                                                \\
		=                                                                                                                                   & -\frac{\kappa_1}{h^2}(P_{j,k-1}^{n+\overline{\text{S}}}+P_{j-1, k}^{n+\overline{\text{W}}} - 4 P_{j,k}^{n+1/2}+ P_{j+1, k}^{n+\overline{\text{E}}} + P_{j, k+1}^{n+\overline{\text{N}}}) \\
		    &-2\gamma U_{j,k}^{n+1}P_{j,k}^{n+1/2}.
	\end{aligned}
\end{equation*}
Here, $\bm p^{\theta(i),*} = (p_1^{n}, \cdots, p_{i-1}^{n}, p_i^{\theta}, p_{i+1}^{n+1}, \cdots, p_M^{n+1})^\top$ represents the reverse ordering of $\bm p^{\theta(i)}$. Therefore, we have $\overline{\text{S}}=1-{\text{S}}$ and so on which corresponds the exchange of $n+1\leftrightarrow n$. Similarly, we can derive the rest discrete gradient $\overline{\nabla}^{\text{grd},*}_{\bm q}\mathcal{E}_h, \overline{\nabla}^{\text{grd},*}_{\bm u}\mathcal{E}_h, \overline{\nabla}^{\text{grd},*}_{\bm v}\mathcal{E}_h$, and the corresponding adjoint DP-AVF scheme for the KGS equations yields
\begin{equation}\label{dp-avf-kj-adj}
	\left\{\begin{aligned}
		\delta_t^+P_{j,k}^n= & -\frac{\kappa_1}{2h^2}(Q_{j,k-1}^{n+\overline{\text{S}}}+Q_{j-1, k}^{n+\overline{\text{W}}} - 4 Q_{j,k}^{n+1/2}+ Q_{j+1, k}^{n+\overline{\text{E}}} + Q_{j, k+1}^{n+\overline{\text{N}}})-\gamma U_{j,k}^{n+1}Q_{j,k}^{n+1/2},                                     \\
		\delta_t^+Q_{j,k}^n= & \frac{\kappa_1}{2h^2}(P_{j,k-1}^{n+\overline{\text{S}}}+P_{j-1, k}^{n+\overline{\text{W}}} - 4 P_{j,k}^{n+1/2}+ P_{j+1, k}^{n+\overline{\text{E}}} + P_{j, k+1}^{n+\overline{\text{N}}})+\gamma U_{j,k}^{n+1}P_{j,k}^{n+1/2},                                      \\
		\delta_t^+U_{j,k}^n= & V_{j,k}^{n+1/2},                                                                                                                                                                                                                                                   \\
		\delta_t^+V_{j,k}^n= & \frac{\kappa_2}{h^2}(U_{j,k-1}^{n+\overline{\text{S}}}+U_{j-1, k}^{n+\overline{\text{W}}} - 4 U_{j,k}^{n+1/2}+ U_{j+1, k}^{n+\overline{\text{E}}} + U_{j, k+1}^{n+\overline{\text{N}}})\\
		&-\mu^2 U_{j,k}^{n+1/2}+\gamma \big((P_{j,k}^{n})^2+(Q_{j,k}^{n})^2\big).
	\end{aligned}\right.
\end{equation}

By applying the symmetric composition \eqref{dp-comp} to the schemes \eqref{dp-avf-kj} and \eqref{dp-avf-kj-adj}, we can derive a second-order energy-preserving scheme for the KGS equations, referred to as DP-AVF2 in the subsequent discussion. It is important to note that DP-AVF2 represents a broad family of energy-preserving methods, as the subscripts $(\text{S}, \text{W}, \text{E}, \text{N})$ are determined by the partitioning strategies chosen for the grid points, which offer a variety of options.

For example, if the update order shown in Figure \ref{order}(a) is adopted, where each grid point $(j, k)$ satisfies the ordering \eqref{order1}, then in the corresponding scheme \eqref{dp-avf-kj}, the parameters are $\text{S} = \text{W} = 1$ and $\text{E} = \text{N} = 0$ for all $j, k = 1, 2, \dots, N$. Conversely, if the update order shown in Figure \ref{order}(b) is used, where each grid point $(j, k)$ satisfies the reverse ordering \eqref{order2}, then in the corresponding scheme \eqref{dp-avf-kj}, the parameters are $\text{S} = \text{W} = 0$ and $\text{E} = \text{N} = 1$ for all $j, k = 1, 2, \dots, N$.  In general, once a specific update order for the grid points is given, the DP-AVF scheme \eqref{dp-avf-kj} is fully determined. Consequently, its adjoint scheme \eqref{dp-avf-kj-adj} is also uniquely defined. Combining the two yields the second-order energy-preserving scheme DP-AVF2 for the KGS equations.

\begin{rmk}
	It is important to note that for the sake of simplicity in both notation and exposition, we assumed that all variables $\bm{p}, \bm{q}, \bm{u}, \bm{v}$ follow the same grid-point-based partitioning strategy. However, according to the proof of energy conservation, each variable can actually adopt a different grid-point-based partitioning strategy, and the resulting DP-AVF scheme for the KGS equations will still preserve energy.
\end{rmk}

\section{Explicit and Parallel implementation}\label{sec:4}

In this section,  we will demonstrate that, regardless of the grid point ordering, the DP-AVF schemes \eqref{dp-avf-kj}, \eqref{dp-avf-kj-adj}, and their composition scheme DP-AVF2 enable pointwise explicit computation. Moreover, with a suitable ordering strategy, the scheme can be parallelized for both CPU and GPU implementations.

\subsection{Pointwise explicit computation}

Let $\Psi_{j,k}=P_{i,j}+iQ_{j,k}$. After some arrangement, we can rewrite the DP-AVF scheme \eqref{dp-avf-kj}  as
\begin{equation}\label{scheme1}
	\left\{\begin{aligned}
		 & (i-\frac{\tau\kappa_1}{h^2}+\frac{\tau\gamma}{2}U_{j,k}^n) \Psi_{j,k}^{n+1}=(i+\frac{\tau\kappa_1}{h^2}-\frac{\tau\gamma}{2} U_{j,k}^n)\Psi_{j,k}^{n}                     \\
		 & \qquad\qquad-\frac{\tau\kappa_1}{2h^2}(\Psi_{j,k-1}^{n+\text{S}}+\Psi_{j-1, k}^{n+\text{W}} + \Psi_{j+1, k}^{n+\text{E}} + \Psi_{j, k+1}^{n+\text{N}}),                   \\
		 & U_{j,k}^{n+1}-\frac{\tau}{2} V_{j,k}^{n+1}=U_{j,k}^n+\frac{\tau}{2} V_{j,k}^{n},                                                                                          \\
		 & (\frac{2\tau\kappa_2}{h^2}+\frac{\tau\mu^2}{2} )U_{j,k}^{n+1}+V_{j,k}^{n+1}=V_{j,k}^n-(\frac{2\tau\kappa_2}{h^2}+\frac{\tau\mu^2}{2}) U_{j,k}^{n}                         \\
		 & \qquad\qquad+\frac{\tau\kappa_2}{h^2}(U_{j,k-1}^{n+\text{S}}+U_{j-1, k}^{n+\text{W}}+ U_{j+1, k}^{n+\text{E}} + U_{j, k+1}^{n+\text{N}})+\tau\gamma |\Psi_{j,k}^{n+1}|^2.
	\end{aligned}\right.j, k = 1, \cdots, N.
\end{equation}
Apparently, if the values at the four neighboring grid points of $(j,k)$ are known, we can explicitly solve for $\Psi_{j,k}^{n+1}$ from the first equation. Substituting this into the last equation, and noting that the $2\times 2$ linear system for $U_{j,k}^{n+1}$ and $V_{j,k}^{n+1}$ has constant coefficients (its inverse can be computed in advance), we can also explicitly calculate $U_{j,k}^{n+1}$ and $V_{j,k}^{n+1}$. By iterating over all $j,k$ indices, the entire scheme can be computed pointwise in an explicit manner. In the following, we will demonstrate that if the update order during computation aligns with the ordering of the grid points, then when updating $\Psi_{j,k}^{n+1}$ and $U_{j,k}^{n+1}$, the values at the four neighboring nodes are precomputed, regardless of whether the superscripts $\text{S}, \text{W}, \text{E}, \text{N}$ take 1 or 0.

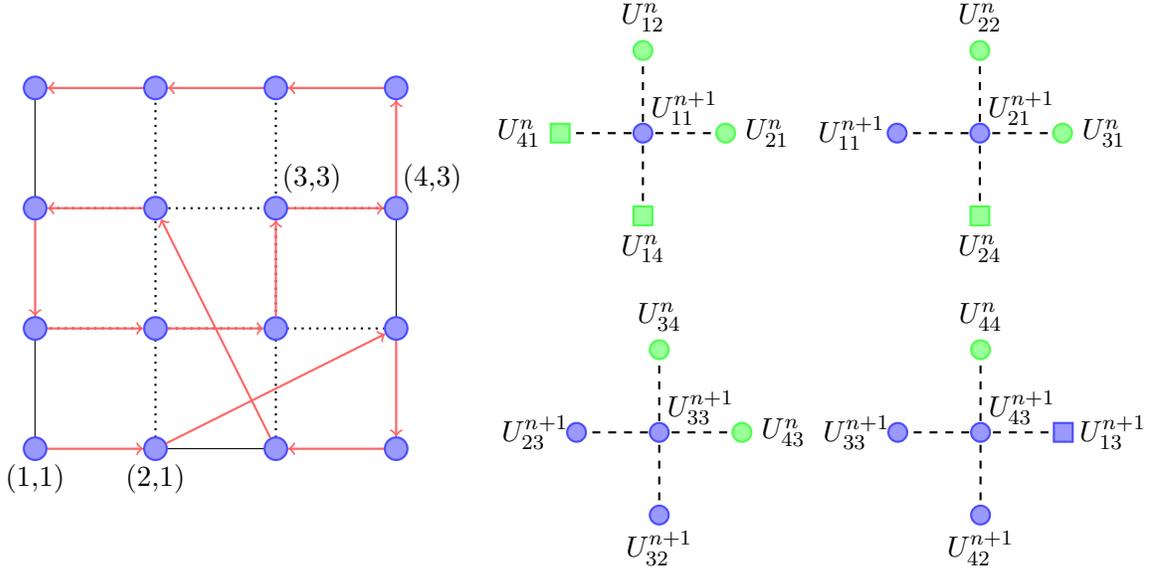
\begin{figure}[H]
	\centering

	\begin{minipage}{0.4\textwidth}
		\centering
		\begin{tikzpicture}[scale=0.8,inner sep=0pt,minimum size=3mm,thick]

			\draw[solid,thin] (0,0) --(0,-6)--(6,-6)--(6,0)--(0,0);

			\foreach \x in {2,4}
			\draw[dotted] (\x,0)--(\x,-6);

			\foreach \y in {-4,-2}
			\draw[dotted] (0,\y)--(6,\y);

			\foreach \x in {0,1,2,3} {
					\foreach \y in {0,1,2,3} {
							\node(\x\y) at (2*\x, -2*\y) [circle,draw=blue!70,fill=blue!40] {};
						}
				}

			\draw[->, red!60,thick] (03) -- (13); 
			\draw[->, red!60,thick] (13) -- (32); 
			\draw[->, red!60,thick] (32) -- (33); 
			\draw[->, red!60,thick] (33) -- (23); 
			\draw[->, red!60,thick] (23) -- (11); 
			\draw[->, red!60,thick] (11) -- (01); 
			\draw[->, red!60,thick] (01) -- (02); 
			\draw[->, red!60,thick] (02) -- (12); 
			\draw[->, red!60,thick] (12) -- (22); 
			\draw[->, red!60,thick] (22) -- (21); 
			\draw[->, red!60,thick] (21) -- (31); 
			\draw[->, red!60,thick] (31) -- (30); 
			\draw[->, red!60,thick] (30) -- (20); 
			\draw[->, red!60,thick] (20) -- (10); 
			\draw[->, red!60,thick] (10) -- (00); 

			\node at (0,-6.5) {(1,1)};
			\node at (2,-6.5) {(2,1)};
			\node at (4.6,-1.5) {(3,3)};
			\node at (6.6,-1.5) {(4,3)};

		\end{tikzpicture}
	\end{minipage}
	\begin{minipage}{0.5\textwidth}
		\centering
		\begin{minipage}{0.45\textwidth}
			\centering
			\begin{tikzpicture}[scale=0.55,inner sep=0pt,minimum size=2.5mm,thick]
				\draw[dashed] (0,-2)--(4,-2);
				\draw[dashed] (2,0)--(2,-4);
				\node at (4,-2) [circle,draw=green!70,fill=green!40]{};
				\node at (2,-2) [circle,draw=blue!70,fill=blue!40] {};
				\node at (2,0) [circle,draw=green!70,fill=green!40] {};
				\node at (0,-2) [rectangle,draw=green!70,fill=green!40] {};
				\node at (2,-4) [rectangle,draw=green!70,fill=green!40] {};

				\node at (3,-1.4) {$U_{11}^{n+1}$};
				\node at (2,0.8) {$U_{12}^{n}$};
				\node at (2,-4.8) {$U_{14}^{n}$};
				\node at (-1,-2) {$U_{41}^{n}$};
				\node at (5,-2) {$U_{21}^{n}$};

			\end{tikzpicture}
		\end{minipage}
		\hspace{0.3cm}
		\begin{minipage}{0.45\textwidth}
			\centering
			\begin{tikzpicture}[scale=0.55,inner sep=0pt,minimum size=2.5mm,thick]
				\draw[dashed] (0,-2)--(4,-2);
				\draw[dashed] (2,0)--(2,-4);
				\node at (4,-2) [circle,draw=green!70,fill=green!40]{};
				\node at (2,-2) [circle,draw=blue!70,fill=blue!40] {};
				\node at (2,0) [circle,draw=green!70,fill=green!40] {};
				\node at (0,-2) [circle,draw=blue!70,fill=blue!40] {};
				\node at (2,-4) [rectangle,draw=green!70,fill=green!40] {};

				\node at (3,-1.4) {$U_{21}^{n+1}$};
				\node at (2,0.8) {$U_{22}^{n}$};
				\node at (2,-4.8) {$U_{24}^{n}$};
				\node at (-1.1,-2) {$U_{11}^{n+1}$};
				\node at (5,-2) {$U_{31}^{n}$};

			\end{tikzpicture}
		\end{minipage}
		\vspace{0.5cm}

		\begin{minipage}{0.45\textwidth}
			\centering
			\begin{tikzpicture}[scale=0.55,inner sep=0pt,minimum size=2.5mm,thick]
				\draw[dashed] (0,-2)--(4,-2);
				\draw[dashed] (2,0)--(2,-4);
				\node at (4,-2) [circle,draw=green!70,fill=green!40]{};
				\node at (2,-2) [circle,draw=blue!70,fill=blue!40] {};
				\node at (2,0) [circle,draw=green!70,fill=green!40] {};
				\node at (0,-2) [circle,draw=blue!70,fill=blue!40] {};
				\node at (2,-4) [circle,draw=blue!70,fill=blue!40] {};

				\node at (3,-1.4) {$U_{33}^{n+1}$};
				\node at (2,0.8) {$U_{34}^{n}$};
				\node at (2,-4.8) {$U_{32}^{n+1}$};
				\node at (-1/1,-2) {$U_{23}^{n+1}$};
				\node at (5,-2) {$U_{43}^{n}$};

			\end{tikzpicture}
		\end{minipage}
		\hspace{0.3cm}
		\begin{minipage}{0.45\textwidth}
			\centering
			\begin{tikzpicture}[scale=0.55,inner sep=0pt,minimum size=2.5mm,thick]
				\draw[dashed] (0,-2)--(4,-2);
				\draw[dashed] (2,0)--(2,-4);
				\node at (4,-2) [rectangle,draw=blue!70,fill=blue!40] {};
				\node at (2,-2) [circle,draw=blue!70,fill=blue!40] {};
				\node at (2,0) [circle,draw=green!70,fill=green!40] {};
				\node at (0,-2) [circle,draw=blue!70,fill=blue!40] {};
				\node at (2,-4) [circle,draw=blue!70,fill=blue!40] {};

				\node at (3,-1.4) {$U_{43}^{n+1}$};
				\node at (2,0.8) {$U_{44}^{n}$};
				\node at (2,-4.8) {$U_{42}^{n+1}$};
				\node at (-1/1,-2) {$U_{33}^{n+1}$};
				\node at (5.2,-2) {$U_{13}^{n+1}$};

			\end{tikzpicture}
		\end{minipage}

	\end{minipage}

	\caption{A general update order on a $4 \times 4$ grid (left). The arrows indicate the update order for each grid point. Four representative grid points are selected to illustrate the detailed implementation, with their corresponding stencils shown on the right. Green nodes represent known values at time level $n$, while blue nodes represent unknown values at time level $n+1$. Rectangular nodes incorporate the periodic boundary condition.}
	\label{fig:ex}
\end{figure}

Figure~\ref{fig:ex} shows a general grid point update order. To demonstrate that the corresponding scheme \eqref{scheme1} can indeed perform pointwise explicit computation, we selected four typical grid points to show the solution process: $(1,1)$ represents a corner point, $(2,1)$ and $(4,3)$ are boundary points, and $(3,3)$ is an interior point. We only show the update process for the variable $U_{j,k}$ at these four points, as the solution process for $\Psi_{j,k}$ is analogous. The following steps outline the explicit computation process at each of these grid points:

\begin{itemize}
	\item[1.] {\bf Compute $U_{11}^{n+1}$}:  
   To compute $U_{11}^{n+1}$, we require the values of the two boundary points, $U_{41}$ and $U_{14}$. However, these boundary points are updated after $U_{11}$, so we only need to use the values from the previous time step. In the scheme \eqref{scheme1}, both $W$ and $S$ are set to 0, and similarly, $N$ and $E$ are also 0. This setup allows $U_{11}^{n+1}$ to be computed explicitly.

	\item[2.] {\bf  Compute $U_{21}^{n+1}$:} 
   The update for $U_{21}^{n+1}$ differs from that of $U_{11}^{n+1}$ in that, in addition to using the boundary points from the known layer, it also involves $U_{11}^{n+1}$. Since $U_{11}^{n+1}$ has already been computed, $U_{21}^{n+1}$ can be computed explicitly as well.

	\item[3.] {\bf  Compute $U_{33}^{n+1}$:} 
   Updating $U_{33}^{n+1}$ involves the unknown values of $U_{23}$ and $U_{32}$. However, the update order places these values before $U_{33}$, allowing them to be precomputed. As a result, $U_{33}^{n+1}$ can also be computed explicitly in a similar manner.

	\item[4.] {\bf  Compute $U_{43}^{n+1}$:} 
   For the boundary point $U_{43}^{n+1}$, the update requires three neighboring unknown values, including $U_{13}^{n+1}$ under periodic boundary conditions. The update order ensures that all necessary values are precomputed, enabling $U_{43}^{n+1}$ to be computed explicitly.

\end{itemize}

From this general example, it can be seen that the values at the four neighboring grid points to $(j,k)$ involved on the right-hand side of equation \eqref{scheme1} are either already known at time level $n$ or have been precomputed, as long as the update order is consistent with the grid point ordering. Therefore, the scheme \eqref{scheme1} can indeed be computed explicitly point-by-point. Similarly, for its adjoint scheme given by
\begin{equation}\label{scheme1-adj}
	\left\{\begin{aligned}
		 & (i-\frac{\tau\kappa_1}{h^2}+\frac{\tau\gamma}{2}U_{j,k}^{n+1}) \Psi_{j,k}^{n+1}=(i+\frac{\tau\kappa_1}{h^2}-\frac{\tau\gamma}{2} U_{j,k}^n)\Psi_{j,k}^{n}                                                           \\
		 & \qquad\qquad-\frac{\tau\kappa_1}{2h^2}(\Psi_{j,k-1}^{n+\overline{\text{S}}}+\Psi_{j-1, k}^{n+\overline{\text{W}}} + \Psi_{j+1, k}^{n+\overline{\text{E}}} + \Psi_{j, k+1}^{n+\overline{\text{N}}}),                 \\
		 & U_{j,k}^{n+1}-\frac{\tau}{2} V_{j,k}^{n+1}=U_{j,k}^n+\frac{\tau}{2} V_{j,k}^{n},                                                                                                                                    \\
		 & (\frac{2\tau\kappa_2}{h^2}+\frac{\tau\mu^2}{2} )U_{j,k}^{n+1}+V_{j,k}^{n+1}=V_{j,k}^n-(\frac{2\tau\kappa_2}{h^2}+\frac{\tau\mu^2}{2}) U_{j,k}^{n}                                                                   \\
		 & \qquad\qquad+\frac{\tau\kappa_2}{h^2}(U_{j,k-1}^{n+\overline{\text{S}}}+U_{j-1, k}^{n+\overline{\text{W}}}+ U_{j+1, k}^{n+\overline{\text{E}}} + U_{j, k+1}^{n+\overline{\text{N}}})+\tau\gamma |\Psi_{j,k}^{n}|^2,
	\end{aligned}\right.j, k = 1, \cdots, N,
\end{equation}
the same update strategy can be applied to achieve explicit point-by-point computation. In this case, the update order is the reverse of that in scheme \eqref{scheme1}. We need to first solve for $U_{j,k}^{n+1}$ using the last two equations, and then substitute these values into the first equation to find $\Psi_{j,k}^{n+1}$. Again, the values at the four neighboring points involved on the right-hand side are either known or precomputed, as long as the update order is consistent with the grid point ordering.

By combining scheme \eqref{scheme1} with its adjoint scheme \eqref{scheme1-adj}, we obtain the second-order symmetric and energy-preserving DP-AVF2 scheme, which can also be computed explicitly point-by-point. The specific implementation process of the DP-AVF2 scheme is provided in the following {\bf Algorithm~1}.
\begin{table}[H]\label{tab:algo}
	\fontsize{10pt}{12pt}\selectfont
	\renewcommand\arraystretch{1}
	\centering
	\begin{tabular*}{0.9\textwidth}[h]{@{\extracolsep{\fill}}l} \toprule[2pt]
		{\bf Algorithm 1} Implementation of the DP-AVF2 Scheme in 2D  \\\hline
		{\small 1.} {\bf Input:} Initial conditions, update order and parameters  \\
		{\small 2.} {\bf Output:} $\Psi^{n+1}, U^{n+1}, V^{n+1}$\\\hline
		{\small 3.} Compute the inverse of the $2\times 2$ coefficient matrix in the last two equations of scheme \eqref{scheme1}\\
		{\small 4.} {\bf for} $n=1,\cdots,T$\\
		{\small 5.} \qquad {\bf for} $j=\widetilde{1},\cdots,\widetilde{N}$\\
		{\small 6.} \qquad\qquad {\bf for} $k=\widehat{1},\cdots,\widehat{N}$\\
		{\small 7.} \qquad\qquad\qquad Solve the first equation of scheme \eqref{scheme1} for ${\Psi}_{j,k}^{n+1}$ explicitly with $\frac{\tau}{2}$\\
		{\small 8.} \qquad\qquad\qquad Solve the last two equations of scheme \eqref{scheme1} for ${U}_{j,k}^{n+1}$, ${V}_{j,k}^{n+1}$ explicitly with $\frac{\tau}{2}$\\
		{\small 9.} \qquad\qquad {\bf end} \\
		{\small 10.} \qquad {\bf end}\\
		{\small 11.} \qquad {\bf for} $j=\widetilde{N},\cdots,\widetilde{1}$ \\
		{\small 12.} \qquad\qquad {\bf for} $k=\widehat{N},\cdots,\widehat{1}$\\
		{\small 13.} \qquad\qquad\qquad Solve the last two equations of scheme \eqref{scheme1-adj} for $U_{j,k}^{n+1}$, $V_{j,k}^{n+1}$ explicitly with $\frac{\tau}{2}$ \\
		{\small 14.} \qquad\qquad\qquad Solve the first equation of scheme \eqref{scheme1-adj} for $\Psi_{j,k}^{n+1}$  explicitly with $\frac{\tau}{2}$\\
		{\small 15.} \qquad\qquad {\bf end} \\
		{\small 16.} \qquad {\bf end}\\
		{\small 17.} {\bf end}\\
		\bottomrule[2pt]
	\end{tabular*}
\end{table}

Once the update order is determined (e.g., $ j = \widetilde{1}, \cdots, \widetilde{N} $ and $ k = \widehat{1}, \cdots, \widehat{N} $, where the ordering of $ j $ and $ k $ can be independent), only $ N^2 $ iterations are required, with each iteration explicitly updating the three variables at the $(j, k)$ grid point for the $n+1$ layer. Notably, no additional memory is needed to store the updated values, as the values from the $n$-th layer can be overwritten directly. This ensures that when updating later grid points, any required values from earlier points are taken from the already updated $n+1$ layer. This approach simplifies programming and conserves memory. For the adjoint scheme, we simply reverse the order of $j$ and $k$ and swap the sequence of solving the equations. Similarly, only $N^2$ iterations of the for-loop are required to achieve explicit pointwise updates for all grid points using the same memory storage.

\subsection{CPU and GPU parallel computation}

As described above, given any update order of grid points, even if randomly generated, the DP-AVF2 scheme can perform explicit pointwise computation with a computational complexity of $\mathcal{O}(N^2)$. In this section, we will further exploit the locality of the spatial central difference operator and demonstrate how the DP-AVF2 scheme can achieve efficient parallel computation on both CPUs and GPUs by carefully selecting the update strategy.

For the basic scheme \eqref{scheme1} and its adjoint scheme \eqref{scheme1-adj} used in the DP-AVF2 method, Figure~\ref{cpu-2d} illustrates an example of a parallel update order that can be implemented on CPUs. It is important to note that any parallel update order is essentially derived from a specific serial update order. Taking the parallelization of the base method shown in the left panel as an example, the direction of the arrows defines a clear serial update order. Solving the problem following this order still allows for explicit pointwise computation. However, we observe that while the grid points in the right blue block must be updated after those in the left blue block. The locality of the central difference scheme ensures that updates within each region only involve the $n$-th layer values at the neighboring grid points within the two adjacent green regions. As a result, in practice, the two blue blocks can be updated simultaneously in two separate processes, $\mathcal{P}_1$ and $\mathcal{P}_2$. Similarly, the adjoint scheme can also be parallelized on CPUs, with the key difference being that the update order for the grid points in the red and green blocks is reversed compared to the base method.

\begin{figure}[H]
	\centering
	\begin{minipage}{0.35\textwidth}

		\begin{tikzpicture}[scale=0.5,inner sep=0pt,minimum size=3mm,thick]

			\draw[solid] (0,0) -- (12,0)--(12,-6)--(0,-6)--cycle;

			\foreach \x in {-4,-2}
			\draw[dotted] (0,\x)--(12,\x);

			\foreach \x in {2,4,6,8,10,12}
			\draw[dotted] (\x,0)--(\x,-6);

			\foreach \x in {1,2,3,4}
			\foreach \y in {1,2,3,4,5,6,7}
			\node(\x\y) at (2*\y-2,2*\x-8) [circle,draw=blue!50,fill=blue!20] {};

			\path[->,dashed] (11) edge (21);
			\path[->,dashed] (21) edge (31);
			\path[->,dashed] (31) edge (41);
			\path[->,dashed] (41) edge (12);

			\path[->,dashed] (12) edge (22);
			\path[->,dashed] (22) edge (32);
			\path[->,dashed] (32) edge (42);
			\path[->,dashed] (42) edge (13);
			\path[->,dashed] (13) edge (23);
			\path[->,dashed] (23) edge (33);
			\path[->,dashed] (33) edge (43);

			\path[<-,dashed] (15) edge (25);
			\path[<-,dashed] (25) edge (35);
			\path[<-,dashed] (35) edge (45);
			\path[->,dashed] (15) edge (46);
			\path[<-,dashed] (16) edge (26);
			\path[<-,dashed] (26) edge (36);
			\path[<-,dashed] (36) edge (46);

			\path[->,dashed] (14) edge (24);
			\path[->,dashed] (24) edge (34);
			\path[->,dashed] (34) edge (44);

			\path[->,dashed] (16) edge (47);
			\path[<-,dashed] (17) edge (27);
			\path[<-,dashed] (27) edge (37);
			\path[<-,dashed] (37) edge (47);

			\draw [<-,dashed] (14) .. controls (7,-7.5) and (11,-7.5).. (17);
			\draw [->,dashed] (43) .. controls (4.55,1.5) and (7.5,1.5).. (45);

			\filldraw[fill=blue!20!white, draw=blue!70!black,very thin,opacity=0.5]
			(1.5,-6.5) rectangle (4.5,0.5);
			\filldraw[fill=blue!20!white, draw=blue!70!black,very thin,opacity=0.5]
			(7.5,-6.5) rectangle (10.5,0.5);

			\filldraw[fill=magenta!20!white, draw=red!70!black,very thin,opacity=0.5]
			(-0.5,-6.5) rectangle (0.5,0.5);
			\filldraw[fill=green!20!white, draw=green!70!black,very thin,opacity=0.5]
			(5.5,-6.5) rectangle (6.5,0.5);
			\filldraw[fill=green!20!white, draw=green!70!black,very thin,opacity=0.5]
			(11.5,-6.5) rectangle (12.5,0.5);
			\node at (3,1.5) {$\mathcal{P}_1$};
			\node at (9,1.5) {$\mathcal{P}_2$};
		\end{tikzpicture}
	\end{minipage}
	\hspace{1.5cm}
	\begin{minipage}{0.35\textwidth}

		\begin{tikzpicture}[scale=0.5,inner sep=0pt,minimum size=3mm,thick]

			\draw[solid] (0,0) -- (12,0)--(12,-6)--(0,-6)--cycle;

			\foreach \x in {-4,-2}
			\draw[dotted] (0,\x)--(12,\x);

			\foreach \x in {2,4,6,8,10,12}
			\draw[dotted] (\x,0)--(\x,-6);

			\foreach \x in {1,2,3,4}
			\foreach \y in {1,2,3,4,5,6,7}
			\node(\x\y) at (2*\y-2,2*\x-8) [circle,draw=blue!50,fill=blue!20] {};

			\path[<-,dashed] (11) edge (21);
			\path[<-,dashed] (21) edge (31);
			\path[<-,dashed] (31) edge (41);
			\path[<-,dashed] (41) edge (12);

			\path[<-,dashed] (12) edge (22);
			\path[<-,dashed] (22) edge (32);
			\path[<-,dashed] (32) edge (42);
			\path[<-,dashed] (42) edge (13);
			\path[<-,dashed] (13) edge (23);
			\path[<-,dashed] (23) edge (33);
			\path[<-,dashed] (33) edge (43);

			\path[->,dashed] (15) edge (25);
			\path[->,dashed] (25) edge (35);
			\path[->,dashed] (35) edge (45);
			\path[<-,dashed] (15) edge (46);
			\path[->,dashed] (16) edge (26);
			\path[->,dashed] (26) edge (36);
			\path[->,dashed] (36) edge (46);

			\path[<-,dashed] (14) edge (24);
			\path[<-,dashed] (24) edge (34);
			\path[<-,dashed] (34) edge (44);

			\path[<-,dashed] (16) edge (47);
			\path[->,dashed] (17) edge (27);
			\path[->,dashed] (27) edge (37);
			\path[->,dashed] (37) edge (47);

			\draw [->,dashed] (14) .. controls (7,-7.5) and (11,-7.5).. (17);
			\draw [<-,dashed] (43) .. controls (4.55,1.5) and (7.5,1.5).. (45);

			\filldraw[fill=blue!20!white, draw=blue!70!black,very thin,opacity=0.5]
			(1.5,-6.5) rectangle (4.5,0.5);
			\filldraw[fill=blue!20!white, draw=blue!70!black,very thin,opacity=0.5]
			(7.5,-6.5) rectangle (10.5,0.5);

			\filldraw[fill=green!20!white, draw=green!70!black,very thin,opacity=0.5]
			(-0.5,-6.5) rectangle (0.5,0.5);
			\filldraw[fill=magenta!20!white, draw=red!70!black,very thin,opacity=0.5]
			(5.5,-6.5) rectangle (6.5,0.5);
			\filldraw[fill=magenta!20!white, draw=red!70!black,very thin,opacity=0.5]
			(11.5,-6.5) rectangle (12.5,0.5);
			\node at (3,1.5) {$\mathcal{P}_1$};
			\node at (9,1.5) {$\mathcal{P}_2$};
		\end{tikzpicture}
	\end{minipage}

	\caption{Parallelizable update order of the base method (left) and adjoint method (right) for the KGS equations in 2D.}\label{cpu-2d}
\end{figure}
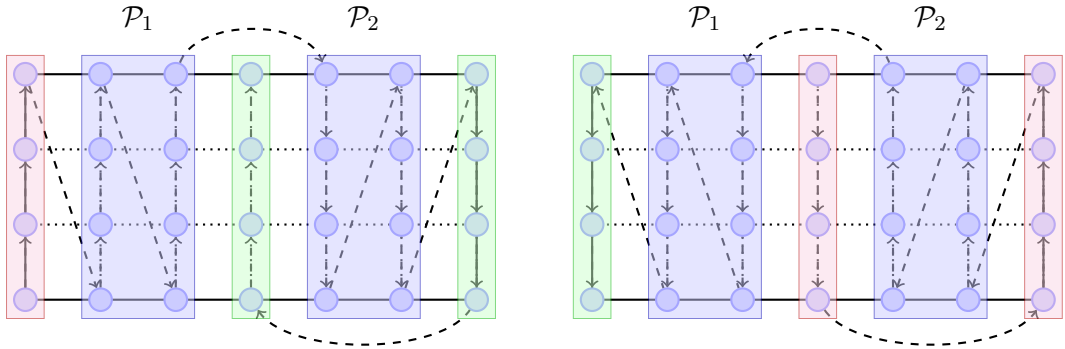

This parallelization strategy can be similarly extended to 3D grids, as illustrated in Figure~\ref{cpu-3d}. In the figure, we have omitted the arrows indicating the update order of the grid points and focused only on the core parallel block structure. It is clear that, following a similar update order to the 2D case, the grid points within the blue regions in Figure~\ref{cpu-3d} can be updated concurrently across different processes. Moreover, as long as the overall update order between blocks is maintained, the update order within each individual block can be arbitrary. This does not disrupt the parallelization strategy and, in fact, enhances the flexibility of the scheme.

\begin{figure}[H]
	\centering
	\begin{minipage}{0.45\textwidth}
			\begin{tikzpicture}[scale=0.52,inner sep=0pt,minimum size=2mm,thick]

				\foreach \y in {-2,-4} {
						\foreach \z in {0,1.2} {
								\draw[dotted] (1,\y,\z) -- (13,\y,\z);
							}
					}
				\foreach \y in {-2,-4,-6} {
						\foreach \z in {1.2} {
								\draw[dotted] (1,\y,\z) -- (13,\y,\z);
							}
					}
				\foreach \y in {0,-6} {
						\foreach \z in {2.4} {
								\draw[solid,thin] (1,\y,\z) -- (13,\y,\z);
							}
					}
				\foreach \y in {0} {
						\foreach \z in {0} {
								\draw[solid,thin] (1,\y,\z) -- (13,\y,\z);
							}
					}
				\foreach \y in {-2,-4} {
						\foreach \z in {2.4} {
								\draw[solid,thin] (1,\y,\z) -- (13,\y,\z);
							}
					}
				\foreach \y in {0} {
						\foreach \z in {1.2} {
								\draw[solid,thin] (1,\y,\z) -- (13,\y,\z);
							}
					}
				\foreach \y in {-6} {
						\foreach \z in {0} {
								\draw[dotted] (1,\y,\z) -- (13,\y,\z);
							}
					}
				\foreach \x in {3,5,7,9,11} {
						\foreach \z in {0,1.2,2.4} {
								\draw[dotted] (\x,0,\z) -- (\x,-6,\z);
							}
					}
				\foreach \x in {1,3,5,7,9,11,13} {
						\foreach \z in {1.2} {
								\draw[dotted] (\x,0,\z) -- (\x,-6,\z);
							}
					}
				\foreach \x in {1,13} {
						\foreach \z in {0,2.4} {
								\draw[solid,thin] (\x,0,\z) -- (\x,-6,\z);
							}
					}
				\foreach \x in {3,5,7,9,11} {
						\foreach \z in {2.4} {
								\draw[solid,thin] (\x,0,\z) -- (\x,-6,\z);
							}
					}
				\foreach \x in {13} {
						\foreach \z in {1.2} {
								\draw[solid,thin] (\x,0,\z) -- (\x,-6,\z);
							}
					}
				\foreach \x in {3,5,7,9,11} {
						\foreach \y in {0,-2,-4,-6} {
								\draw[dotted] (\x,\y,0) -- (\x,\y,2.4);
							}
					}
				\foreach \x in {1,3,5,7,9,11,13} {
						\foreach \y in {-2,-4} {
								\draw[dotted] (\x,\y,0) -- (\x,\y,2.4);
							}
					}
				\foreach \x in {1,13} {
						\foreach \y in {0,-6} {
								\draw[solid,thin] (\x,\y,0) -- (\x,\y,2.4);
							}
					}
				\foreach \x in {13} {
						\foreach \y in {-2,-4} {
								\draw[solid,thin] (\x,\y,0) -- (\x,\y,2.4);
							}
					}
				\foreach \x in {3,5,7,9,11} {
						\foreach \y in {0} {
								\draw[solid,thin] (\x,\y,0) -- (\x,\y,2.4);
							}
					}
				\foreach \x in {1,3,5,7,9,11,13} {
						\foreach \y in {0,-2,-4,-6} {
								\foreach \z in {0,1.2,2.4} {
										\node[circle,draw=blue!50,fill=blue!20] at (\x,\y,\z) {};
									}
							}
					}
				\draw[fill=blue!20!white, draw=blue!70!black,very thin,opacity=0.5] (1.8,-7.5,0) -- (4.4,-7.5,0) -- (4.4,-0.48,0) -- (1.8,-0.48,0) -- cycle ;
				\draw[fill=blue!20!white, draw=blue!70!black,very thin,opacity=0.5] (4.4,-7.5,0) -- (4.4,-7.5,-3) -- (4.4,-0.48,-3) -- (4.4,-0.48,0) -- cycle ;
				\draw[fill=blue!20!white, draw=blue!70!black,very thin,opacity=0.5] (1.8,-0.48,0) -- (4.4,-0.48,0) -- (4.4,-0.48,-3) -- (1.8,-0.48,-3) -- cycle ;
				\draw[fill=blue!20!white, draw=blue!70!black,very thin,opacity=0.5] (7.8,-7.5,0) -- (10.4,-7.5,0) -- (10.4,-0.48,0) -- (7.8,-0.48,0) -- cycle ;
				\draw[fill=blue!20!white, draw=blue!70!black,very thin,opacity=0.5] (10.4,-7.5,0) -- (10.4,-7.5,-3) -- (10.4,-0.48,-3) -- (10.4,-0.48,0) -- cycle ;
				\draw[fill=blue!20!white, draw=blue!70!black,very thin,opacity=0.5] (7.8,-0.48,0) -- (10.4,-0.48,0) -- (10.4,-0.48,-3) -- (7.8,-0.48,-3) -- cycle ;
				\draw[fill=magenta!20!white, draw=red!70!black,very thin,opacity=0.5] (-0.2,-7.5,0) -- (0.4,-7.5,0) -- (0.4,-0.48,0) -- (-0.2,-0.48,0) -- cycle ;
				\draw[fill=magenta!20!white, draw=red!70!black,very thin,opacity=0.5] (0.4,-7.5,0) -- (0.4,-7.5,-3) -- (0.4,-0.48,-3) -- (0.4,-0.48,0) -- cycle ;
				\draw[fill=magenta!20!white, draw=red!70!black,very thin,opacity=0.5] (-0.2,-0.48,0) -- (0.4,-0.48,0) -- (0.4,-0.48,-3) -- (-0.4,-0.48,-3) -- cycle ;

				\draw[fill=green!20!white, draw=green!70!black,very thin,opacity=0.5] (5.8,-7.5,0) -- (6.4,-7.5,0) -- (6.4,-0.48,0) -- (5.8,-0.48,0) -- cycle ;
				\draw[fill=green!20!white, draw=green!70!black,very thin,opacity=0.5] (6.4,-7.5,0) -- (6.4,-7.5,-3) -- (6.4,-0.48,-3) -- (6.4,-0.48,0) -- cycle ;
				\draw[fill=green!20!white, draw=green!70!black,very thin,opacity=0.5] (5.8,-0.48,0) -- (6.4,-0.48,0) -- (6.4,-0.48,-3) -- (5.8,-0.48,-3) -- cycle ;
				\draw[fill=green!20!white, draw=green!70!black,very thin,opacity=0.5] (11.8,-7.5,0) -- (12.4,-7.5,0) -- (12.4,-0.48,0) -- (11.8,-0.48,0) -- cycle ;
				\draw[fill=green!20!white, draw=green!70!black,very thin,opacity=0.5] (12.4,-7.5,0) -- (12.4,-7.5,-3) -- (12.4,-0.48,-3) -- (12.4,-0.48,0) -- cycle ;
				\draw[fill=green!20!white, draw=green!70!black,very thin,opacity=0.5] (11.8,-0.48,0) -- (12.4,-0.48,0) -- (12.4,-0.48,-3) -- (11.8,-0.48,-3) -- cycle ;

	\draw[->,dashed] (0,-0.48,-1.5) .. controls (0.5,0.8,-1.5) and (2.5,0.8,-1.5) .. (2.9,-0.48,-1.5)  ;	
	\draw[->,dashed] (3.1,-0.48,-1.5) .. controls (3.5,1.5,-1.5) and (8.5,1.5,-1.5) .. (8.9,-0.48,-1.5)  ;	
	\draw[->,dashed] (9.1,-0.48,-1.5)  .. controls (9.5,0.8,-1.5) and (11.5,0.8,-1.5) .. (12,-0.48,-1.5)  ;	
	\draw[->,dashed] (12,-7.5,-1.5)  .. controls (11.5,-9.5,-1.5) and (6.5,-9.5,-1.5) .. (6,-7.5,-1.5)  ;
	
	\node at (3,1,-1.5)  {$\mathcal{P}_1$};
	\node at (9,1,-1.5) {$\mathcal{P}_2$};

			\end{tikzpicture}
	\end{minipage}
	\begin{minipage}{0.45\textwidth}
			\begin{tikzpicture}[scale=0.52,inner sep=0pt,minimum size=2mm,thick]

				\foreach \y in {-2,-4} {
						\foreach \z in {0,1.2} {
								\draw[dotted] (1,\y,\z) -- (13,\y,\z);
							}
					}
				\foreach \y in {-2,-4,-6} {
						\foreach \z in {1.2} {
								\draw[dotted] (1,\y,\z) -- (13,\y,\z);
							}
					}
				\foreach \y in {0,-6} {
						\foreach \z in {2.4} {
								\draw[solid,thin] (1,\y,\z) -- (13,\y,\z);
							}
					}
				\foreach \y in {0} {
						\foreach \z in {0} {
								\draw[solid,thin] (1,\y,\z) -- (13,\y,\z);
							}
					}
				\foreach \y in {-2,-4} {
						\foreach \z in {2.4} {
								\draw[solid,thin] (1,\y,\z) -- (13,\y,\z);
							}
					}
				\foreach \y in {0} {
						\foreach \z in {1.2} {
								\draw[solid,thin] (1,\y,\z) -- (13,\y,\z);
							}
					}	
				\foreach \y in {-6} {
						\foreach \z in {0} {
								\draw[dotted] (1,\y,\z) -- (13,\y,\z);
							}
					}
				\foreach \x in {3,5,7,9,11} {
						\foreach \z in {0,1.2,2.4} {
								\draw[dotted] (\x,0,\z) -- (\x,-6,\z);
							}
					}
				\foreach \x in {1,3,5,7,9,11,13} {
						\foreach \z in {1.2} {
								\draw[dotted] (\x,0,\z) -- (\x,-6,\z);
							}
					}
				\foreach \x in {1,13} {
						\foreach \z in {0,2.4} {
								\draw[solid,thin] (\x,0,\z) -- (\x,-6,\z);
							}
					}
				\foreach \x in {3,5,7,9,11} {
						\foreach \z in {2.4} {
								\draw[solid,thin] (\x,0,\z) -- (\x,-6,\z);
							}
					}
				\foreach \x in {13} {
						\foreach \z in {1.2} {
								\draw[solid,thin] (\x,0,\z) -- (\x,-6,\z);
							}
					}
				\foreach \x in {3,5,7,9,11} {
						\foreach \y in {0,-2,-4,-6} {
								\draw[dotted] (\x,\y,0) -- (\x,\y,2.4);
							}
					}
				\foreach \x in {1,3,5,7,9,11,13} {
						\foreach \y in {-2,-4} {
								\draw[dotted] (\x,\y,0) -- (\x,\y,2.4);
							}
					}
				\foreach \x in {1,13} {
						\foreach \y in {0,-6} {
								\draw[solid,thin] (\x,\y,0) -- (\x,\y,2.4);
							}
					}
				\foreach \x in {13} {
						\foreach \y in {-2,-4} {
								\draw[solid,thin] (\x,\y,0) -- (\x,\y,2.4);
							}
					}
				\foreach \x in {3,5,7,9,11} {
						\foreach \y in {0} {
								\draw[solid,thin] (\x,\y,0) -- (\x,\y,2.4);
							}
					}
				\foreach \x in {1,3,5,7,9,11,13} {
						\foreach \y in {0,-2,-4,-6} {
								\foreach \z in {0,1.2,2.4} {
										\node[circle,draw=blue!50,fill=blue!20] at (\x,\y,\z) {};
									}
							}
					}
				\draw[fill=blue!20!white, draw=blue!70!black,very thin,opacity=0.5] (1.8,-7.5,0) -- (4.4,-7.5,0) -- (4.4,-0.48,0) -- (1.8,-0.48,0) -- cycle ;
				\draw[fill=blue!20!white, draw=blue!70!black,very thin,opacity=0.5] (4.4,-7.5,0) -- (4.4,-7.5,-3) -- (4.4,-0.48,-3) -- (4.4,-0.48,0) -- cycle ;
				\draw[fill=blue!20!white, draw=blue!70!black,very thin,opacity=0.5] (1.8,-0.48,0) -- (4.4,-0.48,0) -- (4.4,-0.48,-3) -- (1.8,-0.48,-3) -- cycle ;
				\draw[fill=blue!20!white, draw=blue!70!black,very thin,opacity=0.5] (7.8,-7.5,0) -- (10.4,-7.5,0) -- (10.4,-0.48,0) -- (7.8,-0.48,0) -- cycle ;
				\draw[fill=blue!20!white, draw=blue!70!black,very thin,opacity=0.5] (10.4,-7.5,0) -- (10.4,-7.5,-3) -- (10.4,-0.48,-3) -- (10.4,-0.48,0) -- cycle ;
				\draw[fill=blue!20!white, draw=blue!70!black,very thin,opacity=0.5] (7.8,-0.48,0) -- (10.4,-0.48,0) -- (10.4,-0.48,-3) -- (7.8,-0.48,-3) -- cycle ;
				\draw[fill=green!20!white, draw=green!70!black,very thin,opacity=0.5] (-0.2,-7.5,0) -- (0.4,-7.5,0) -- (0.4,-0.48,0) -- (-0.2,-0.48,0) -- cycle ;
				\draw[fill=green!20!white, draw=green!70!black,very thin,opacity=0.5] (0.4,-7.5,0) -- (0.4,-7.5,-3) -- (0.4,-0.48,-3) -- (0.4,-0.48,0) -- cycle ;
				\draw[fill=green!20!white, draw=green!70!black,very thin,opacity=0.5] (-0.2,-0.48,0) -- (0.4,-0.48,0) -- (0.4,-0.48,-3) -- (-0.4,-0.48,-3) -- cycle ;

				\draw[fill=magenta!20!white, draw=red!70!black,very thin,opacity=0.5] (5.8,-7.5,0) -- (6.4,-7.5,0) -- (6.4,-0.48,0) -- (5.8,-0.48,0) -- cycle ;
				\draw[fill=magenta!20!white, draw=red!70!black,very thin,opacity=0.5] (6.4,-7.5,0) -- (6.4,-7.5,-3) -- (6.4,-0.48,-3) -- (6.4,-0.48,0) -- cycle ;
				\draw[fill=magenta!20!white, draw=red!70!black,very thin,opacity=0.5] (5.8,-0.48,0) -- (6.4,-0.48,0) -- (6.4,-0.48,-3) -- (5.8,-0.48,-3) -- cycle ;
				\draw[fill=magenta!20!white, draw=red!70!black,very thin,opacity=0.5] (11.8,-7.5,0) -- (12.4,-7.5,0) -- (12.4,-0.48,0) -- (11.8,-0.48,0) -- cycle ;
				\draw[fill=magenta!20!white, draw=red!70!black,very thin,opacity=0.5] (12.4,-7.5,0) -- (12.4,-7.5,-3) -- (12.4,-0.48,-3) -- (12.4,-0.48,0) -- cycle ;
				\draw[fill=magenta!20!white, draw=red!70!black,very thin,opacity=0.5] (11.8,-0.48,0) -- (12.4,-0.48,0) -- (12.4,-0.48,-3) -- (11.8,-0.48,-3) -- cycle ;

	\draw[<-,dashed] (0,-0.48,-1.5) .. controls (0.5,0.8,-1.5) and (2.5,0.8,-1.5) .. (2.9,-0.48,-1.5)  ;	
	\draw[<-,dashed] (3.1,-0.48,-1.5) .. controls (3.5,1.5,-1.5) and (8.5,1.5,-1.5) .. (8.9,-0.48,-1.5)  ;	
	\draw[<-,dashed] (9.1,-0.48,-1.5)  .. controls (9.5,0.8,-1.5) and (11.5,0.8,-1.5) .. (12,-0.48,-1.5)  ;	
	\draw[<-,dashed] (12,-7.5,-1.5)  .. controls (11.5,-9.5,-1.5) and (6.5,-9.5,-1.5) .. (6,-7.5,-1.5)  ;
	
		\node at (3,1,-1.5)  {$\mathcal{P}_1$};
	\node at (9,1,-1.5) {$\mathcal{P}_2$};
			\end{tikzpicture}
	\end{minipage}

	\caption{Parallelizable update order of the base method (left) and adjoint method (right) for the KGS equations in 3D.}\label{cpu-3d}
\end{figure}
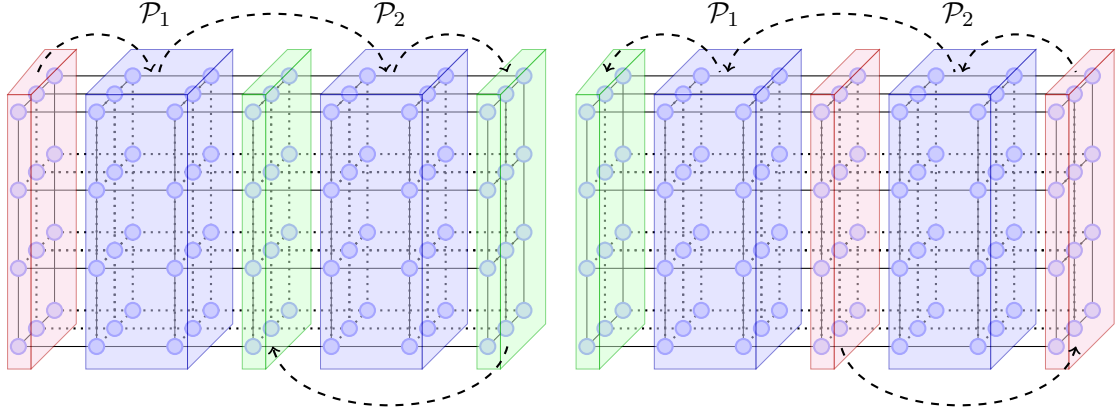

{\bf Algorithm 2} provides the specific implementation flow for the CPU parallel strategy on the 2D grid shown in Figure~\ref{cpu-2d}, and a similar parallel implementation can be obtained for the 3D grid shown in Figure~\ref{cpu-3d}. In {\bf Algorithm 2}, we first define the number of processes required for parallel computation, denoted as $N_p$, and then divide the grid into $N_p$ square regions, as shown by the blue boxes in Figure~\ref{cpu-2d}. Each pair of adjacent square regions is separated by a column, as indicated by the green boxes in Figure~\ref{cpu-2d}. To ensure load balancing across the processes, we strive to make the number of grid points within each of the $N_p$ square regions as equal as possible. During the update process, we first update the grid points on the left boundary, as shown by the red box in Figure~\ref{cpu-2d}. Then, the grid points within the blue boxes are updated simultaneously across $N_p$ processes. Note that, for simplicity of notation, the update process in {\bf Algorithm 2} only rearranges the index $k$, while the index $j$ is sorted in ascending order as $ j_{l-1}, j_{l-1}+1, \cdots, j_l $. In practice, the index $j$ can also be rearranged, meaning that for each process, the update order of the grid points within each blue square can be fully arbitrary for both the $j$ and $k$ indices, without affecting the parallelism or the conservation properties. Similarly, for the adjoint scheme \eqref{scheme1-adj}, the calculation process can be derived from the base method \eqref{scheme1} by reversing the variable computation and grid indices. Therefore, the DP-AVF2 scheme not only enables explicit pointwise serial computation but also allows for easy extension to CPU parallelization.

\begin{table}[H]\label{tab:algo_cpu}
	\fontsize{10pt}{12pt}\selectfont
	\renewcommand\arraystretch{1}
	\centering
	\begin{tabular*}{0.9\textwidth}[h]{@{\extracolsep{\fill}}l} \toprule[2pt]
		{\bf Algorithm 2} CPU Parallel Implementation of the DP-AVF2 Scheme in 2D  \\\hline
		{\small 1.} {\bf Input:} Initial conditions, update order, parameters, number of threads/processors $N_p$  \\
		{\small 2.} {\bf Output:} $\Psi^{n+1}, U^{n+1}, V^{n+1}$\\\hline
		{\small 3.} Divide the spatial domain into $N_p$ subdomains (i.e., $N_p$ blue blocks as shown in Figure \ref{cpu-2d}). \\
		{\small 4.} Record the row indices of the red and green blocks as $j_0, j_1, \dots, j_{N_p}$. \\
		{\small 5.} Compute the inverse of the $2 \times 2$ coefficient matrix in the last two equations of scheme \eqref{scheme1}. \\
		{\small 6.} {\bf for} $n = 1, \dots, T$ \\
		{\small 7.~} \qquad {\bf for} $k = \widehat{1}, \dots, \widehat{N}$ \hfill // {\scriptsize\it Update left boundary} \\
		{\small 8.} \qquad \qquad Solve scheme \eqref{scheme1} explicitly with $\frac{\tau}{2}$ for $\Psi_{j_0,k}^{n+1}$, followed by $U_{j_0,k}^{n+1}$, $V_{j_0,k}^{n+1}$. \\
		{\small 9.} ~\qquad {\bf end} \\
		{\small 10.} \qquad {\bf Parallel for} $l = 1, \dots, N_p$ \hfill // {\scriptsize\it Update blocks} \\
		{\small 11.} \qquad \qquad {\bf for} $j = j_{l-1}, j_{l-1}+1, \dots, j_l$ \\
		{\small 12.} \qquad \qquad \qquad {\bf for} $k = \widehat{1}, \dots, \widehat{N}$ \\
		{\small 13.} \qquad \qquad \qquad \qquad Solve scheme \eqref{scheme1} explicitly with $\frac{\tau}{2}$ for $\Psi_{j,k}^{n+1}$, followed by $U_{j,k}^{n+1}$, $V_{j,k}^{n+1}$. \\
		{\small 14.} \qquad \qquad \qquad {\bf end} \\
		{\small 15.} \qquad \qquad {\bf end} \\
		{\small 16.} \qquad {\bf end} \\
		{\small 17.} \qquad {\bf for} $j = j_1, \dots, j_{N_p}$ \hfill // {\scriptsize\it Update slides} \\
		{\small 18.} \qquad \qquad {\bf for} $k = \widehat{1}, \dots, \widehat{N}$ \\
		{\small 19.} \qquad \qquad \qquad Solve scheme \eqref{scheme1} explicitly with $\frac{\tau}{2}$ for $\Psi_{j,k}^{n+1}$, followed by $U_{j,k}^{n+1}$, $V_{j,k}^{n+1}$. \\
		{\small 20.} \qquad \qquad {\bf end} \\
		{\small 21.} \qquad {\bf end} \\
		{\small 22.} \qquad {\bf for} $j = j_{N_p}, \dots, j_1$ \hfill // {\scriptsize\it Update slides in reverse order} \\
		{\small 23.} \qquad \qquad {\bf for} $k = \widehat{N}, \dots, \widehat{1}$ \\
		{\small 24.} \qquad \qquad \qquad Solve scheme \eqref{scheme1} explicitly with $\frac{\tau}{2}$ for $U_{j,k}^{n+1}$, $V_{j,k}^{n+1}$, followed by $\Psi_{j,k}^{n+1}$. \\
		{\small 25.} \qquad \quad {\bf end} \\
		{\small 26.} \qquad {\bf end} \\
		{\small 27.} \qquad {\bf Parallel for} $l = 1, \dots, N_p$ \hfill // {\scriptsize\it Update blocks in reverse order} \\
		{\small 28.} \qquad \qquad {\bf for} $j = j_l, j_l-1, \dots, j_{l-1}$ \\
		{\small 29.} \qquad \qquad \qquad {\bf for} $k = \widehat{N}, \dots, \widehat{1}$ \\
		{\small 30.} \qquad \qquad \qquad \qquad Solve scheme \eqref{scheme1} explicitly with $\frac{\tau}{2}$ for $U_{j,k}^{n+1}$, $V_{j,k}^{n+1}$, followed by $\Psi_{j,k}^{n+1}$. \\
		{\small 31.} \qquad \qquad \qquad {\bf end} \\
		{\small 32.} \qquad \qquad {\bf end} \\
		{\small 33.} \qquad {\bf end} \\
		{\small 34.} \qquad {\bf for} $k = \widehat{N}, \dots, \widehat{1}$  \hfill // {\scriptsize\it Update left boundary} \\
		{\small 35.} \qquad \qquad Solve scheme \eqref{scheme1} explicitly with $\frac{\tau}{2}$ for $U_{j_0,k}^{n+1}$, $V_{j_0,k}^{n+1}$, followed by $\Psi_{j_0,k}^{n+1}$. \\
		{\small 36.} \qquad {\bf end} \\
		{\small 37.} {\bf end} \\
		\bottomrule[2pt]
	\end{tabular*}
\end{table}

The CPU parallel strategy for the DP-AVF2 scheme exploits spatial locality, primarily along the $x$-direction, as shown in Figure~\ref{cpu-2d} and \ref{cpu-3d}. By considering locality in all directions, more efficient GPU parallel strategies can be developed. Figure~\ref{gpu-2d} illustrates a checkerboard grid of red and blue grid points, where each point interacts only with its four immediate neighbors. Moreover, since no red grid point depends on other red points, updates to red points can be performed in any order without affecting the results, and the same applies to the blue grid points. This independence allows for a sequential update strategy for the base method \eqref{scheme1}, where all red grid points are updated first, followed by all blue points. Both updates are executed in parallel on the GPU in any arbitrary order.

For the adjoint method, both serial and CPU parallel strategies require reversing the update order to maintain the symmetry and second-order accuracy of the DP-AVF2 scheme. However, in the GPU parallel strategy, since the order within the red or blue grid points does not impact the result, only the update order between red and blue points needs to be swapped, with both updates still performed in parallel on the GPU.

\begin{figure}[H]
	\centering
	\begin{minipage}{0.45\textwidth}
	\centering
		\begin{tikzpicture}[scale=0.75,inner sep=0pt,minimum size=3mm,thick]

			\draw[solid] (0,0) -- (6,0)--(6,-6)--(0,-6)--cycle;

			\foreach \x in {-6,-4,-2}
			\draw[dotted] (0,\x)--(6,\x);

			\foreach \x in {2,4,6}
			\draw[dotted] (\x,0)--(\x,-6);


			\node(10) at (0,-6) [circle,draw=blue!70,fill=blue!40] {};
			\node(11) at (2,-6) [circle,draw=red!70,fill=red!40] {};
			\node(12) at (4,-6) [circle,draw=blue!70,fill=blue!40] {};
			\node(13) at (6,-6) [circle,draw=red!70,fill=red!40] {};

			\node(20) at (0,-4) [circle,draw=red!70,fill=red!40] {};
			\node(21) at (2,-4) [circle,draw=blue!70,fill=blue!40] {};
			\node(22) at (4,-4) [circle,draw=red!70,fill=red!40] {};
			\node(23) at (6,-4) [circle,draw=blue!70,fill=blue!40] {};

			\node(30) at (0,-2) [circle,draw=blue!70,fill=blue!40] {};
			\node(31) at (2,-2) [circle,draw=red!70,fill=red!40] {};
			\node(32) at (4,-2) [circle,draw=blue!70,fill=blue!40] {};
			\node(33) at (6,-2) [circle,draw=red!70,fill=red!40] {};

			\node(40) at (0,0) [circle,draw=red!70,fill=red!40] {};
			\node(41) at (2,0) [circle,draw=blue!70,fill=blue!40] {};
			\node(42) at (4,0) [circle,draw=red!70,fill=red!40] {};
			\node(43) at (6,0) [circle,draw=blue!70,fill=blue!40] {};

		\end{tikzpicture}
	\end{minipage}
	\caption{2D checkerboard update order for GPU parallelization of the DP-AVF2 scheme for the KGS equations.}
	\label{gpu-2d}
\end{figure}
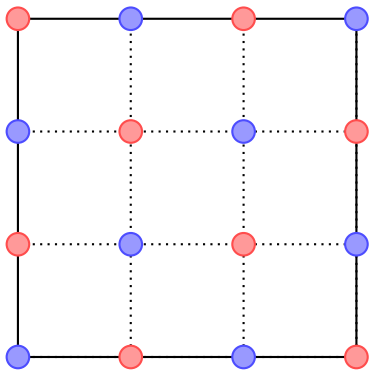

The GPU parallel strategy for the 2D case can be extended to 3D, as shown in Figure~\ref{gpu-3d}. Each red grid point in the 3D checkerboard grid depends only on its six adjacent blue grid points, with no dependencies among red points. For the base method \eqref{scheme1}, all red grid points are updated first, followed by all blue points, with both updates executed in parallel on the GPU in any order. For the adjoint method \eqref{scheme1-adj}, the update order is reversed: blue points are updated first, followed by red points, without affecting the internal order within each group.

\begin{figure}[H]
	\centering
\begin{minipage}{0.38\textwidth}
		\resizebox{1.0\textwidth}{5.5cm}{
			\begin{tikzpicture}[scale=1.0, every node/.style={circle, draw, inner sep=0pt, minimum size=2mm}]
				\draw[solid] (0,0,0) -- (3,0,0) -- (3,4.2,0) -- (0,4.2,0) -- cycle; 
				\draw[solid] (0,0,3) -- (3,0,3) -- (3,4.2,3) -- (0,4.2,3) -- cycle; 
				\draw[solid] (0,0,0) -- (0,0,3);
				\draw[solid] (3,0,0) -- (3,0,3);
				\draw[solid] (3,4.2,0) -- (3,4.2,3);
				\draw[solid] (0,4.2,0) -- (0,4.2,3);

				\fill[green!20,opacity=0.5] (0,0,0) -- (3,0,0) -- (3,0,3) -- (0,0,3) -- cycle; 
				\fill[blue!20,opacity=0.5] (0,1.4,0) -- (3,1.4,0) -- (3,1.4,3) -- (0,1.4,3) -- cycle; 
				\fill[green!20,opacity=0.5] (0,2.8,0) -- (3,2.8,0) -- (3,2.8,3) -- (0,2.8,3) -- cycle; 
				\fill[blue!20,opacity=0.5] (0,4.2,0) -- (3,4.2,0) -- (3,4.2,3) -- (0,4.2,3) -- cycle; 

				\foreach \z in {1, 2} {
						\draw[dotted] (0,0,\z) -- (3,0,\z);
						\draw[dotted] (0,4.2,\z) -- (3,4.2,\z);
						\draw[dotted] (0,0,\z) -- (0,4.2,\z);
						\draw[dotted] (3,0,\z) -- (3,4.2,\z);
						\draw[dotted] (1,0,\z) -- (1,4.2,\z);
						\draw[dotted] (2,0,\z) -- (2,4.2,\z);
						\draw[dotted] (0,1.4,\z) -- (3,1.4,\z);
						\draw[dotted] (0,2.8,\z) -- (3,2.8,\z);
					}

				\foreach \x in {1, 2} {
						\draw[dotted] (\x,0,0) -- (\x,4.2,0);
						\draw[dotted] (\x,0,3) -- (\x,4.2,3);
						\draw[dotted] (\x,0,0) -- (\x,0,3);
						\draw[dotted] (\x,4.2,0) -- (\x,4.2,3);
						\draw[dotted] (\x,1.4,0) -- (\x,1.4,3);
						\draw[dotted] (\x,2.8,0) -- (\x,2.8,3);
					}

				\foreach \y in {1.4, 2.8} {
						\draw[dotted] (0,\y,0) -- (3,\y,0);
						\draw[dotted] (0,\y,3) -- (3,\y,3);
						\draw[dotted] (0,\y,0) -- (0,\y,3);
						\draw[dotted] (3,\y,0) -- (3,\y,3);
						\draw[dotted] (1,\y,0) -- (1,\y,3);
						\draw[dotted] (2,\y,0) -- (2,\y,3);
					}

				\foreach \z in {1, 2}
				\foreach \x in {1, 2}
				\foreach \y in {1.4, 2.8} {
						\draw[dotted] (\x,\y,0) -- (\x,\y,3);
					}

				\foreach \x in {0, 1, 2, 3}
				\foreach \y in {0, 1.4, 2.8, 4.2}
				\foreach \z in {0, 1, 2, 3} {
						\pgfmathtruncatemacro{\colorindex}{mod(\x+\y/1.4+\z, 2)}
						\ifnum\colorindex=0
							\node[draw=red!70,fill=red!40] at (\x, \y, \z) {};
						\else
							\node[draw=blue!70,fill=blue!40] at (\x, \y, \z) {};
						\fi
					}
			\end{tikzpicture}}
	\end{minipage}
	\caption{3D checkerboard update order for GPU parallelization of the DP-AVF2 scheme for the KGS equations.}
	\label{gpu-3d}
\end{figure}
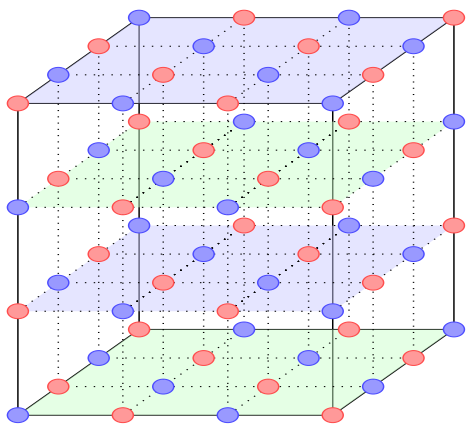


{\bf Algorithm 3} describes the GPU parallel implementation using a 2D checkerboard grid, as shown in Figure~\ref{gpu-2d}, with a similar approach applicable to the 3D checkerboard grid in Figure~\ref{gpu-3d}. The process begins by copying the required data to GPU memory to ensure efficient access during computation. For each time step $n$, the calculation is carried out in parallel by multiple threads, which are classified into red and blue grid points based on the parity of $(j+k)$. In {\bf Algorithm 3}, the base method is computed first, where red grid points are updated before blue points. Then, the adjoint method is computed, with the update order reversed: blue grid points are updated first, followed by red points. Each red or blue grid point is assigned to an individual thread, and all updates are performed in parallel across all grid points, fully exploiting the parallel capabilities of the GPU. After completing the updates, the results are copied back to the host memory. By employing this parallelization strategy, the algorithm fully exploits the GPU’s computational power, enabling a significant acceleration in solving large-scale problems.

\begin{table}[H]\label{tab:algo-gpu}
	\fontsize{10pt}{12pt}\selectfont
	\renewcommand\arraystretch{1}
	\centering
	\begin{tabular*}{0.9\textwidth}[h]{@{\extracolsep{\fill}}l} \toprule[2pt]
		{\bf Algorithm 3} GPU Parallel Implementation of the DP-AVF2 Scheme with Checkerboard Grid   \\\hline
		{\small 1.} {\bf Input:} Initial conditions, parameters \\
		{\small 2.} {\bf Output:} $\Psi^{n+1}$, $U^{n+1}$, $V^{n+1}$  \\\hline
		{\small 3.} Compute the inverse of the $2 \times 2$ coefficient matrix in the last two equations of scheme \eqref{scheme1}. \\
		{\small 4.} Copy data to GPU memory. \\
		{\small 5.} {\bf for} $n = 1, \cdots, T$ \\
		{\small 6.} \qquad {\bf for each thread $(j,k)$ {(in parallel):}}\\
		{\small 7.} \qquad \qquad {\bf if} $(j + k) \bmod 2 == 1$:  \hfill {\scriptsize \it // Check if red point} \\
		{\small 8.} \qquad \qquad \qquad Solve scheme \eqref{scheme1} explicitly with $\frac{\tau}{2}$ for $\Psi_{j,k}^{n+1}$, followed by $U_{j,k}^{n+1}$, $V_{j,k}^{n+1}$. \\
		{\small 9.} \qquad \qquad {\bf end} \\
		{\small 10.} \qquad {\bf end} \\
		{\small 11.} \qquad {\bf for each thread $(j,k)$ (in parallel):} \\
		{\small 12.} \qquad \qquad {\bf if} $(j + k) \bmod 2 == 0$:  \hfill {\scriptsize \it // Check if blue point} \\
		{\small 13.} \qquad \qquad \qquad Solve scheme \eqref{scheme1} explicitly with $\frac{\tau}{2}$ for $\Psi_{j,k}^{n+1}$, followed by $U_{j,k}^{n+1}$, $V_{j,k}^{n+1}$. \\
		{\small 14.} \qquad \qquad {\bf end} \\
		{\small 15.} \qquad {\bf end} \\
		{\small 16.} \qquad {\bf for each thread $(j,k)$ (in parallel):} \\
		{\small 17.} \qquad \qquad {\bf if} $(j + k) \bmod 2 == 0$:  \hfill {\scriptsize \it // Check if blue point} \\
		{\small 18.} \qquad \qquad \qquad Solve scheme \eqref{scheme1} explicitly with $\frac{\tau}{2}$ for $\Psi_{j,k}^{n+1}$, followed by $U_{j,k}^{n+1}$, $V_{j,k}^{n+1}$. \\
		{\small 19.} \qquad \qquad {\bf end} \\
		{\small 20.} \qquad {\bf end} \\
		{\small 21.} \qquad {\bf for each thread $(j,k)$ {(in parallel):}}\\
		{\small 22.} \qquad \qquad {\bf if} $(j + k) \bmod 2 == 1$:  \hfill {\scriptsize \it // Check if red point} \\
		{\small 23.} \qquad \qquad \qquad Solve scheme \eqref{scheme1} explicitly with $\frac{\tau}{2}$ for $U_{j,k}^{n+1}$, $V_{j,k}^{n+1}$, followed by $\Psi_{j,k}^{n+1}$. \\
		{\small 24.} \qquad \qquad {\bf end} \\
		{\small 25.} \qquad {\bf end} \\
		{\small 26.} {\bf end} \\
		{\small 27.} Copy $\Psi^{n+1}$, $U^{n+1}$, and $V^{n+1}$ back from GPU memory to host memory. \\
		\bottomrule[2pt]
	\end{tabular*}
\end{table}

\section{Numerical experiments}\label{sec:5}

In this section, we present numerical experiments to assess the efficiency and conservative properties of the proposed DP-AVF2 scheme, comparing it with three existing energy-preserving methods: the Coupled Implicit Finite Difference (CIFD) method \cite{kgs_cifd}, the Decoupled Implicit Finite Difference (DIFD) method \cite{kgs_difd}, and the SAV Crank-Nicolson (SAV-CN) scheme \cite{fkgs_sav_mcs,kgs_sav_jsc}. 

The CIFD scheme is given by the following system of equations:
\begin{equation*}
\left\lbrace
\begin{aligned}
 & i \delta_{t}^+ \Psi^{n} = -\frac{\kappa_1}{2}\Delta\Psi^{n+\frac{1}{2}} - \gamma \Psi^{n+\frac{1}{2}} U^{n+\frac{1}{2}}, \\
 & \delta_t^2 U^n = \kappa_2 \Delta_h U^{\bar{n}} - \mu^2 U^{\bar{n}} + \gamma |\Psi^n|^2,
\end{aligned}
\right.
\end{equation*}
where $\delta_t^2 U^n = (U^{n+1} - 2U^n + U^{n-1})/\tau^2$ and $U^{\bar{n}} = (U^{n+1} + U^{n-1})/2$. In the CIFD scheme, $ U^{n+1} $ is computed first by solving a constant linear system derived from the second equation. This result is then substituted into the first equation to solve for $ \Psi^{n+1} $, which requires solving a variable linear system.

The DIFD scheme is expressed as:
\begin{equation}\label{difd}
\left\lbrace
\begin{aligned}
 & i \delta_{t} \Psi^n = -\frac{\kappa_1}{2}\Delta_h\Psi^{\bar{n}} - \gamma \Psi^{\bar{n}} U^n, \\
 & \delta_t^2 U^n = \kappa_2 \Delta_h U^{\bar{n}} - \mu^2 U^{\overline{n}} + \gamma |\Psi^n|^2,
\end{aligned}
\right.
\end{equation}
where $ \delta_{t} U^n = (U^{n+1} - U^{n-1})/(2\tau)$. The DIFD scheme allows for the simultaneous solutions of $ \Psi^{n+1} $ and $ U^{n+1} $ from the first and second equations, respectively. This enables parallel computation by variables. Although $ U^{n+1} $ can be efficiently solved using FFT, the primary computational cost of the DIFD scheme arises from solving for $ \Psi^{n+1} $, which involves a variable linear system, making the computational efficiency of the DIFD method comparable to that of the CIFD scheme.

The SAV-CN scheme is given by:
\begin{equation*}
\left\lbrace
\begin{aligned}
 & i \delta_{t}^+ \Psi^{n}  = -\frac{\kappa_1}{2} \Delta_h \Psi^{n+\frac{1}{2}} - \gamma \frac{ R^{n+\frac{1}{2}}}{\sqrt{\mathcal{E}_1+C_0}}\Psi^{\star} U^{\star}, \\
 & \delta_{t}^+ V^{n} =  \kappa_2 \Delta_h U^{n+\frac{1}{2}} - \mu^2 U^{n+\frac{1}{2}} + \gamma \frac{ R^{n+\frac{1}{2}}}{\sqrt{\mathcal{E}_1+C_0}}\left|\Psi^{\star}\right|^2, \\
 & \delta_{t}^+ U^{n} = V^{n+\frac{1}{2}}, \\
 & \delta_{t}^+ R^{n} = - \frac{ 1 }{2} \frac{ 2{\rm Re} \left( \Psi^{\star} U^{\star}, \delta_t \Psi^{n} \right)_h + \left(|\Psi^{\star}|^2, \delta_t U^{n} \right)_h }{\sqrt{\mathcal{E}_1+C_0}},
\end{aligned}
\right.
\end{equation*}
where $\mathcal{E}_1 := -\left(|\Psi^{\star}|^2, U^{\star}\right)_h$, and  $U^{\star} = (3U^n - U^{n-1})/2$. The SAV-CN scheme is highly efficient, as it leverages FFT to solve several constant linear systems along with a scalar linear equation. This makes it one of the most computationally efficient schemes among those presented. However, in the following experiments, we demonstrate that the proposed DP-AVF2 scheme, even in its serial implementation, outperforms the SAV-CN scheme in terms of computational speed, with even greater advantages seen in its CPU and GPU parallel implementations.

\subsection{Examples of 2D KGS equations}

We begin by assessing the accuracy, energy conservation, and computational efficiency of the proposed DP-AVF2 scheme, comparing it with the above three existing energy-preserving methods. These comparisons are conducted using the 2D KGS equations \eqref{kgs} with the following initial conditions:
\[
\left\lbrace
\begin{aligned}
  & \psi_0(\bm{x}) = (1 + i) \exp{\left( -(x^2 + y^2) \right)}, \\
  & u_0(\bm{x}) = \tanh{\left( x^2 + y^2 \right)}, \\
  & v_0(\bm{x}) = \sin{(x + y)} \exp{\left( -2(x^2 + y^2) \right)}.
\end{aligned}
\right.
\]
The computational domain is specified as $ \Omega = (-10, 10) $, and the model parameters are set as $ \kappa_1 = \kappa_2 = \gamma = \mu = 1 $. Since the exact solution of the KGS equations is unavailable for these initial conditions, a high-accuracy reference solution is computed using a sixth-order Petrov-Galerkin method for time discretization and a Fourier spectral method for spatial discretization, as described in Refs. \cite{kgs_hbvms}. To ensure high precision, the discretization parameters are chosen as $ (h, \tau) = \left(\frac{5}{256}, \frac{1}{200}\right) $, which results in the reference solutions $ \psi^{\text{ref}} $ and $ u^{\text{ref}} $. 

We first conduct accuracy tests on the DP-AVF2 scheme and the three comparison schemes by selecting different spatiotemporal step sizes, with the computation terminating at $ t = 1 $. The discrete $ L^2 $ errors of $\psi$ and $u$ are computed as follows
\[
\text{Error}_{\psi}(h, \tau) = \max\left\{ \left\| \text{Re} \left( \Psi^{N_t} - \psi^{\text{ref}} \right) \right\|_h, \left\| \text{Im} \left( \Psi^{N_t} - \psi^{\text{ref}} \right) \right\|_h \right\},
\quad
\text{Error}_u (h, \tau) = \left\| U^{N_t} - u^{\text{ref}} \right\|_h,
\]
where $ \Psi^{N_t} $ and $ U^{N_t} $ are the numerical solutions at the final time step. The convergence rate is then determined using the formula
\[
\text{Order} = \log_2 \left( \frac{\text{Error}(h, \tau)}{\text{Error}(h/2, \tau/2)} \right).
\]

Table~\ref{tab:conv2d} presents the computational errors and corresponding convergence rates for the four schemes with different spatiotemporal step sizes. We adopted a strategy of halving both the time and space step sizes simultaneously, and as shown in Table~\ref{tab:conv2d}, all schemes achieve second-order convergence in both time and space. Moreover, for each step size, the errors of the four schemes are generally of the same order of magnitude. Therefore, to compare computational efficiency, we only need to directly examine the computational time for each scheme when simulating up to $ t = 1 $. Table~\ref{cost-2d} reports the CPU times for the four schemes with a fixed time step size of $ \tau = 0.005 $ and different spatial discretizations. This test was performed on a laptop with an 8-core Apple M3 Pro chip and 32 GB of memory, using Cython for the programming language.

From Table~\ref{cost-2d}, it is evident that the DP-AVF2 scheme requires the least CPU time, as it allows for pointwise explicit solution, with a computational complexity of $ \mathcal{O}(N^2) $. The SAV-CN scheme, involving the solution of constant coefficient linear systems and efficiently implemented using FFT, has a computational complexity of $ \mathcal{O}(N^2 \log N) $, making it the second most efficient method after DP-AVF2. In contrast, both the CIFD and DIFD schemes involve solving variable coefficient linear systems, making them the two most time-consuming schemes. 

\renewcommand{\arraystretch}{1.5}
\begin{table}[H]
	\caption{Solution errors and convergence orders of different schemes with $(h_0, \tau_0) = \left(\frac{5}{16}, \frac{1}{50}\right)$.}\label{tab:conv2d}
	\centering
		\resizebox{\textwidth}{!}{
			\begin{tabular}{lllllllll}
				\toprule
				$(h, \tau)$ & ${ \rm Error }_u$ & ${ \rm Order }_u$ & ${ \rm Error }_\psi$ & ${ \rm Order }_\psi$ & ${ \rm Error }_u$ & ${ \rm Order }_u$ & ${ \rm Error }_\psi$ & ${ \rm Order }_\psi$                                                                                              \\\hline
				& DP-AVF2 & & &  & SAV-CN & & &\\
				$(h_0, \tau_0)$              & 7.5372e-02                  & -                           & 7.6161e-02               & -                        & 7.6973e-02           & -                    & 7.7062e-02           & -                    \\
				$(h_0/2, \tau_0/2)$          & 1.9193e-02                  & 1.9734                      & 1.9582e-02               & 1.9595                   & 1.9848e-02           & 1.9554               & 1.9869e-02           & 1.9555               \\
				$(h_0/4, \tau_0/4)$          & 4.5475e-03                  & 2.0774                      & 4.8110e-03               & 2.0251                   & 4.8798e-03           & 2.0241               & 4.8848e-03           & 2.0241               \\
				\midrule
				& CIFD & & &  & DIFD & & &\\
				$(h_0, \tau_0)$              & 2.1597e-02                  & -                           & 2.2127e-02               & -                        & 2.2052e-02           & -                    & 2.2113e-02           & -                    \\
				$(h_0/2, \tau_0/2)$          & 4.8403e-03                  & 2.1577                      & 5.3051e-03               & 2.0604                   & 5.2957e-03           & 2.0579               & 5.3028e-03           & 2.0601               \\
				$(h_0/4, \tau_0/4)$          & 8.5701e-04                  & 2.4977                      & 1.3089e-03               & 2.0190                   & 1.3071e-03           & 2.0185               & 1.3086e-03           & 2.0187               \\
				\bottomrule
			\end{tabular}}
\end{table}

\begin{table}[H]
	\caption{The CPU time of the proposed different schemes under different spatial grid sizes.}\label{tab:cpu_time_diff_methods2d}\label{cost-2d}
	\centering
		\resizebox{\textwidth}{!}{
		{\scriptsize\begin{tabular}{ccccccc}
			\toprule
			$N$    & DP-AVF2 & {SAV-CN} & {CIFD}-FFT & {CIFD}-PBiCGS & {DIFD}-FFT & {DIFD}-PBiCGS \\
			\midrule
			$256$  & 7.54e$-$02       & 8.35e$-$01      & 2.63e+00          & 5.57e+00             & 2.96e+00          & 6.15e+00             \\
			$512$  & 2.99e+00         & 4.83e+00        & 1.55e+01          & 1.42e+01             & 1.58e+01          & 1.34e+01             \\
			$1024$ & 1.19e+01         & 2.42e+01        & 7.34e+01          & 4.85e+01             & 7.36e+01          & 4.94e+01             \\
			\bottomrule
		\end{tabular}}}
\end{table}

We also compared two strategies for solving variable coefficient linear systems for the CIFD and DIFD schemes: the first involves using existing iterative solvers for linear systems, i.e., the multigrid preconditioned BiCGS stable (PBiCGS) method. The second approach combines FFT with fixed-point iteration techniques. Taking the solution of $ \Psi^{n+1} $ in the DIFD scheme \eqref{difd} as an example, we can solve the first equation as:
\[
\left( i + \frac{\tau \kappa_1 \Delta}{2} \right) \Psi^{n+1,s+1} = \left( i - \frac{\tau \kappa_1 \Delta_h}{2} \right) \Psi^{n-1} - \tau \gamma U^{n-1} (\Psi^{n+1,s} + \Psi^{n-1}), \quad s = 0, 1, \cdots
\]
Here, the superscript $ s $ represents the iteration index, and $ \Psi^{n+1,0} = \Psi^n $. In each iteration, only a constant linear system needs to be solved, where FFT can be efficiently applied. As shown in Table~\ref{cost-2d}, when $ N $ is small, the FFT-based fixed-point iteration method is faster than the iterative solvers. However, as $ N $ increases, the iterative solvers outperform the FFT-based approach. Nonetheless, the computational times for both of these strategies remain significantly larger than those for the DP-AVF2 and SAV-CN schemes.

Next, we consider the energy conservation of the presented schemes. The left figure of Figure~\ref{efficiency-energy-2d} shows the evolution of the relative energy error for the four schemes util $t=10$ with $h=\frac{5}{64}, \tau=0.1$. The relative energy is defined as:
\[
RE^n = \frac{\left|\mathcal{E}_h^n - \mathcal{E}_h^0\right|}{\left|\mathcal{E}^0_h\right|}.
\]
It is clear that all schemes maintain the discrete energy to machine precision. However, it is important to note that the energy definition in the CIFD and DIFD schemes involves averaging over two time steps, whereas the energy in the SAV-CN scheme corresponds to a modified quadratic energy. For specific forms of the energy in these three schemes, please refer to Refs. \cite{kgs_cifd,kgs_difd,fkgs_sav_mcs,kgs_sav_jsc}. Only the energy of the DP-AVF2 scheme directly corresponds to the energy \eqref{ene-origin} of the continuous system.

To further demonstrate that the energy conservation property of the DP-AVF2 scheme is independent of the grid point update order, we tested three randomly selected update orders which actually result three different DP-AVF2 schemes. The right-hand plot in Figure~\ref{efficiency-energy-2d} shows the evolution of relative energy errors for the three DP-AVF2 schemes. It is evident that the discrete energy is preserved to machine precision regardless of the update order.

\begin{figure}[H]
	\centering
	\includegraphics[width=0.48\textwidth]{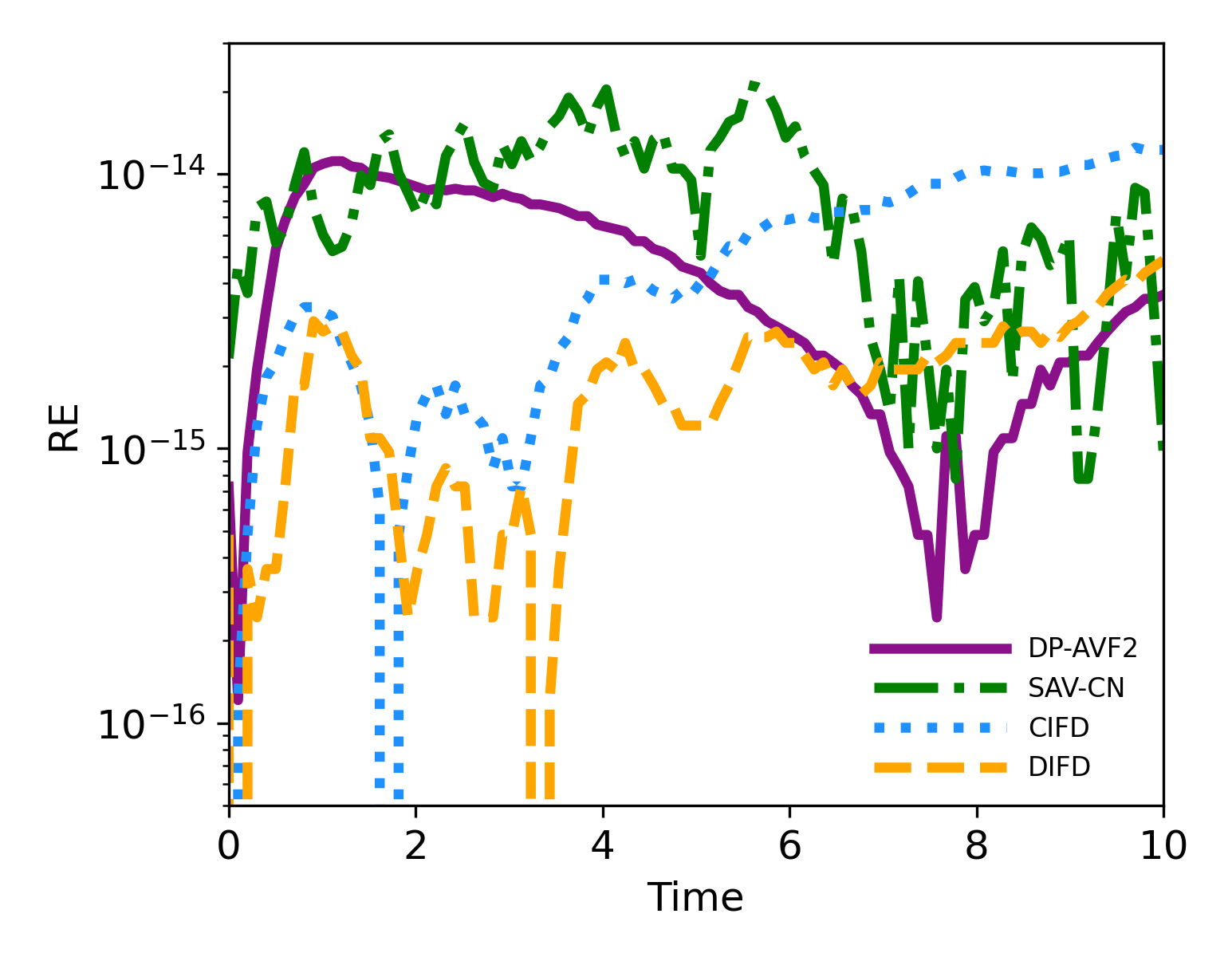}
	\includegraphics[width=0.48\textwidth]{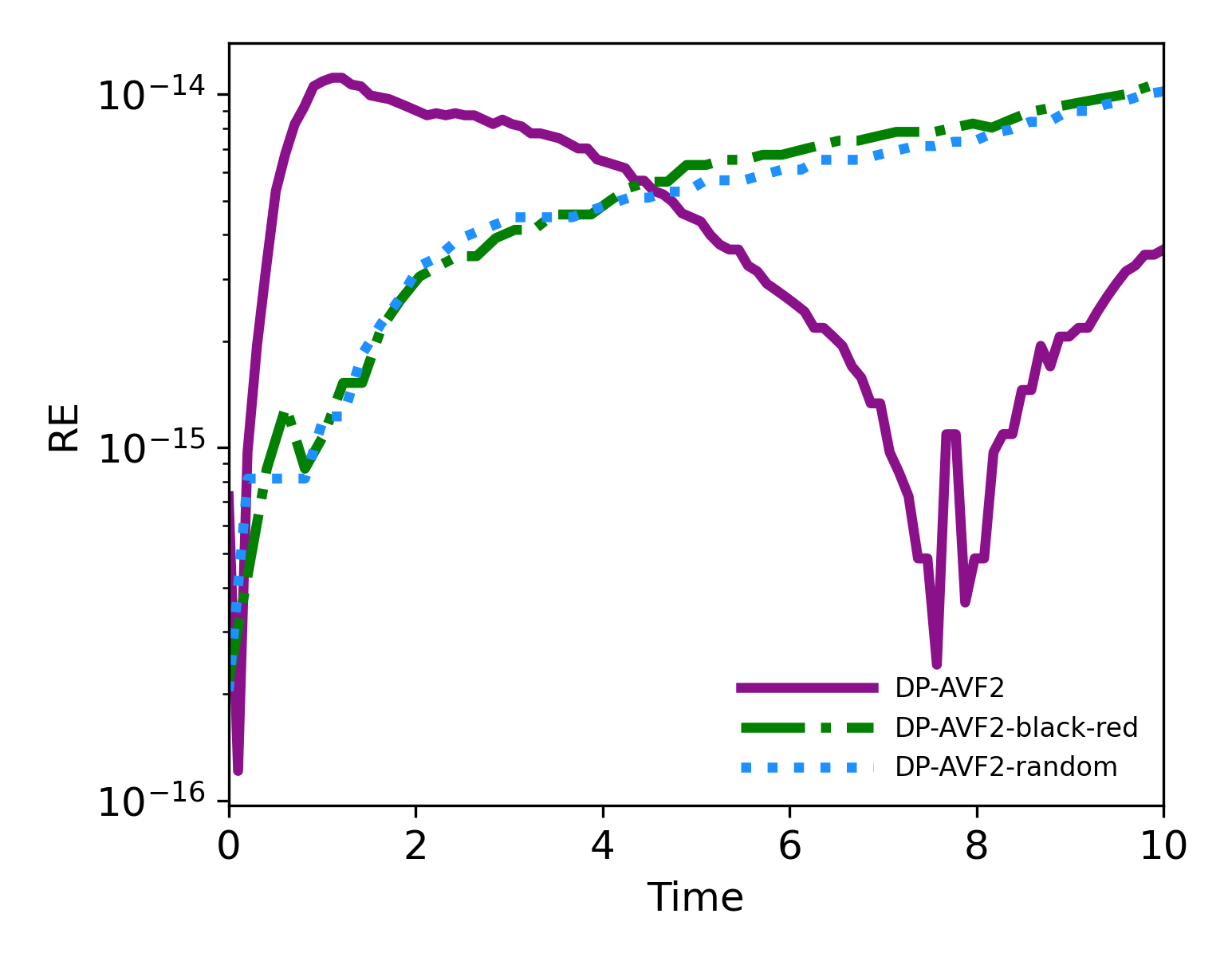}
	\caption{The temporal evolution of the relative discrete energy error of different schemes (left) and different update orders of the DP-AVF2 method (right).}\label{efficiency-energy-2d}
\end{figure}

In the second example, we investigate the parallel efficiency of the DP-AVF2 method applied to the KGS equations with the following initial conditions:
\[
\left\lbrace
\begin{aligned}
\psi_0(\bm{x}) & = (1 + i) \sum\limits_{j = 1}^4 \exp{\left(- (x - x_j)^2 - (y - y_j)^2\right)}, \\
u_0(\bm{x})    & = \sum\limits_{j=1}^4 \tanh{\left( (x - x_j)^2 + (y - y_j)^2 \right)}, \\
v_0(\bm{x})    & = \exp{(-x^2 - y^2)},
\end{aligned}
\right.
\]
where the coordinates are given by $(x_1, y_1) = (0, -3), (x_2, y_2) = (3, 0), (x_3, y_3) = (0, 3), (x_4, y_4) = (-3, 0)$. The spatial domain is set to $\Omega = (-10, 10)^2$. The parameters in \eqref{kgs} are chosen as $\kappa_1 = \kappa_2 = \gamma = \mu = 0.5$, and the time step size is $\tau = 0.01$. The CPU parallel implementation is based on a thread-based environment in Cython, with parallelization achieved by converting the loop range to prange. The GPU implementation is written in C, with the model used being the NVIDIA A40.

Figure~\ref{plot2d} illustrates the evolution of soliton collisions in components $\psi$ and $u$ under CPU or GPU parallel computation with $h = \frac{5}{64}$. Initially, the distribution of $\psi$ consists of four distinct peaks. Over time, these peaks undergo radiation and collision phenomena, leading to the formation of new peaks at the center of the domain. For component $u$, the initial state also consists of four distinct peaks. The results indicate that: (i) interactions among vector solitons generate five peaks within the domain; (ii) during soliton collisions, the amplitude of the central peak continuously increases, while the amplitudes of the other four peaks diminish. Furthermore, as the system evolves, the central peak exhibits noticeable radiative behavior.

\begin{figure}[H]
	\centering
		\includegraphics[width=0.23\textwidth]{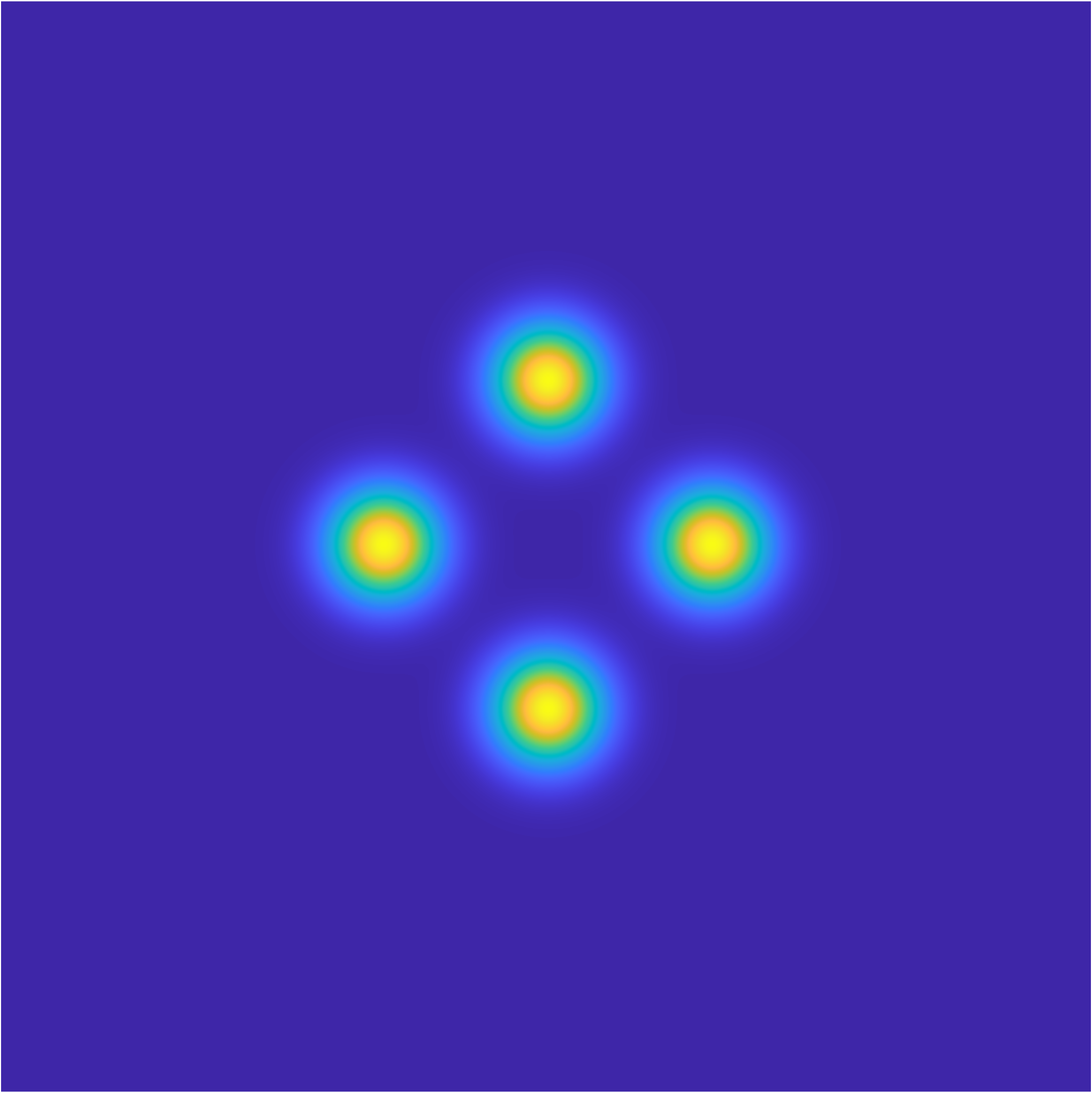}
		\includegraphics[width=0.23\textwidth]{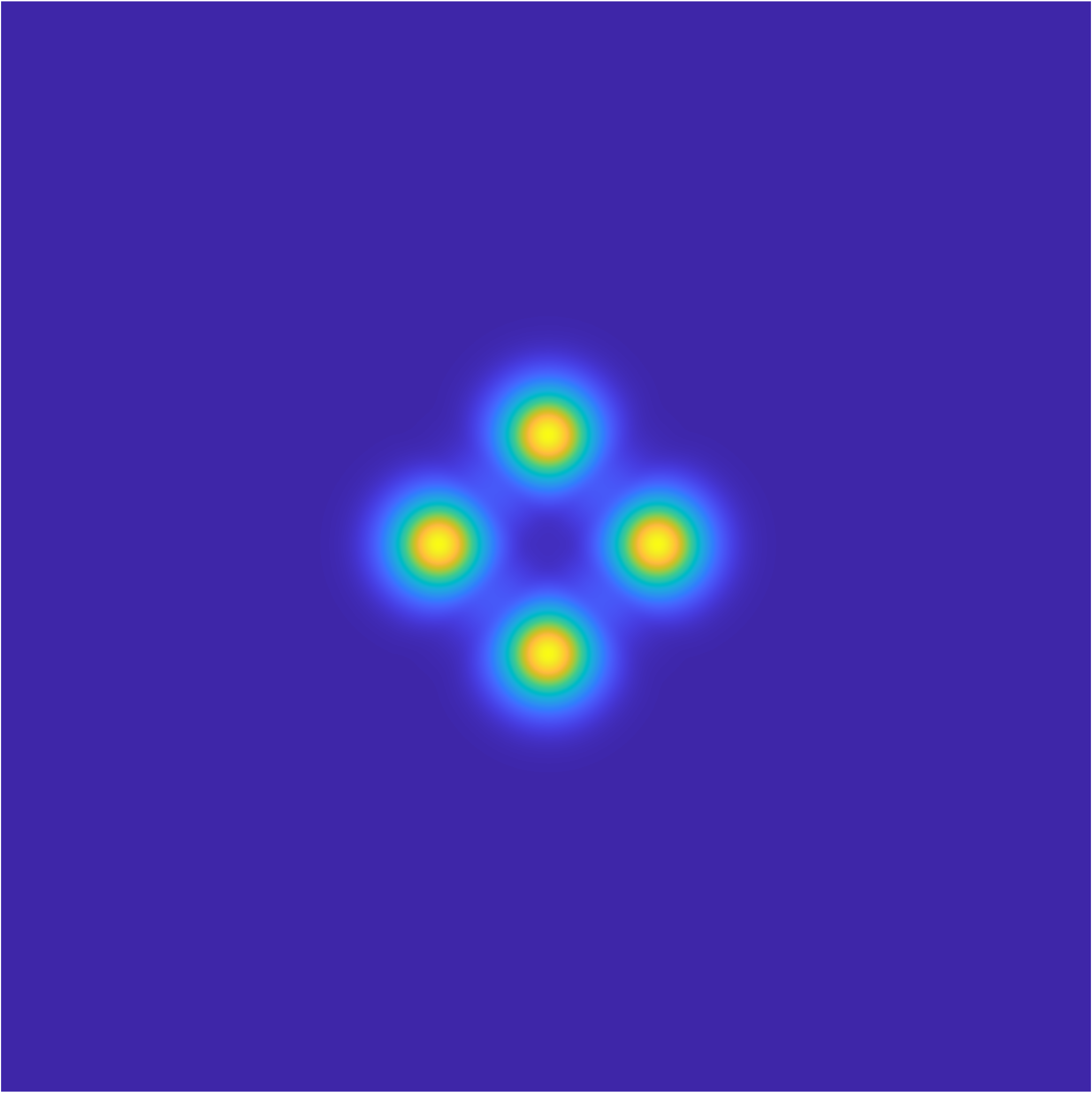}
		\includegraphics[width=0.23\textwidth]{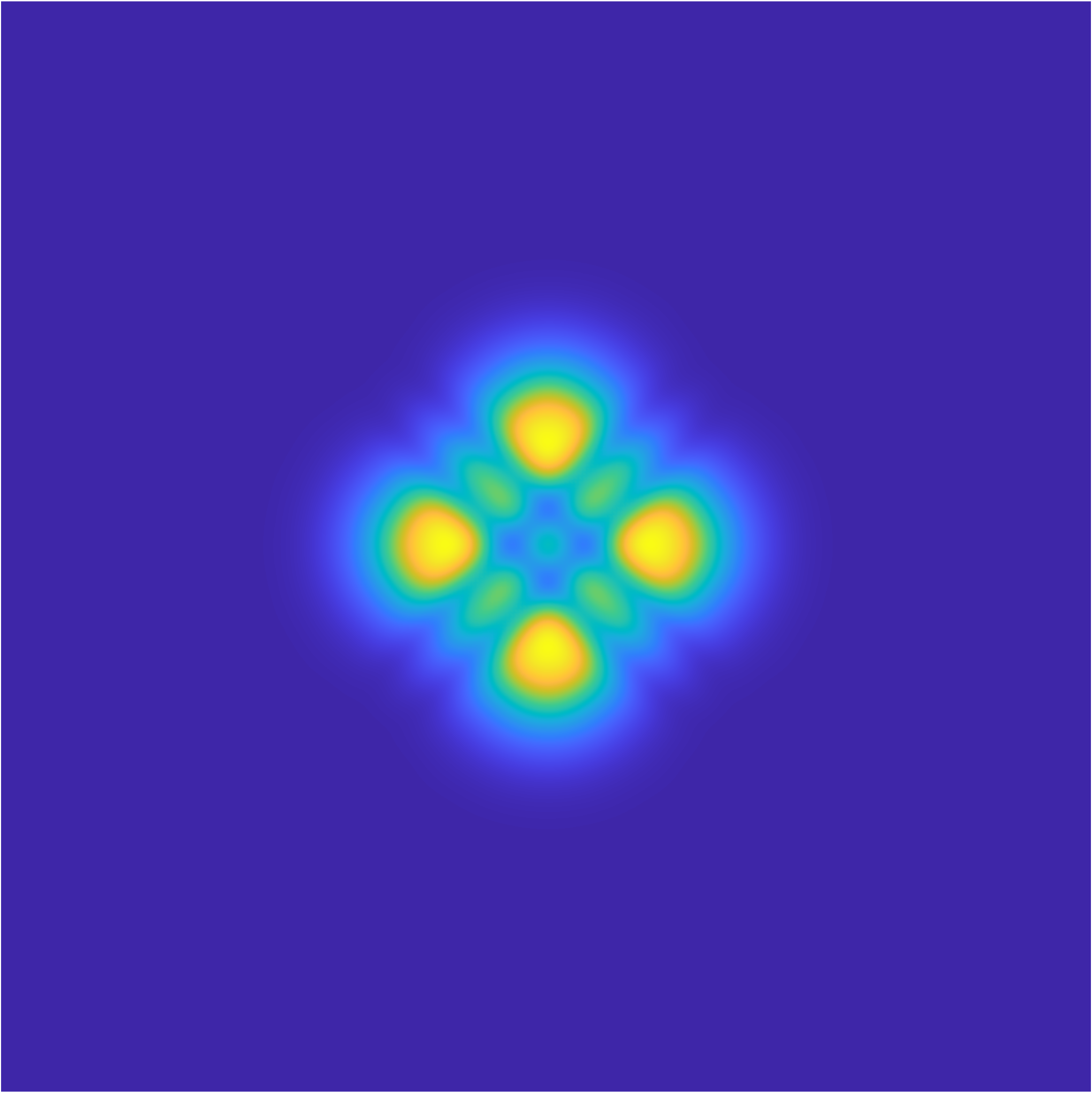}
		\includegraphics[width=0.23\textwidth]{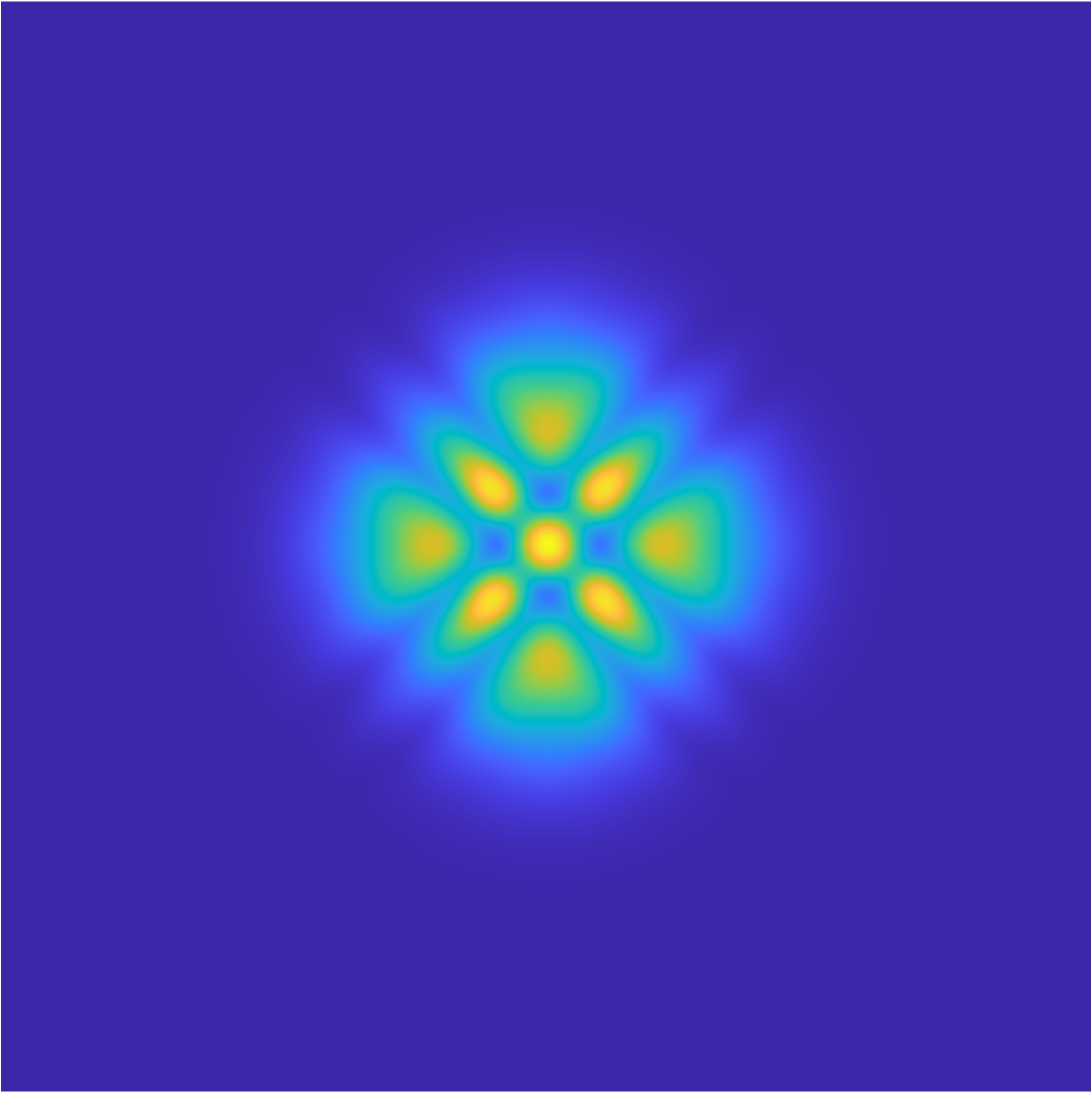}
		\includegraphics[width=0.23\textwidth]{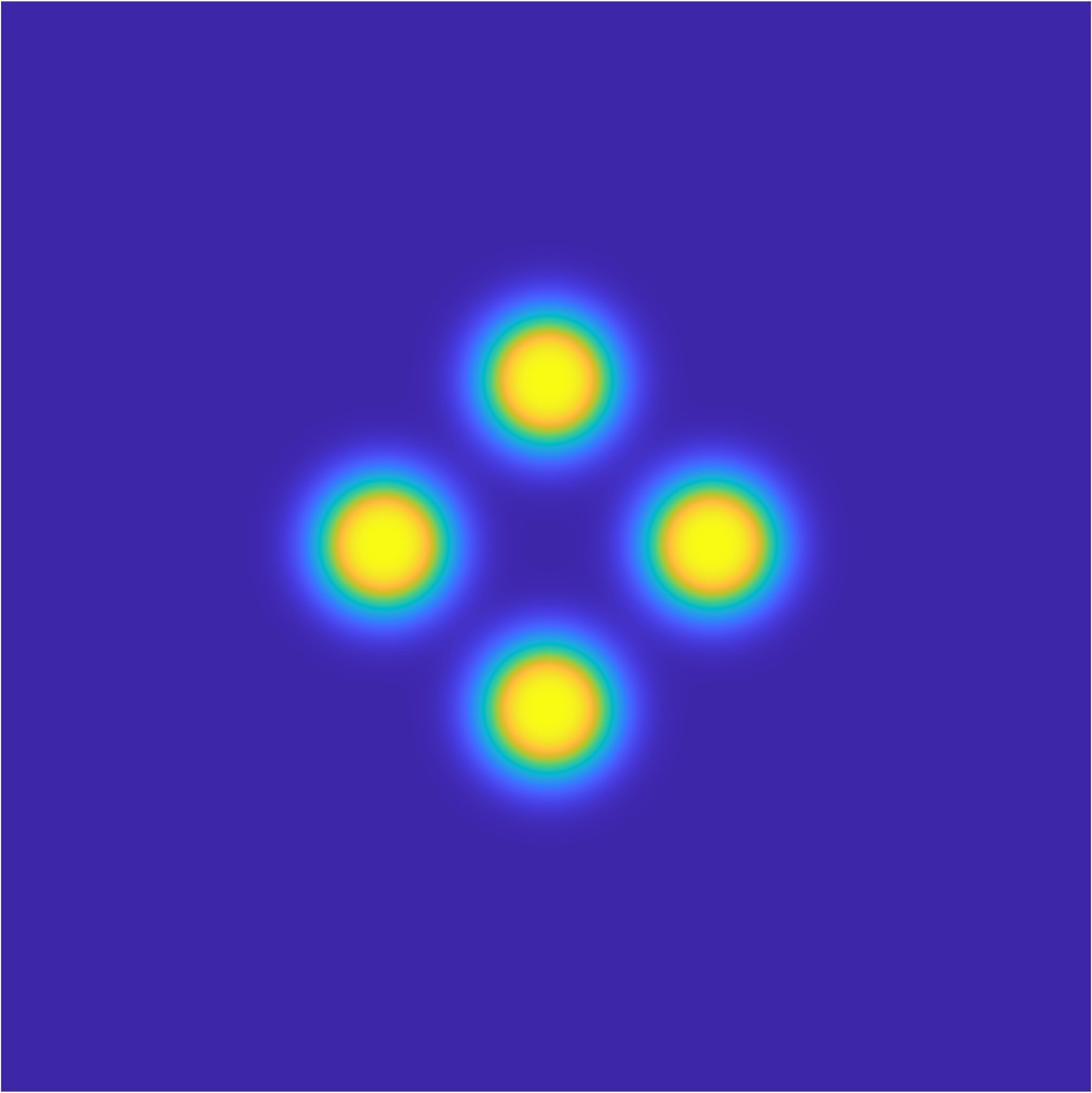}
		\includegraphics[width=0.23\textwidth]{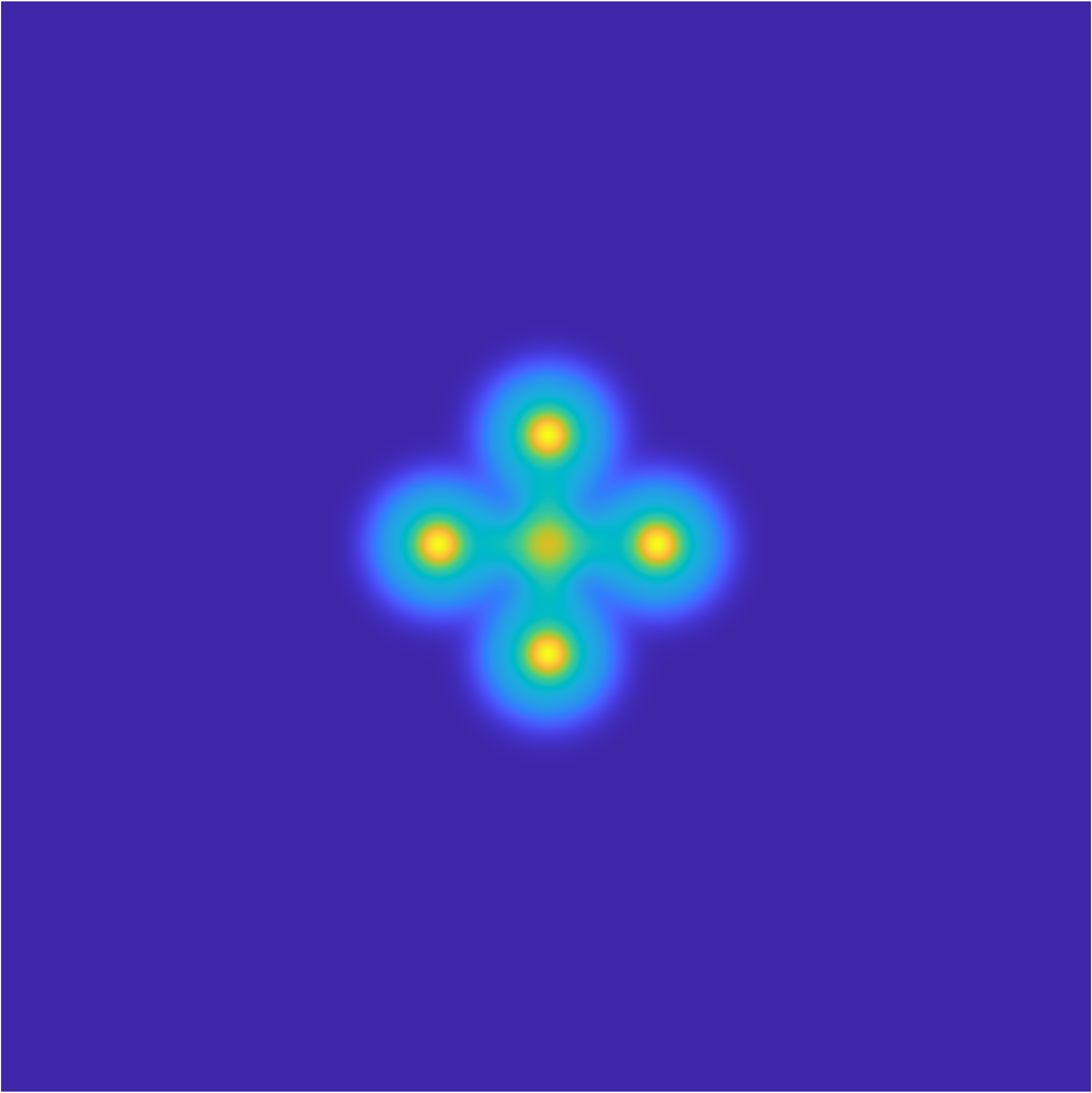}
		\includegraphics[width=0.23\textwidth]{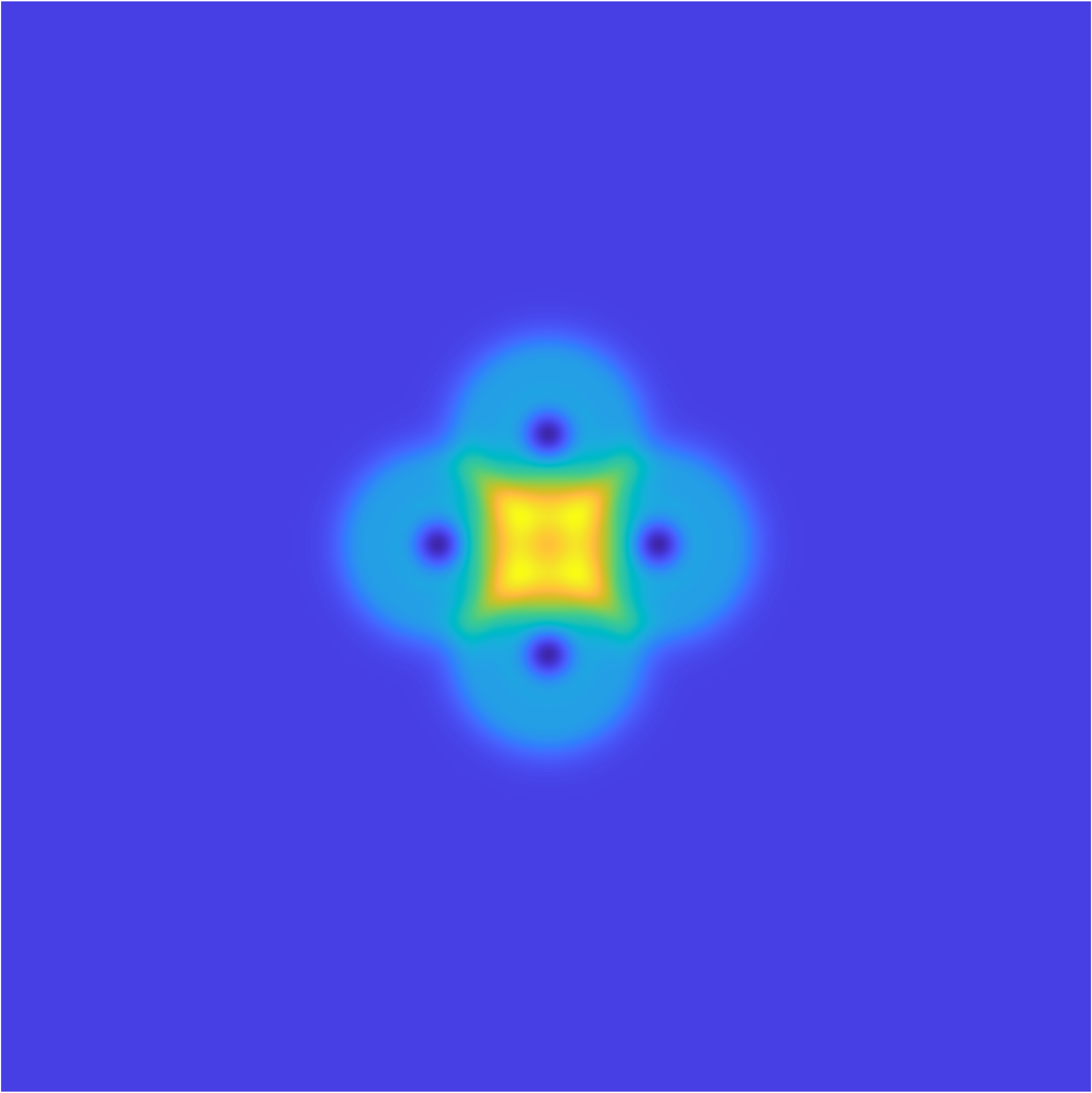}
		\includegraphics[width=0.23\textwidth]{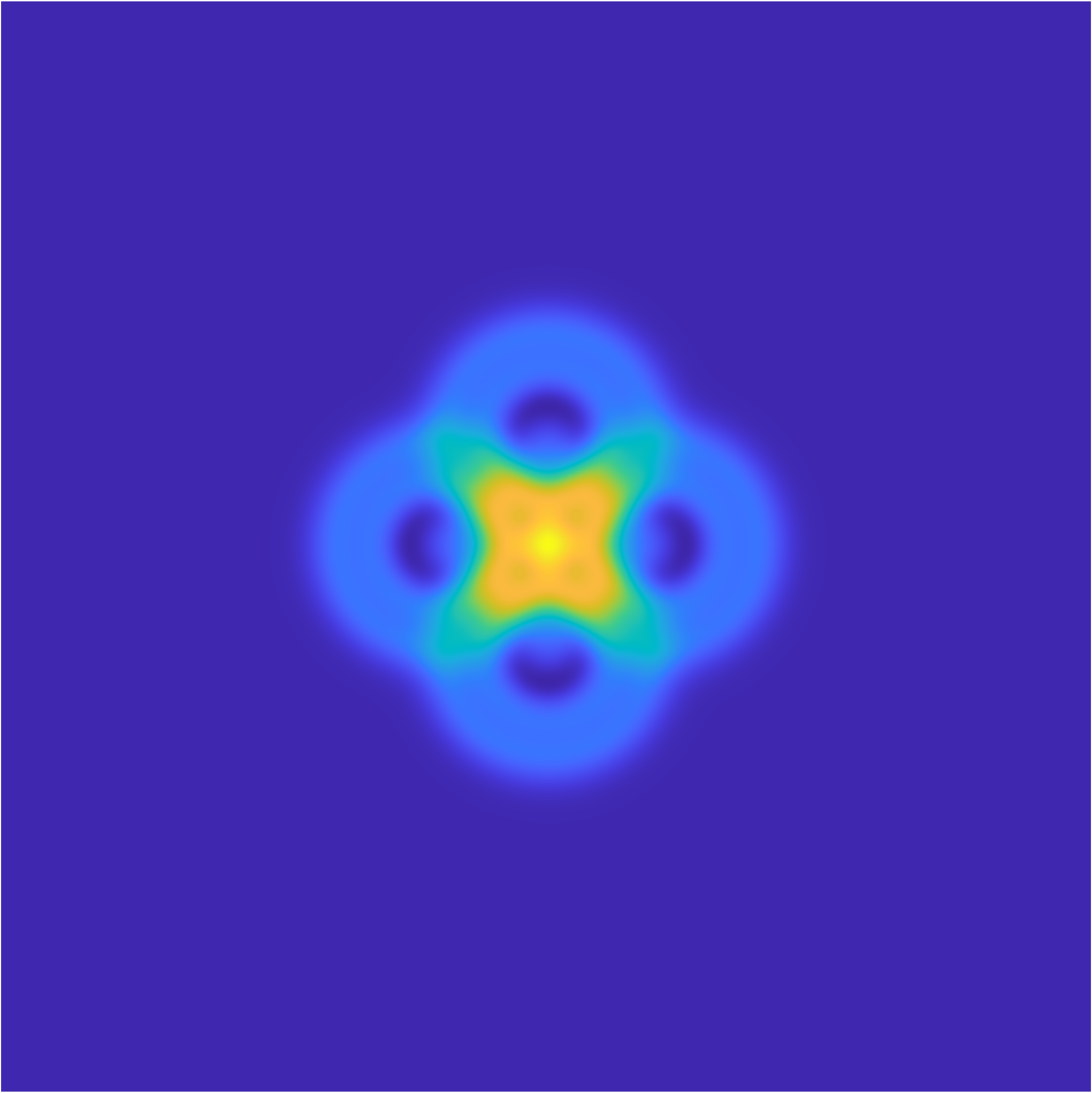}
	\caption{Snapshot $|\psi|$ (top row) and $u$ (bottom row) solved by the DP-AVF2 scheme under different times $t = 0, 1, 2, 3$ (from left to right).}\label{plot2d}
\end{figure}

We then consider the block-splitting parallel strategy illustrated in Figure~\ref{cpu-2d}. Table~\ref{tab:time-2d-cpu} presents the CPU times and speed-ups for the DP-AVF2 schemes, computed up to $t = 1$, at different thread counts and spatial discretizations. The results show that as the number of threads increases, the speed-up improves, but the acceleration is less than the theoretical 2x increase per doubling of threads. This limitation is particularly noticeable when the number of threads approaches the number of CPU cores (8 cores in this case). The parallel efficiency starts to diminish due to thread contention and resource bottlenecks. However, this issue can be mitigated by using parallel clusters with more cores, where the acceleration effect would be more pronounced. Despite this, the DP-AVF2 scheme still demonstrates strong parallel efficiency, making it highly scalable for larger systems and distributed computing environments.

\begin{table}[H]
	\caption{CPU times and speed-ups of the DP-AVF2 schemes with varying numbers of threads and spatial discretizations for the 2D KGS equations. The block-splitting parallel strategy in Figure \ref{cpu-2d} is applied.}\label{tab:time-2d-cpu}
	\centering
	\resizebox{\textwidth}{!}{
		{\scriptsize\begin{tabular}[c]{c|c|cc|cc|cc}
			\hline
			\multicolumn{1}{c|}{}    &
			\multicolumn{7}{c}{\bf CPU time (s)}                                                                            \\
			\hline
			\multicolumn{1}{c|}{$N$} & 1-thread & 2-thread   & speed-up & 4-thread   & speed-up & 8-thread   & speed-up \\
			\hline
			512                      & 1.52e+00 & 8.13e$-$01 & 1.87     & 4.08e$-$01 & 3.73     & 4.02e$-$01 & 3.78     \\
			1024                     & 6.04e+00 & 3.20e+00   & 1.89     & 1.72e+00   & 3.51     & 1.39e+00   & 4.35     \\
			2048                     & 2.42e+01 & 1.27e+01   & 1.91     & 6.47e+00   & 3.74     & 4.75e+00   & 5.09     \\
			\bottomrule
		\end{tabular}}}
\end{table}

Next, we examine the checkerboard parallel strategy of the DP-AVF2 scheme shown in Figure~\ref{gpu-2d}. Although we have focused on the GPU parallel strategy for the checkerboard grid, this approach can also be adapted for CPU parallelism. By assigning the updates of red or blue grid points to multiple CPU threads for simultaneous computation, the same strategy can be applied. Table~\ref{tab:time-2d} lists the computational times for both CPU and GPU parallel strategies based on the checkerboard grid, along with the speed-up compared to single-threaded (serial) computation. 

It is evident that the CPU parallel strategy based on the checkerboard grid requires less time than the block-splitting parallel strategy, as the latter involves updating grid points between blocks, which is not only a serial process but also requires data transfer. In addition, compared to Table~\ref{tab:time-2d-cpu}, the CPU parallel speed-up with the checkerboard grid is slightly lower for the same computation scale, and it still shows a trend where the speed-up decreases more from the theoretical value as the number of processes approaches 8. When the checkerboard grid is used for GPU parallel computation, the computational time is significantly reduced, and the speed-ups exceed 50, demonstrating that the checkerboard parallel strategy is particularly well-suited for GPU parallel computing.

\begin{table}[H]
	\caption{Computational times and speed-ups of the DP-AVF2 schemes on CPU and GPU with varying numbers of threads and spatial discretizations for the 2D KGS equations.}\label{tab:time-2d}
	\centering
	\resizebox{\textwidth}{!}{
		\begin{tabular}[c]{c|c|cc|cc|cc||cc}
			\hline
			\multicolumn{1}{c|}{}             &
			\multicolumn{7}{c||}{\bf CPU time (s)} &
			\multicolumn{2}{c}{\bf GPU time (s)}
			\\
			\hline
			\multicolumn{1}{c|}{$N$}          & 1-thread & 2-thread   & speed-up & 4-thread   & speed-up & 8-thread   & speed-up & time       & speed-up \\
			\hline
			512                               & 1.18e+00 & 6.40e$-$01 & 1.84     & 4.00e$-$01 & 2.95     & 3.59e$-$01 & 3.29     & 2.29e$-$02 & 51.53    \\
			768                               & 2.71e+00 & 1.38e+00   & 1.96     & 8.24e$-$01 & 3.28     & 6.26e$-$01 & 3.73     & 5.25e$-$02 & 51.62    \\
			1024                              & 4.83e+00 & 2.61e+00   & 1.85     & 1.36e+00   & 3.55     & 1.25e+00   & 3.86     & 9.34e$-$02 & 51.71    \\
			1280                              & 7.44e+00 & 3.97e+00   & 1.87     & 2.14e+00   & 3.48     & 1.82e+00   & 4.09     & 1.28e$-$01 & 58.13    \\
			1536                              & 1.08e+01 & 5.51e+00   & 1.96     & 3.02e+00   & 3.58     & 2.57e+00   & 4.20     & 1.96e$-$01 & 55.10    \\
			1792                              & 1.45e+01 & 7.55e+00   & 1.92     & 4.58e+00   & 3.17     & 3.44e+00   & 4.22     & 2.76e$-$01 & 52.54    \\
			2048                              & 1.95e+01 & 1.02e+01   & 1.91     & 5.39e+00   & 3.62     & 4.47e+00   & 4.36     & 3.70e$-$01 & 52.70    \\
			\hline
		\end{tabular}
	}
\end{table}

Moreover, we observe that the computational time for both CPU and GPU parallel implementations scales as $\mathcal{O}(N^2)$. Specifically, doubling $N$ results in a fourfold increase in computational time. This trend is evident from the CPU computational times for different process counts shown in Table~\ref{tab:time-2d}. For GPU parallelization, the computational time similarly scales with the total number of grid points, as task distribution and scheduling overhead depend on the grid size, leading to the same fourfold increase when $N$ is doubled. Figure~\ref{fig:gpu2d_plot} further illustrates the relationship between computational time and the total number of grid points $N^2$ on a logarithmic scale, clearly demonstrating a linear dependence of computational time on $N^2$ for both CPU and GPU parallel strategies, regardless of the number of threads used.

\begin{figure}[H]
	\centering
		\includegraphics[width=0.7\textwidth]{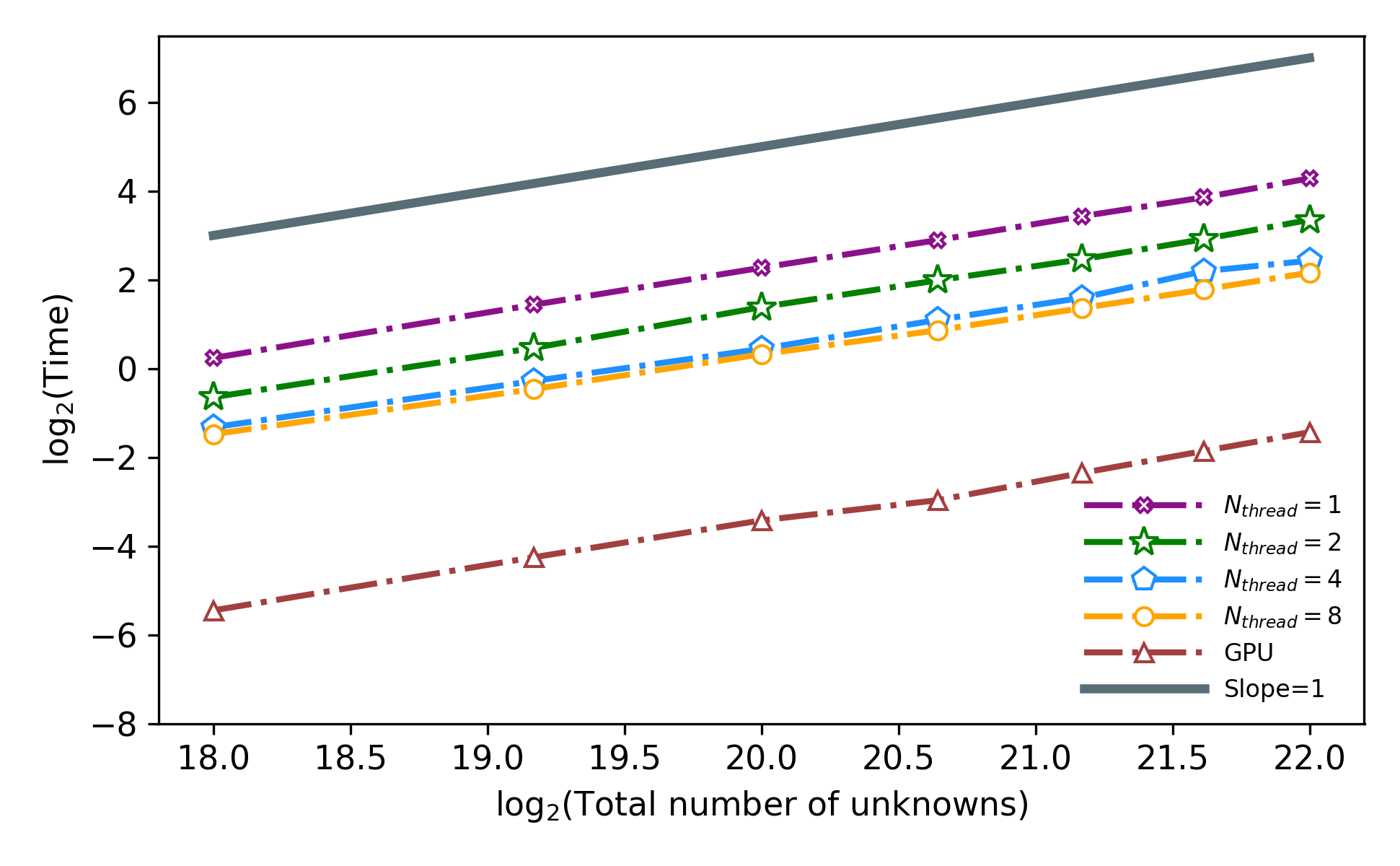}
	\caption{Total number of unknowns versus computational time for the 2D KGS equations.}\label{fig:gpu2d_plot}
\end{figure}

Figure~\ref{parallel-efficiency-energy-2d} illustrates the evolution of the discrete relative energy error for the DP-AVF2 schemes under both CPU and GPU parallel implementations. The left panel shows the results from the CPU parallel strategy based on the block-splitting approach. Different thread counts correspond to different update orders, which in turn result in variations of the DP-AVF2 scheme. Nevertheless, it can still be observed that all variations of the DP-AVF2 scheme preserve the discrete energy within machine precision. The right panel presents the results from the CPU and GPU parallel strategies using the checkerboard grid. For the same grid resolution $N$, the energy errors from CPU and GPU parallel strategies are nearly identical. Therefore, only the GPU parallel energy errors for different $N$ are plotted. It is evident that across different grid resolutions, the discrete energy is also preserved to machine precision.

\begin{figure}[H]
	\centering
		\includegraphics[width=0.48\textwidth]{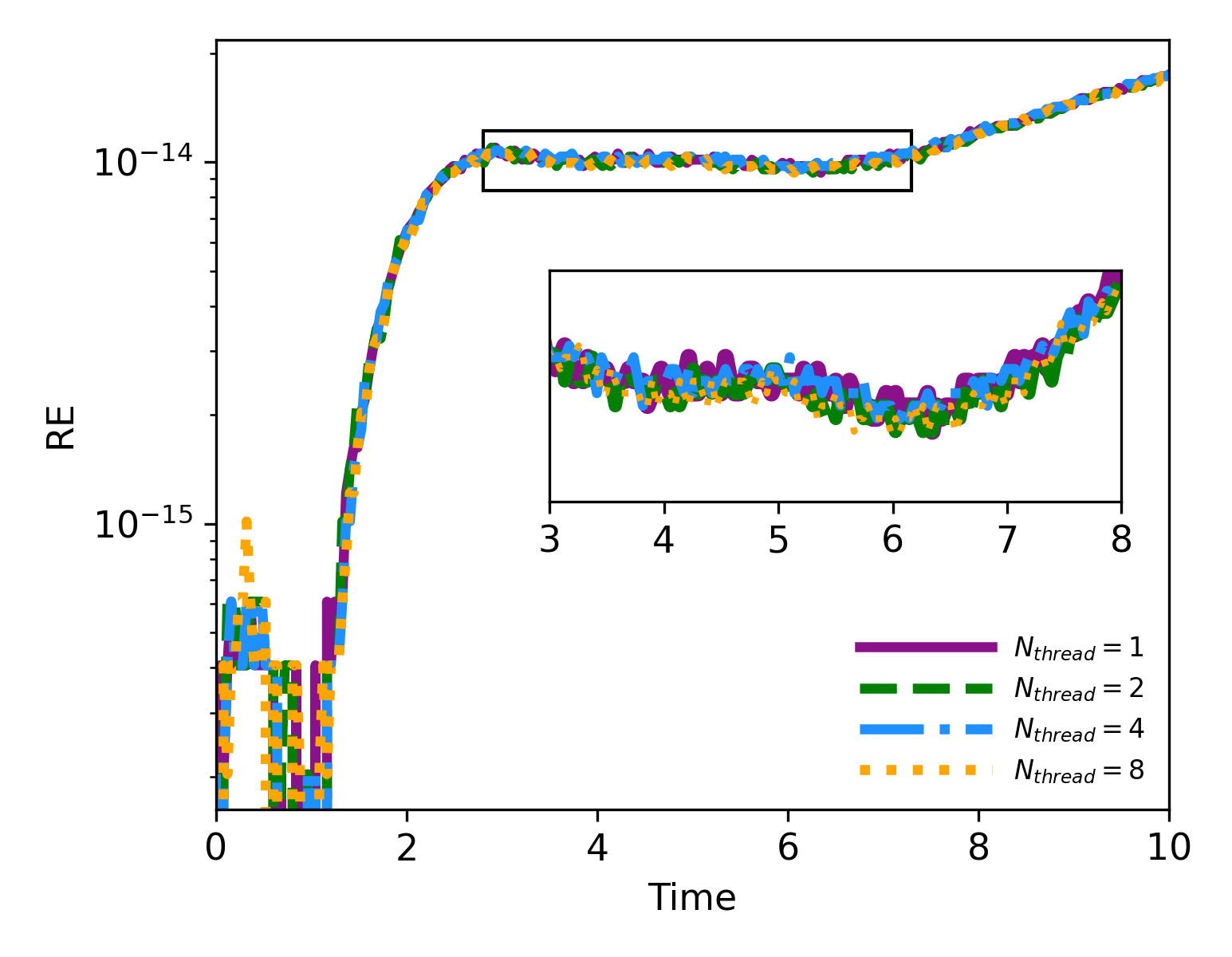}
		\includegraphics[width=0.48\textwidth]{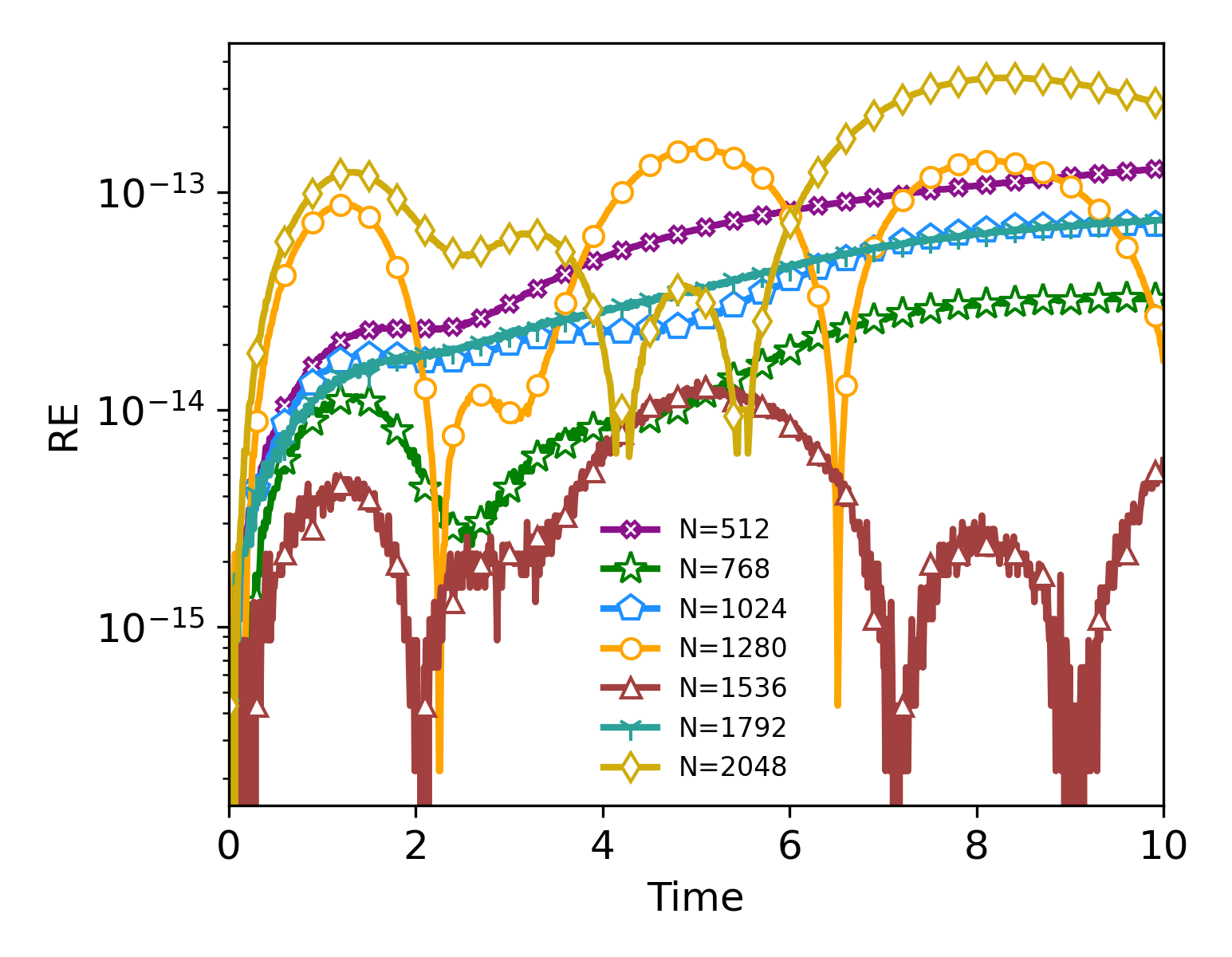}
	\caption{Time evolution of discrete relative energy errors for DP-AVF2 schemes. The left panel shows results from the block-splitting CPU strategy with different thread counts, and the right panel shows results from the checkerboard grid strategy for CPU and GPU implementations.}\label{parallel-efficiency-energy-2d}
\end{figure}

\subsection{Examples of 3D KGS equations}

Since the computational cost of the proposed DP-AVF2 scheme depends linearly on the total number of grid points, it is particularly well-suited for simulating high-dimensional problems. Consider the interactions of 3D circular-elliptical vector solitons governed by the KGS equations. The initial conditions are given as:
\[
\left\lbrace
\begin{aligned}
    & \psi_0(\bm{x}) = \sum_{j = 0}^1 \exp{\left( -\left(x + 2 (-1)^j\right)^2 - y^2 - z^2 \right)} \exp{\left(0.01j \left(x + y + z\right)\right)}, \\
    & u_0(\bm{x}) = \exp{\left(-x^2 - y^2 - (z-2)^2\right)} + \sum_{j = 0}^1 \exp{\left( -\left(x + (-1)^j \sqrt{3}\right)^2 - y^2 - \left(z + 1\right)^2 \right)}, \\
    & v_0(\bm{x}) = \exp{\left( -x^2 - y^2 - z^2 \right)}.
\end{aligned}
\right.
\]
The computational domain is $\Omega = (-10, 10)^3$, and the parameters in \eqref{kgs} are set as $\kappa_1 = -0.4$, $\kappa_2 = 0.1$, $\gamma = 0.2$, and $\mu = 0.1$.

Figure \ref{fig:kgs3dpsi} and \ref{fig:kgs3du} illustrate the interactions of circular-elliptical vector solitons in components $\psi$ and $u$, respectively. The solutions are computed using the DP-AVF2 method with $(h, \tau) = (\frac{5}{64}, \frac{1}{100})$. In component $\psi$, the initial isosurface $ |\psi| = 0.1 $ is composed of two distinct ellipsoids, which radiate and merge at $ t = 1.25 $. As the interaction progresses, a smaller ellipsoid emerges at the center of the domain and gradually expands. In component $u$, the initial configuration comprises three spherical structures. Over time, outward radiation leads to the formation of a new central isosurface $ u = 0.9 $, which subsequently expands and transitions into a single spherical structure as the interaction evolves.

\begin{figure}[H]
	\centering
		\includegraphics[width=0.23\textwidth]{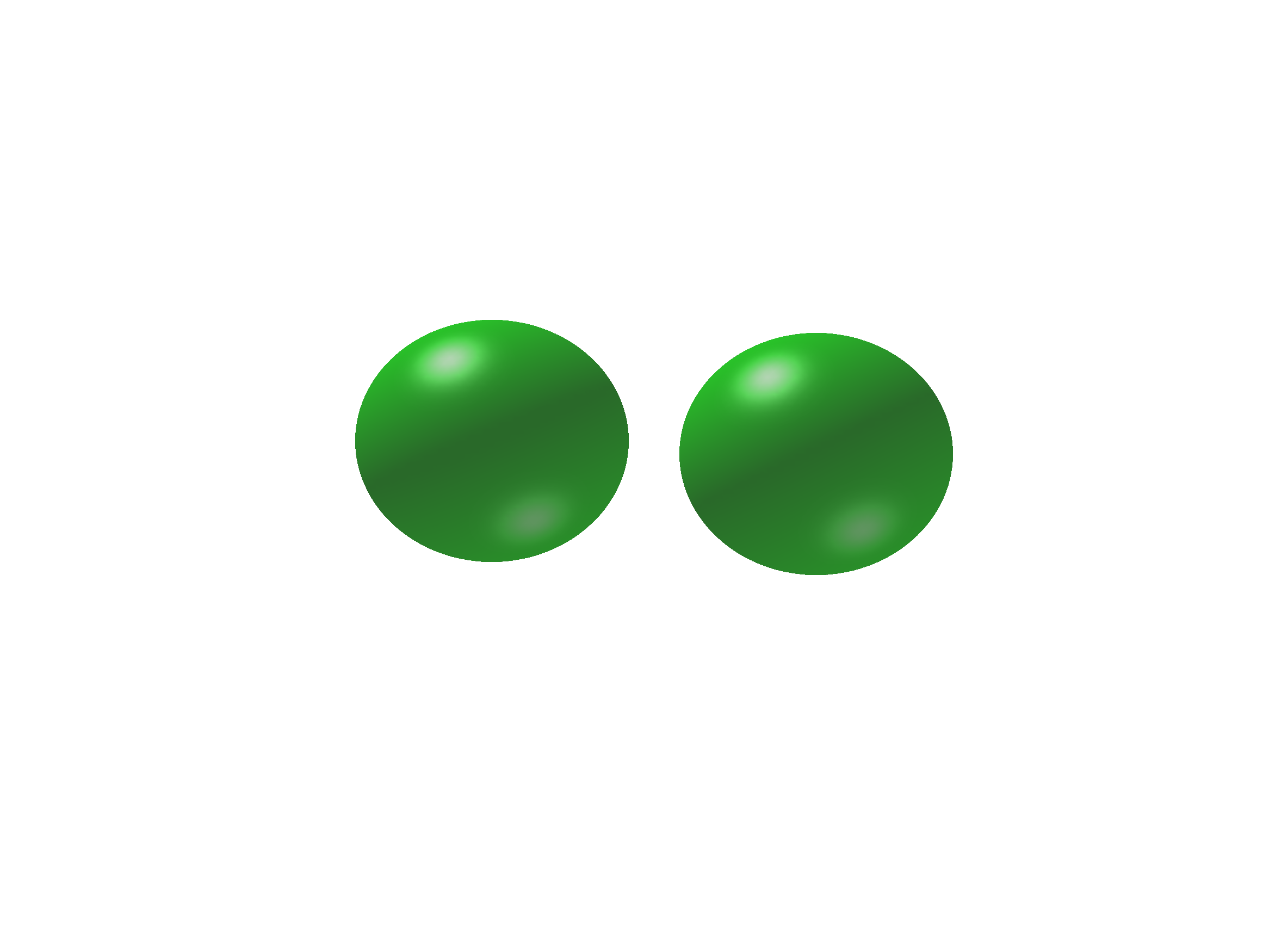}
		\includegraphics[width=0.23\textwidth]{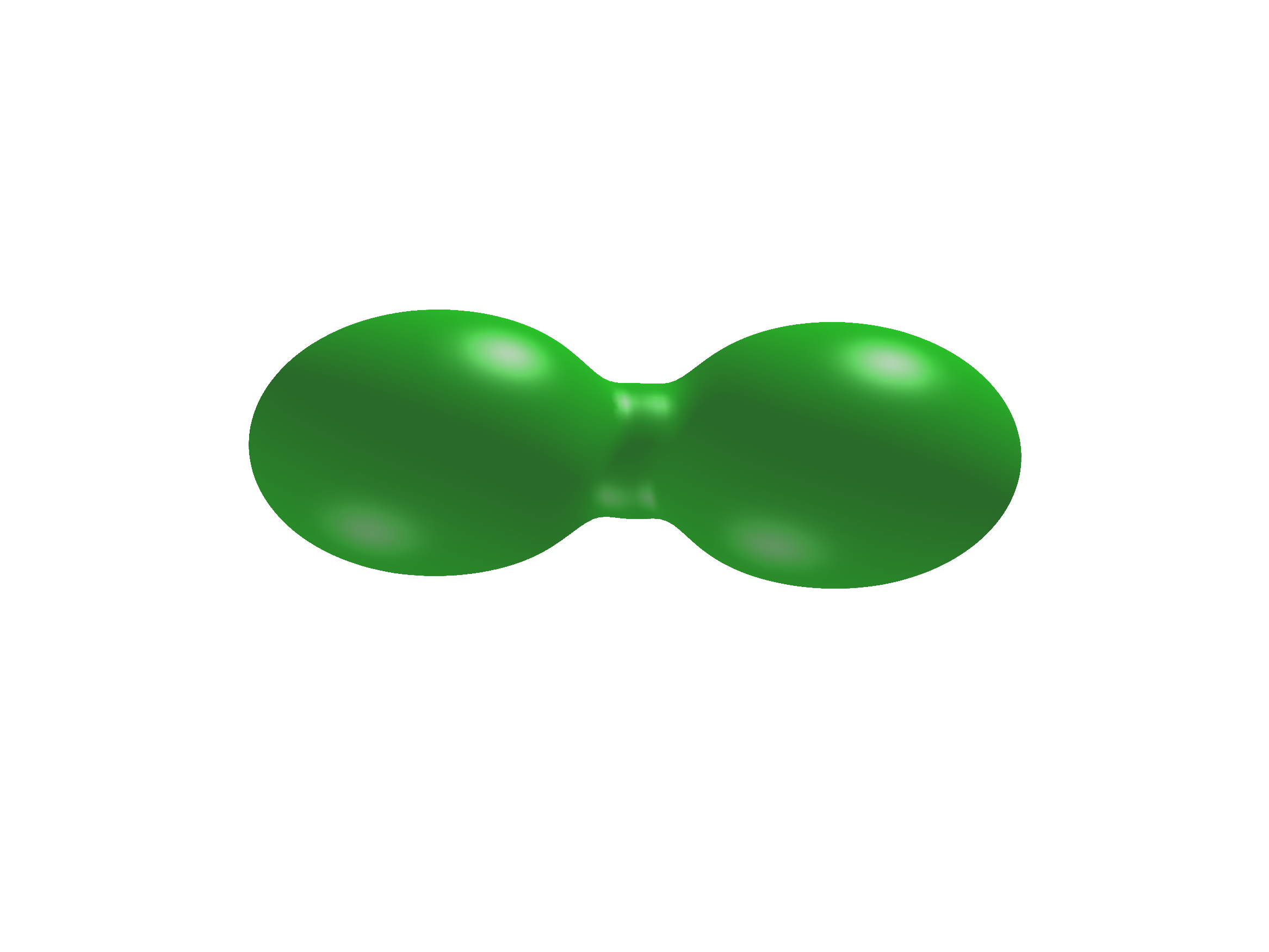}
		\includegraphics[width=0.23\textwidth]{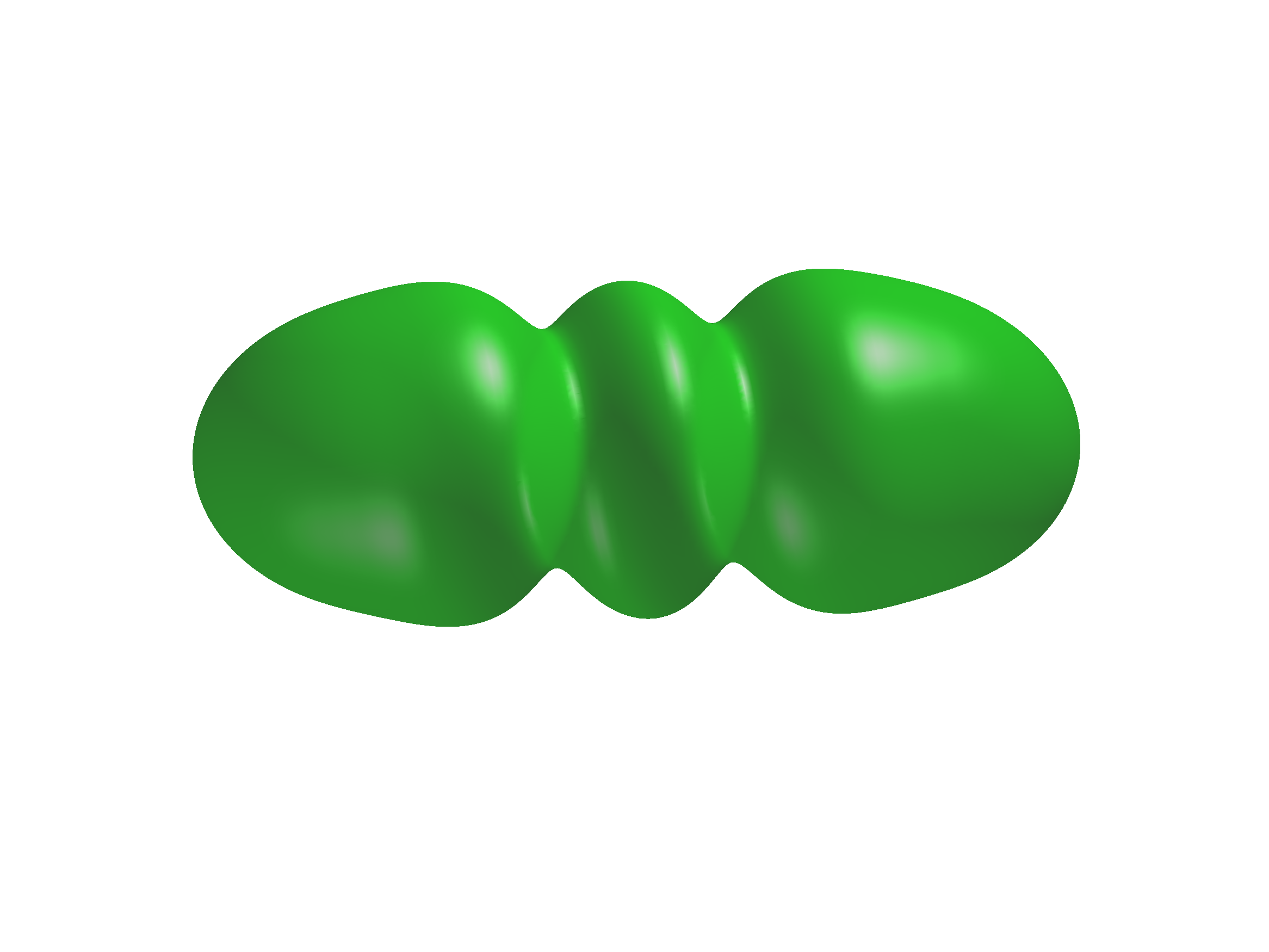}
		\includegraphics[width=0.23\textwidth]{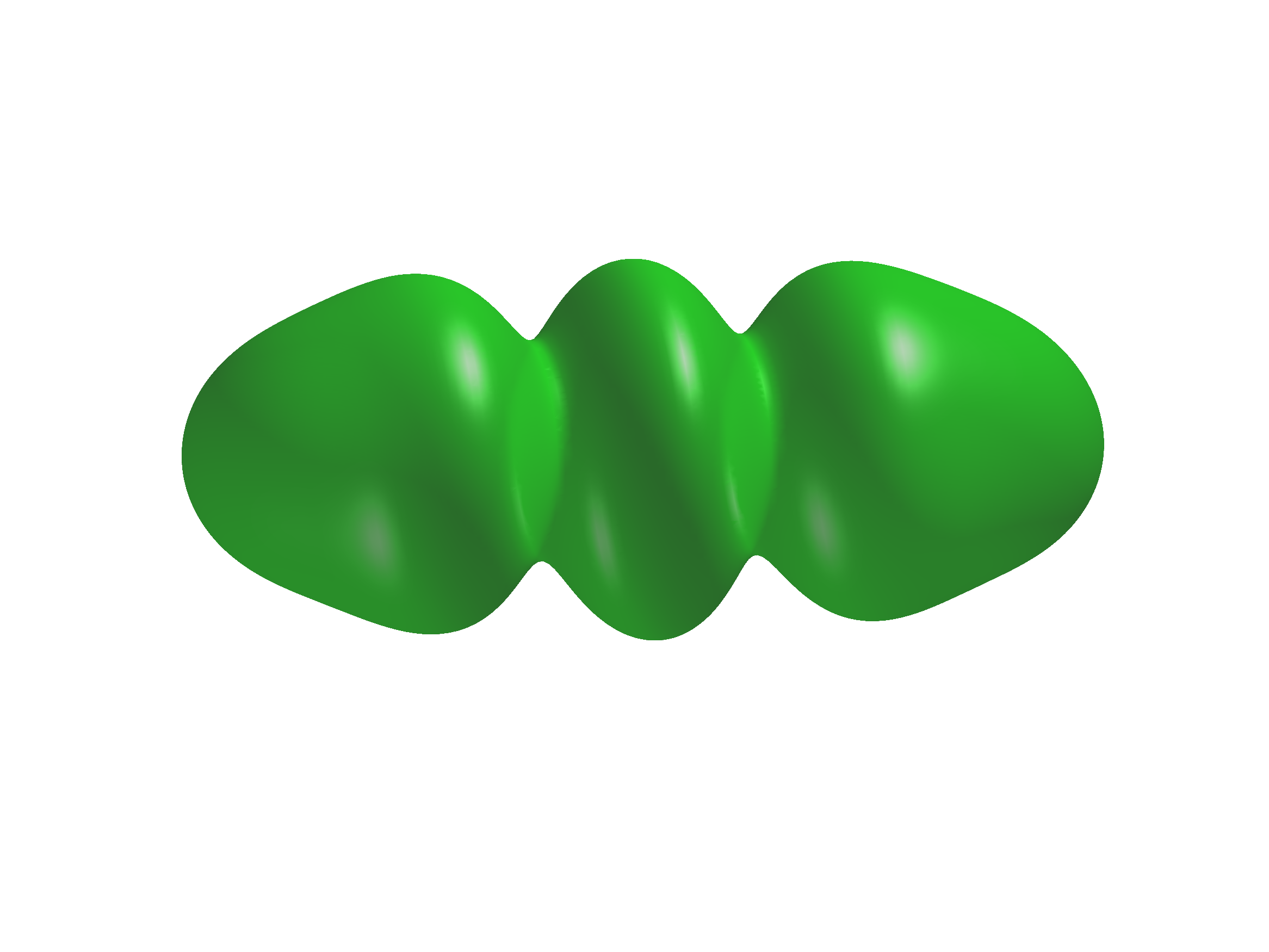}
	\caption{Isosurfaces $|\psi| = 0.1$ solved by the DP-AVF2 scheme under different times $t = 0, 1.25, 2.75, 3.25$.}\label{fig:kgs3dpsi}
\end{figure}

\begin{figure}[H]
	\centering
		\includegraphics[width=0.23\textwidth]{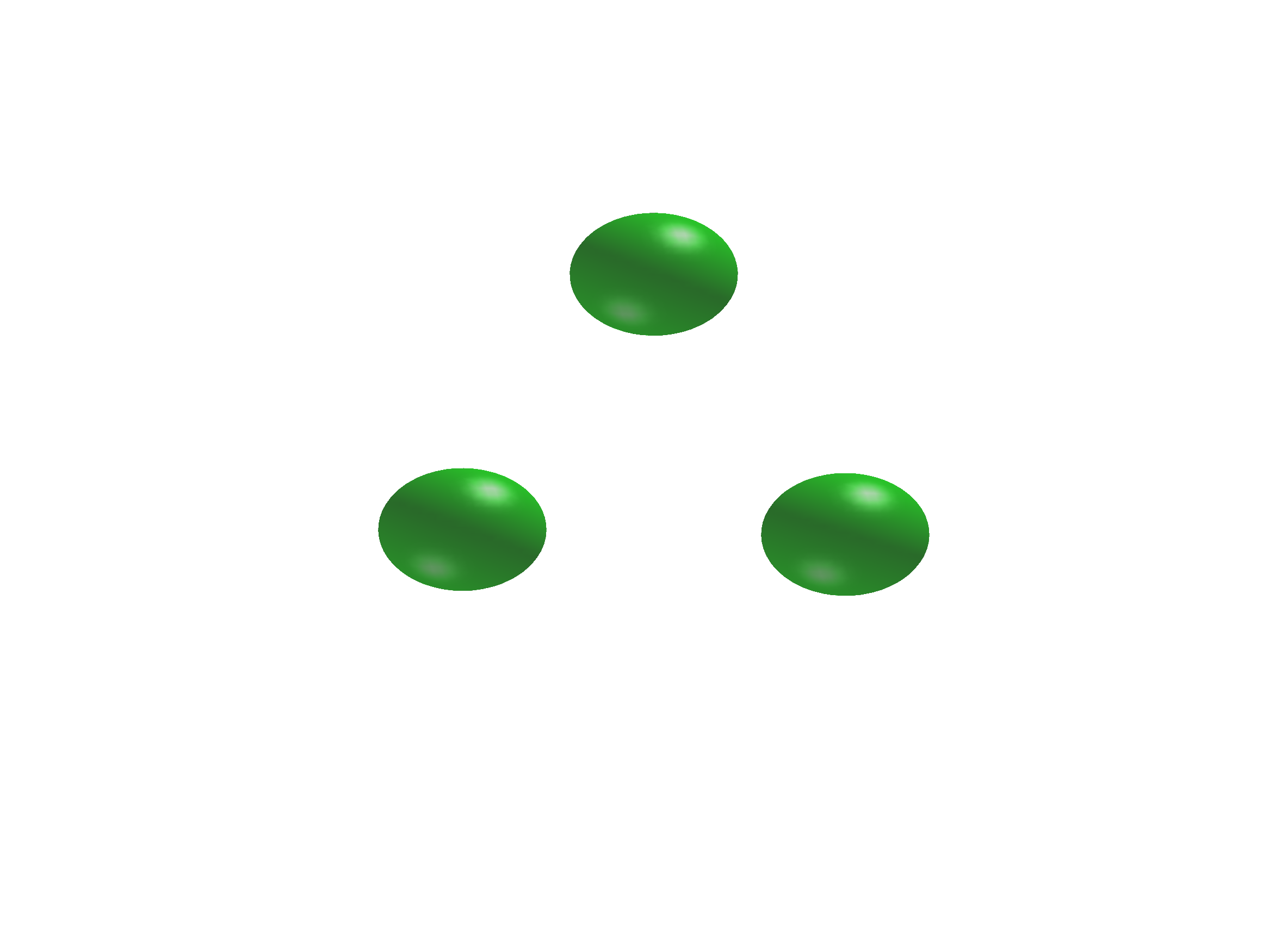}
		\includegraphics[width=0.23\textwidth]{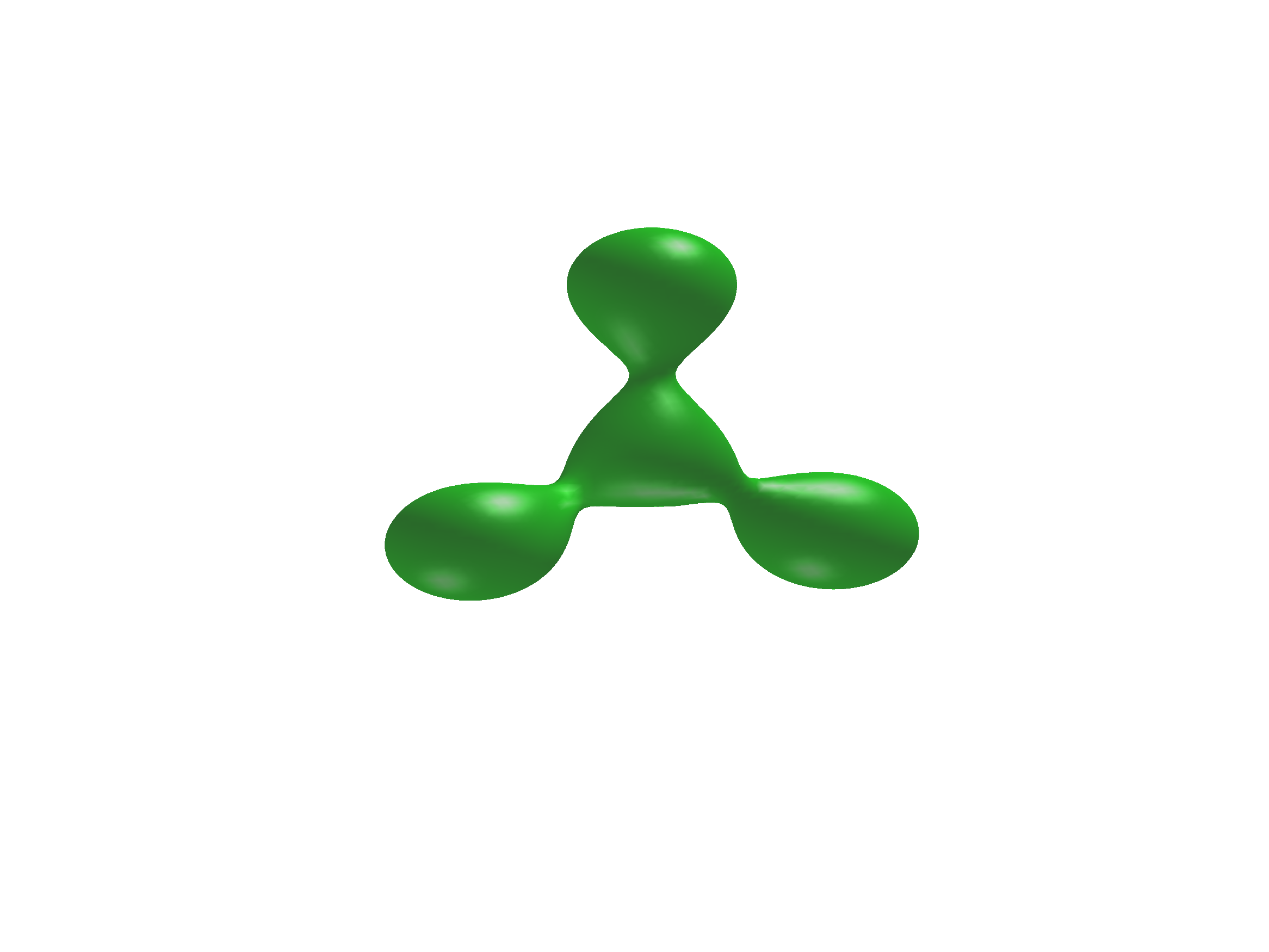}
		\includegraphics[width=0.23\textwidth]{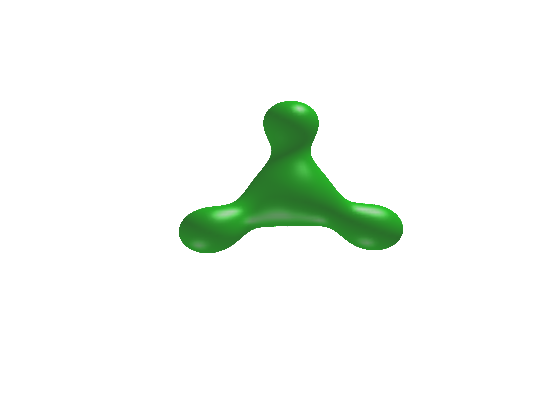}
		\includegraphics[width=0.23\textwidth]{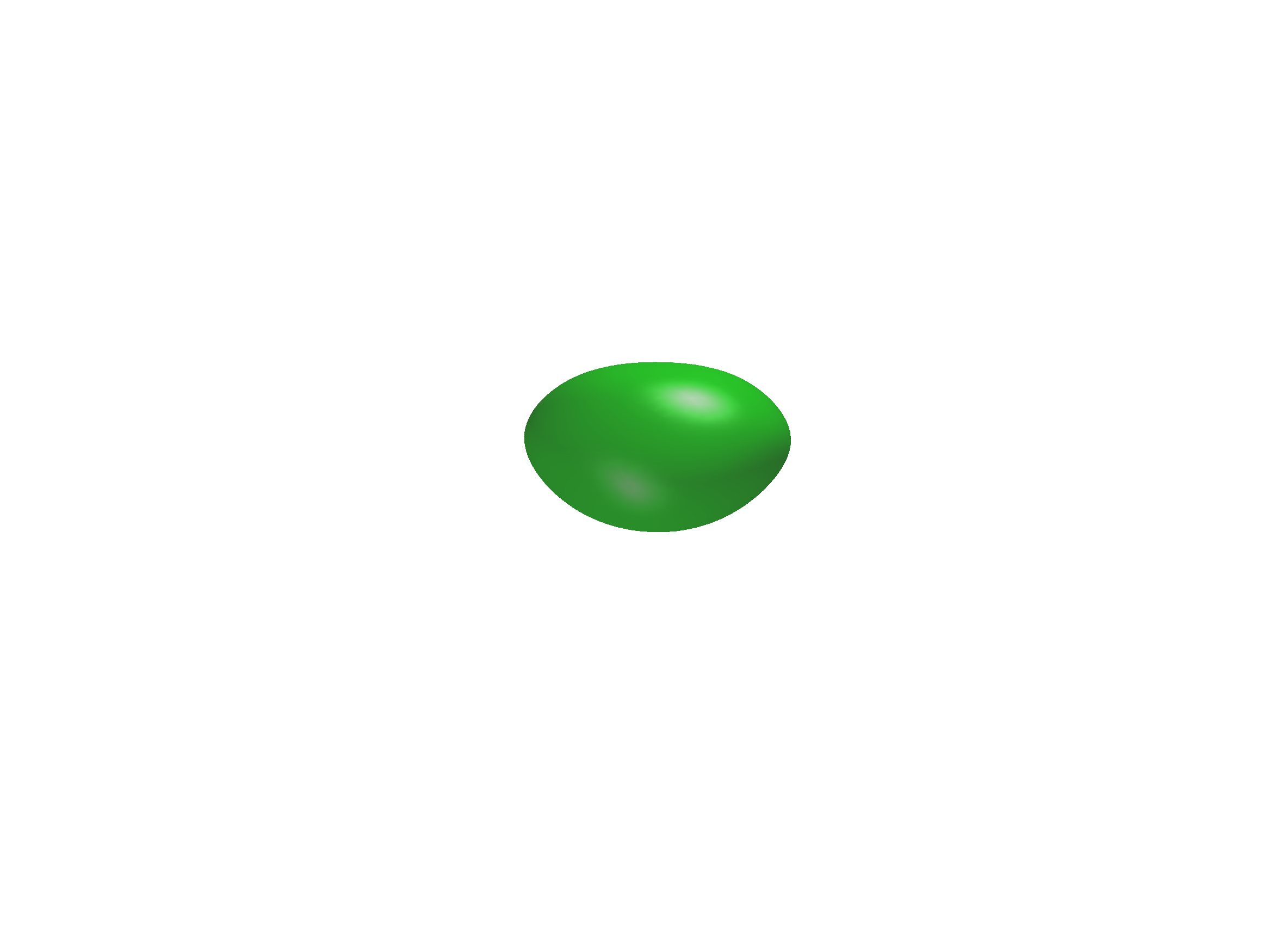}
	\caption{Isosurfaces $u = 0.9$ solved by the DP-AVF2 scheme under different times $t = 0, 0.75, 1, 2.5$.}\label{fig:kgs3du}
\end{figure}

Tables~\ref{tab:cpu3d_list}-\ref{tab:gpu3d_list} present the computational times and speed-ups for CPU and GPU implementations under the block-splitting and checkerboard grid parallel strategies, with $\tau = 0.01$ computed up to $t = 1$. The results reveal several key trends: i) CPU parallelization achieves higher speed-ups as the number of processes increases, but the gap between observed and theoretical speed-up widens as the process count approaches the total number of CPU cores. ii) The checkerboard grid strategy outperforms the block-splitting strategy in terms of CPU computational time, demonstrating greater efficiency. iii) For GPU parallelization with the checkerboard grid, the speed-up significantly exceeds that of CPU parallelization, with values approaching or exceeding 60 for $N = 128, 192, 256$, indicating much higher GPU utilization compared to the 2D case. However, as $N$ increases further, the limited total number of GPU processes restricts consistent speed-up improvements. iv) Computational times for both CPU and GPU parallel strategies scale linearly with the total grid points $N^3$, where doubling $N$ results in an eightfold increase in computational time, a trend further validated by Figure~\ref{fig:gpu3d_plot}.

\begin{table}[H]
	\caption{CPU times and speed-ups of the DP-AVF2 schemes with varying numbers of threads and spatial discretizations for the 3D KGS equations. The block-splitting parallel strategy in Figure \ref{cpu-3d} is applied.}\label{tab:cpu3d_list}
	\centering
	\resizebox{\textwidth}{!}{
		{\scriptsize\begin{tabular}[c]{c|c|cc|cc|cc}
			\toprule
			\multicolumn{1}{c|}{}    &
			\multicolumn{7}{c}{\bf CPU time (s)}                                                                      \\
			\hline
			\multicolumn{1}{c|}{$N$} & 1-thread & 2-thread & speed-up & 4-thread & speed-up & 8-thread & speed-up \\
			\hline
			128                      & 1.25e+01 & 6.66e+00 & 1.88     & 3.44e+00 & 3.63     & 2.87e+00 & 4.36     \\
			256                      & 1.01e+02 & 5.35e+01 & 1.89     & 2.79e+01 & 3.62     & 2.00e+01 & 5.05     \\
			512                      & 8.04e+02 & 4.33e+02 & 1.86     & 2.26e+02 & 3.56     & 1.58e+02 & 5.09     \\
			\bottomrule
		\end{tabular}}
	}
\end{table}

\begin{table}[H]
	\caption{Computational times and speed-ups of the DP-AVF2 schemes on CPU and GPU with varying numbers of threads and spatial discretizations for the 3D KGS equations.}\label{tab:gpu3d_list}
	\centering
	\resizebox{\textwidth}{!}{
		\begin{tabular}[c]{c|c|cc|cc|cc|cc}
			\hline
			\multicolumn{1}{c|}{}             &
			\multicolumn{7}{c|}{\bf CPU time (s)} &
			\multicolumn{2}{c}{\bf GPU time (s)}
			\\
			\hline
			\multicolumn{1}{c|}{$N$}          & 1-thread   & 2-thread   & speed-up & 4-thread   & speed-up & 8-thread   & speed-up & time       & speed-up \\
			\hline
			128                               & 1.19e$+$01 & 6.00e$+$00 & 1.98     & 3.12e$+$00 & 3.81     & 2.76e$+$00 & 4.31     & 1.72e$-$01 & 69.19    \\
			192                               & 3.98e$+$01 & 2.04e$+$01 & 1.95     & 1.04e$+$01 & 3.83     & 7.96e$+$00 & 5.00     & 6.25e$-$01 & 63.68    \\
			256                               & 9.46e$+$01 & 4.81e$+$01 & 1.97     & 2.45e$+$01 & 3.86     & 1.79e$+$01 & 5.28     & 1.61e$+$00 & 58.76    \\
			320                               & 1.87e$+$02 & 9.59e$+$01 & 1.95     & 4.99e$+$01 & 3.75     & 3.48e$+$01 & 5.37     & 3.30e$+$00 & 56.67    \\
			384                               & 3.23e$+$02 & 1.61e$+$02 & 2.01     & 8.80e$+$01 & 3.67     & 5.86e$+$01 & 5.51     & 6.03e$+$00 & 53.57    \\
			448                               & 5.15e$+$02 & 2.69e$+$02 & 1.91     & 1.34e$+$02 & 3.84     & 9.19e$+01$ & 5.60     & 9.96e$+$00 & 51.71    \\
			512                               & 7.71e$+$02 & 3.90e$+$02 & 1.98     & 1.98e$+$02 & 3.89     & 1.33e$+$02 & 5.80     & 1.56e$+$01 & 49.42    \\
			\hline
		\end{tabular}
	}
\end{table}

\begin{figure}[H]
	\centering
		\includegraphics[width=0.7\textwidth]{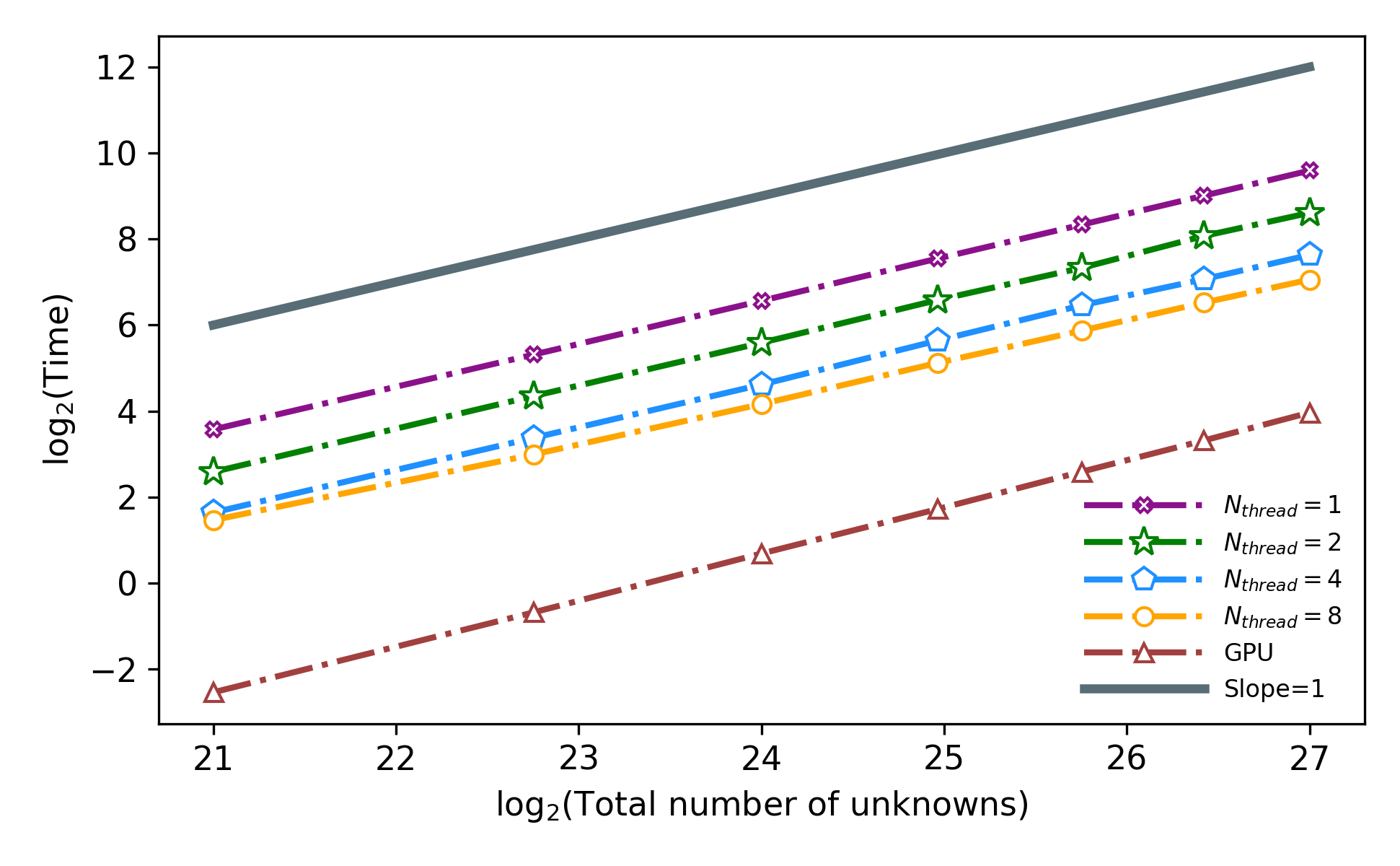}
	\caption{Total number of unknowns versus computational time for the 3D KGS equations.}\label{fig:gpu3d_plot}
\end{figure}

\begin{figure}[H]
	\centering
		\includegraphics[width=0.48\textwidth]{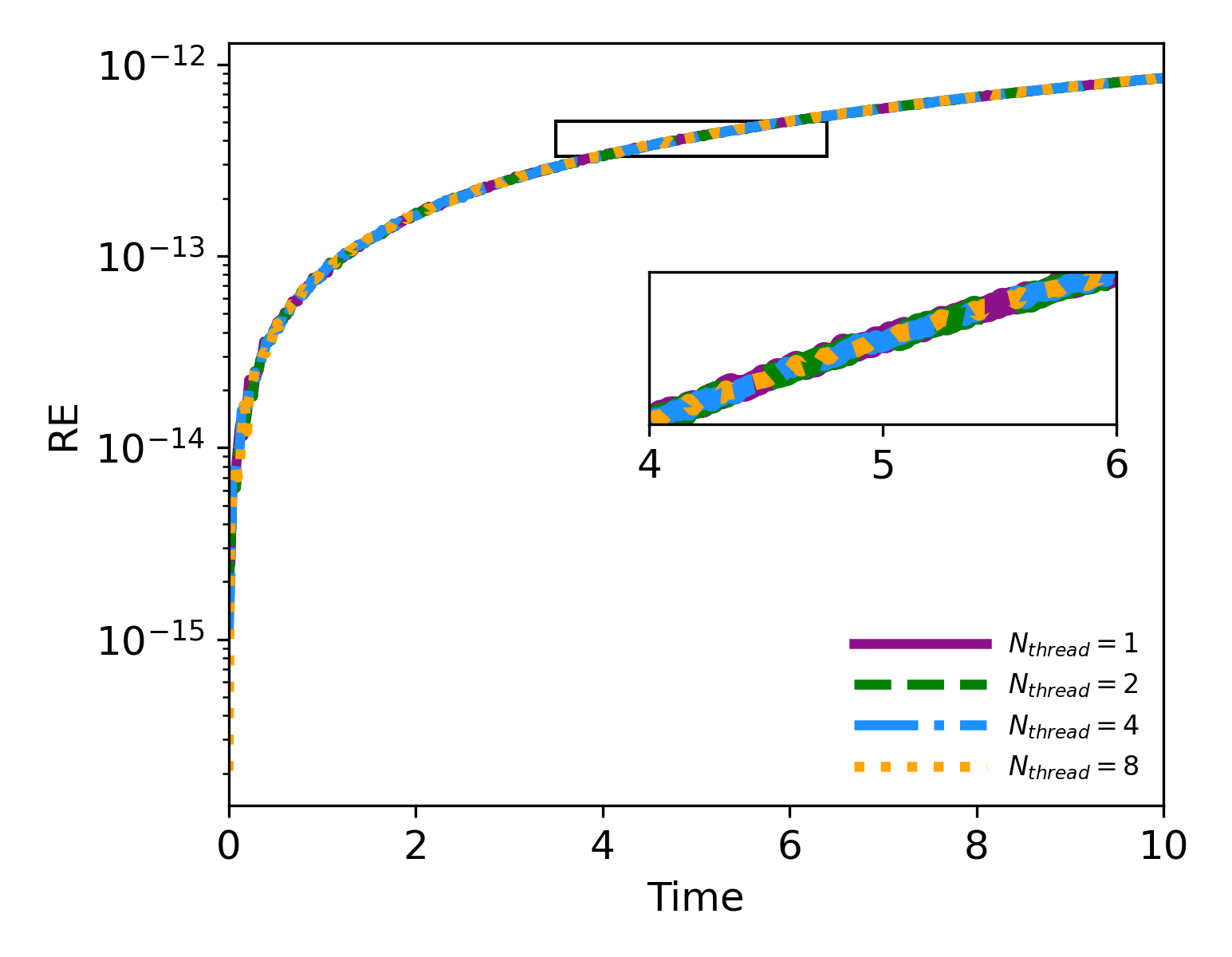}
		\includegraphics[width=0.48\textwidth]{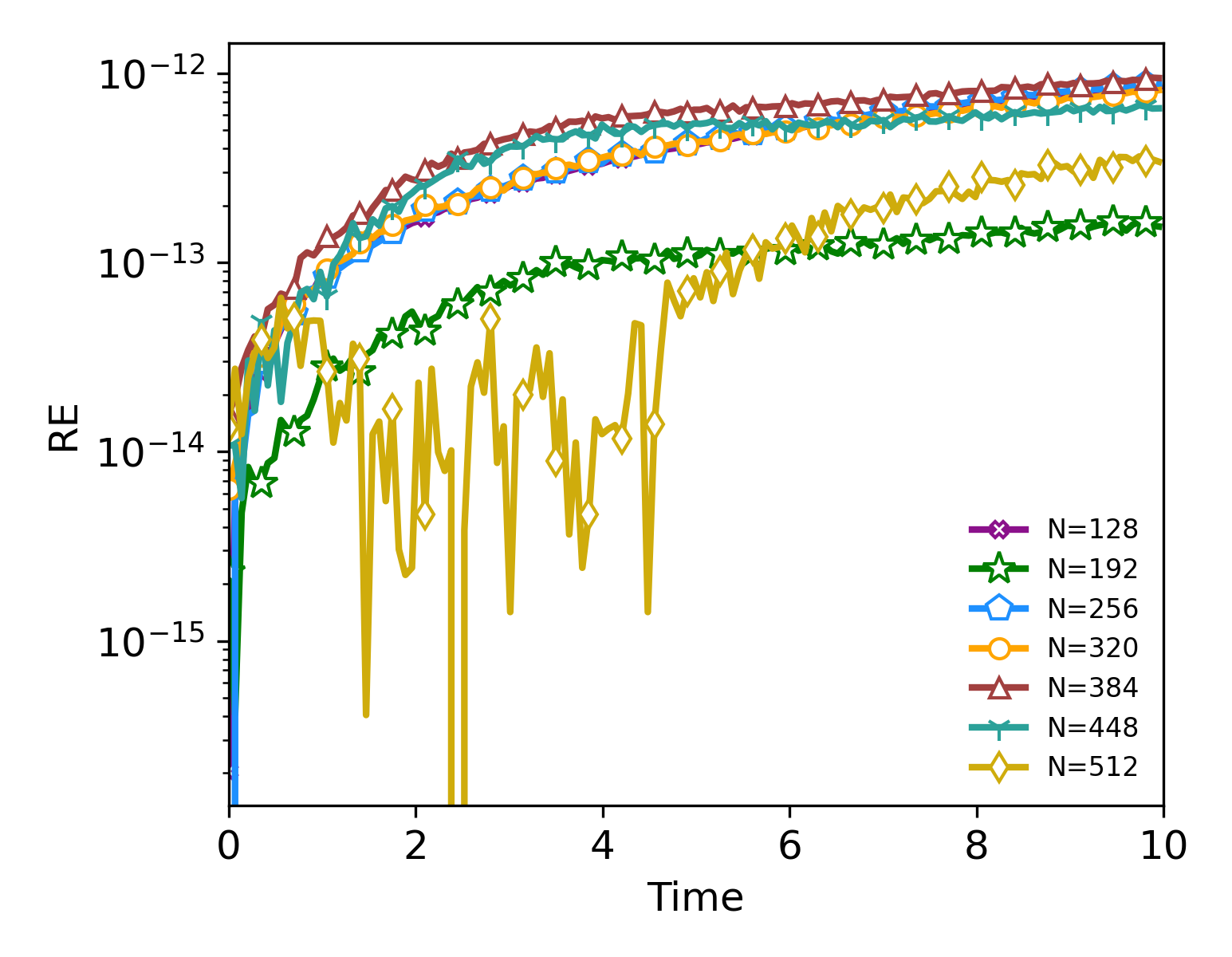}
	\caption{Time evolution of discrete relative energy errors for DP-AVF2 schemes. The left panel shows results from the block-splitting CPU strategy with different thread counts, and the right panel shows results from the checkerboard grid strategy for CPU and GPU implementations. }\label{cpu_energy3d}
\end{figure}

Figure~\ref{cpu_energy3d} shows the evolution of the discrete relative energy error in the 3D case, computed using the DP-AVF2 scheme with a time step of $\tau = 0.02$.  The left panel presents results for the block-splitting strategy with different process counts. Although these correspond to variations of the DP-AVF2 scheme, the discrete energy is still preserved to machine precision. The right panel illustrates GPU parallel results for the checkerboard grid with different values of $N$, as representative examples. The CPU parallel results for the block-splitting strategy are nearly identical. It can also be observed that, even with the random update order inherent in GPU parallelization, the energy in all cases is preserved within round-off errors.

\section{Conclusion}\label{sec:6}

For multivariate coupled systems, we combine the variable-based and grid-point-based partitioned AVF methods to introduce a novel dual-partition AVF approach. This method not only preserves the system's original energy but also provides flexible partition strategies for both variables and grid points, allowing the construction of highly efficient numerical schemes that can be implemented pointwise. We further apply this time integration method to the coupled KGS equations after the central finite difference semi-discretization, resulting in schemes that achieve pointwise and even explicit solutions. Moreover, we propose two parallelizable partition strategies: the block-splitting strategy and the checkerboard grid. Both strategies enable parallel computation on CPUs, with the latter particularly well-suited for parallel execution on GPUs, significantly enhancing the computational efficiency of the dual-partition AVF method for high-dimensional KGS equations. Numerical experiments validate the effectiveness, conservation properties, and computational efficiency of the proposed schemes.

\section*{Acknowledements}
This work is supported by the National Natural Science Foundation of China (12171245, 11971242, 61872422).


\end{document}